\documentclass[a4paper,10pt]{smfart}

\RequirePackage{smfenum}
\usepackage[latin1]{inputenc}
\usepackage[frenchb]{babel}
\RequirePackage[T1]{fontenc}
\RequirePackage{amsfonts,latexsym,amssymb}
\RequirePackage[frenchb]{babel}
\usepackage[all]{xy}
\addto\extrasfrenchb{\mathbfl@nonfrenchitemize
\mathbfl@nonfrenchspacing}
\usepackage{geometry}
\usepackage{mathrsfs}
\usepackage{yhmath}

\newtheoremstyle{theorem}{11pt}{11pt}{\slshape}{}{\mathbfseries}{.}{.5em}{}
\newtheoremstyle{note}{11pt}{11pt}{}{}{\mathbfseries}{.}{.5em}{}

\theoremstyle{plain}
  \newtheorem{theoreme}{Th\'{e}or\`{e}me}[section]
  \newtheorem{proposition}[theoreme]{Proposition}
  
  \newtheorem{lemme}[theoreme]{Lemme}
  \newtheorem{corollaire}[theoreme]{Corollaire}
  
  \newtheorem{question}[theoreme]{Question}

\theoremstyle{definition}
  \newtheorem{definition}[theoreme]{D\'{e}finition}
   
  \newtheorem*{notations}{Notations}

\theoremstyle{remark}
  \newtheorem{example}[theoreme]{Exemple}
  \newtheorem{examples}[theoreme]{Exemples}
  \newtheorem{remarque}[theoreme]{Remarque}
  \newtheorem{remarks}[theoreme]{Remarques}

\usepackage{url}
\usepackage{todonotes}

\let\mathcal\mathcal

\let\mathbf\mathbf

\def\O{{\mathcal O}}

\def\zp{{\Z_p}}

\def\qp{{\Q_p}}

\def\p1{{\mathbf P}^1}
\def\qp{\mathbf{Q}_p}
\def\zp{\mathbf{Z}_p}
\def\z{\mathbf{Z}}

\def\cp{\mathbf{C}_p}
\def\q{\mathbf{Q}}

\def\con{\mathcal{C}^0}

\def\Spa{\mathrm{Spa}}

\def\epsilon{\varepsilon}

\begin{document}
\title[Banach-Colmez et faisceaux cohÈrents sur la courbe de Fargues-Fontaine]{Espaces de Banach-Colmez et faisceaux cohÈrents sur la courbe de Fargues-Fontaine}
\author{Arthur-C\'esar Le Bras}
\address{Ecole Normale SupÈrieure, 45 rue d'Ulm, 75005 Paris, France et Institut de Math\'ematiques de Jussieu, 4 place Jussieu, 75005 Paris, France}
\email{arthur-cesar.le-bras@imj-prg.fr \\ lebras@dma.ens.fr}

\begin{abstract} Nous donnons une nouvelle dÈfinition, plus simple mais Èquivalente, de la catÈgorie abÈlienne des espaces de Banach-Colmez (introduite par Colmez dans \cite{bc}), et nous expliquons la relation prÈcise de cette catÈgorie avec celle des faisceaux cohÈrents sur la courbe de Fargues-Fontaine. On passe d'une catÈgorie ‡ l'autre en modifiant la $t$-structure sur la catÈgorie dÈrivÈe. Chemin faisant, nous obtenons une description de la cohomologie pro-Ètale du disque ouvert et de l'espace affine d'intÈrÍt indÈpendant.
 \end{abstract}

\begin{altabstract} We give a new definition, simpler but equivalent, of the abelian category of Banach-Colmez spaces introduced by Colmez in \cite{bc}, and we explain the precise relationship with the category of coherent sheaves on the Fargues-Fontaine curve. One goes from one category to the other by changing the $t$-structure on the derived category. Along the way, we obtain a description of the pro-Ètale cohomology of the open disk and the affine space, of independent interest. 
\end{altabstract}
\setcounter{tocdepth}{3}

\maketitle

\stepcounter{tocdepth}
{\Small
\tableofcontents
}

\section{Introduction}

Les travaux rÈcents de Fargues (\cite{farguesconj}, \cite{farguescdc}) et Scholze (\cite{berkeley}) sur la gÈomÈtrisation de la correspondance de Langlands locale mËnent naturellement ‡ la rencontre des espaces de Banach-Colmez et ‡ l'Ètude de certaines questions les concernant. Ces objets avaient ÈtÈ introduits par Colmez, suite ‡ certaines constructions de Fontaine, sous le nom d'\textit{Espaces Vectoriels de dimension finie} (noter les lettres capitales) il y a quinze ans (\cite{bc}), avec pour objectif d'obtenir une nouvelle preuve de la conjecture \og faiblement admissible implique admissible \fg{} en thÈorie de Hodge $p$-adique. Rappelons briËvement de quoi il s'agissait.

Soit $C$ le complÈtÈ d'une clÙture algÈbrique de $\qp$. Colmez dÈfinit les espaces de Banach-Colmez comme foncteurs sur la catÈgorie des $C$-algËbres \textit{sympathiques}\footnote{Les algËbres sympathiques sont un certain type d'algËbres perfectoÔdes. Pour la dÈfinition prÈcise, voir la dÈfinition \ref{defsympa}.} ‡ valeurs dans les $\qp$-espaces de Banach. Deux exemples simples de tels foncteurs sont les suivants : d'une part, si $V$ est un $\qp$-espace vectoriel de dimension finie, le foncteur qui ‡ une algËbre sympathique $\Lambda$ associe $V$, notÈ encore $V$ et que l'on appellera un \textit{$\qp$-Espace Vectoriel de dimension finie} ; d'autre part, si $W$ est un $C$-espace vectoriel de dimension finie, le foncteur qui ‡ une algËbre sympathique $\Lambda$ associe $\Lambda \otimes_C W$, notÈ encore $W$ et que l'on appellera un \textit{$C$-Espace Vectoriel de dimension finie}. Un espace de Banach-Colmez est alors un foncteur ne diffÈrant d'un $C$-Espace Vectoriel de dimension finie que par des $\qp$-Espaces Vectoriels de dimension finie : par dÈfinition, tout espace de Banach-Colmez admet une prÈsentation comme quotient par un $\qp$-Espace Vectoriel de dimension finie $V'$ d'une extension d'un $C$-Espace Vectoriel de dimension finie $W$ par un $\qp$-Espace Vectoriel de dimension finie $V$. Cela permet d'attacher ‡ une telle prÈsentation d'un espace de Banach-Colmez deux entiers : sa \textit{dimension} $\dim_C W$ et sa \textit{hauteur} $\dim_{\qp} V - \dim_{\qp} V'$. Cette dÈfinition peut sembler un peu Ètrange, mais Colmez montre que la catÈgorie des espaces de Banach-Colmez est une catÈgorie abÈlienne, que le foncteur d'Èvaluation sur $C$ est exact et conservatif et que les fonctions dimension et hauteur ne dÈpendent pas de la prÈsentation et dÈfinissent deux fonctions additives sur cette catÈgorie. De faÁon remarquable et suprenante, car leur dÈfinition n'en fait pas mention, l'Ètude des espaces de Banach-Colmez fait naturellement apparaÓtre certains anneaux de Fontaine utilisÈs en thÈorie de Hodge $p$-adique.

Les propriÈtÈs de la catÈgorie des espaces de Banach-Colmez ont ensuite ÈtÈ explorÈes par Fontaine et Pl˚t (\cite{FontKa}, \cite{plut}). Elles font fortement penser aux propriÈtÈs bien connues de la catÈgorie des faisceaux cohÈrents sur une courbe. Ceci a amenÈ Fargues et Fontaine, en conjonction avec d'autres indices, ‡ deviner l'existence d'une courbe\footnote{A ce sujet, on pourra consulter avec profit la prÈface de \cite{FF}.}, qui porte aujourd'hui leur nom, dont l'Ètude devrait reflÈter les rÈsultats fondamentaux de la thÈorie de Hodge $p$-adique. Le dÈveloppement de la thÈorie a pris quelques annÈes et l'on dispose maintenant de plusieurs points de vue sur la courbe de Fargues-Fontaine : le point de vue \og algÈbrique \fg{}, qui Ètait la perspective initiale de Fargues et Fontaine ; le point de vue \og analytique \fg{}, qui est souvent le plus pratique ; et enfin le point de vue de la thÈorie des diamants de Scholze, qui est le plus abstrait mais a l'intÈrÍt de faire comprendre pourquoi la courbe de Fargues-Fontaine est un objet naturel et fondamental pour la thÈorie $p$-adique. La dÈfinition la plus rapide de la courbe de Fargues-Fontaine $X$ (pour le corps $\qp$) est sa dÈfinition adique. Notons $C^{\flat}$ le basculÈ du corps perfectoÔde $C$ et soit $p^{\flat} \in C^{\flat}$ tel que $(p^{\flat})^{\sharp} =p$. Soit 
\[ Y = \mathrm{Spa}(W(\O_{C^{\flat}}),W(\O_{C^{\flat}})) \backslash V(p [p^{\flat}]). \]
L'espace adique $Y$ est muni d'un opÈrateur $\varphi$ agissant proprement discontin˚ment ; on dÈfinit 
\[ X=Y/\varphi^{\z}. \]
Les sections globales des faisceaux cohÈrents sur la courbe de Fargues-Fontaine sont les $C$-points d'espaces de Banach-Colmez. Il Ètait donc naturel ‡ ce stade de se demander quelle relation prÈcise entretiennent la catÈgorie des faisceaux cohÈrents sur la courbe de Fargues-Fontaine et celle des espaces de Banach-Colmez. La solution n'est pas immÈdiate, car on se convainc facilement que ces deux catÈgories abÈliennes ne sont pas Èquivalentes. L'objectif de cet article est de rÈpondre ‡ cette question et d'explorer quelques problËmes liÈs.

\subsection{RÈsultats principaux et plan de l'article}
Nous commenÁons par redÈfinir la catÈgorie des espaces de Banach-Colmez dans la section \ref{sec1}. La dÈfinition que nous en donnons est plus Èconomique et plus naturelle que celle de Colmez, mais moins explicite. Notons $\mathrm{Perf}_C$ la catÈgorie des espaces perfectoÔdes sur $C$. Nous munissons $\mathrm{Perf}_C$ de la topologie pro-Ètale \footnote{On pourrait aussi bien utiliser la $v$-topologie de Scholze.}. Deux exemples simples de faisceaux sur ce site ‡ valeurs dans la catÈgorie des $\qp$-espaces vectoriels sont le faisceau $\mathbf{G}_a$, qui ‡ $S \in \mathrm{Perf}_C$ associe $\O_S(S)$, et le faisceau constant $\underline{\qp}$ associÈ ‡ $\qp$. 

\begin{definition} \label{maindefintro}
La catÈgorie $\mathcal{BC}$ des \textit{espaces de Banach-Colmez} est la plus petite sous-catÈgorie abÈlienne stable par extensions contenant les faisceaux $\underline{\qp}$ et $\mathbf{G}_a$, de la catÈgorie des faisceaux de $\qp$-espaces vectoriels sur $\mathrm{Perf}_{C,\mathrm{pro\acute{e}t}}$. 
\end{definition}

Cette dÈfinition donne en fait naissance ‡ une catÈgorie Èquivalente ‡ la catÈgorie originale de Colmez\footnote{Que les espaces de Banach-Colmez puissent Ítre dÈfinis comme faisceaux pro-Ètales avait d'ailleurs ÈtÈ pressenti par Colmez (\cite[p. 5]{bc}).}. Des exemples naturels d'espaces de Banach-Colmez, hors des exemples Èvidents, sont fournis par la thÈorie des groupes $p$-divisibles. Cela explique\footnote{Selon le point de vue...} l'apparition d'anneaux de Fontaine dans l'Ètude de $\mathcal{BC}$. Nous dÈcrivons briËvement comment ‡ la fin de la section \ref{sec1}. 
\\

Afin de relier la catÈgorie $\mathcal{BC}$ ‡ la courbe de Fargues-Fontaine, il nous faut considÈrer une $t$-structure diffÈrente de la $t$-structure standard sur la catÈgorie dÈrivÈe bornÈe $D(X)=D^b(\mathrm{Coh}_X)$ de la catÈgorie abÈlienne $\mathrm{Coh}_X$ des faisceaux cohÈrents sur $X$. La catÈgorie $\mathrm{Coh}_X$ est trËs bien comprise, gr‚ce ‡ \cite{FF} : on peut dÈfinir des fonctions rang et degrÈ, et la filtration de Harder-Narasimhan d'un fibrÈ vectoriel sur $X$ ; on dispose d'un thÈorËme de classification des fibrÈs sur $X$, qui rappelle le thÈorËme de Grothendieck pour les fibrÈs sur $\mathbf{P}^1$. Ces rÈsultats sont rappelÈs dans la section \ref{faisceauxsurlacourbe}. 

C'est prÈcisÈment l'existence d'un formalisme de Harder-Narasimhan sur $X$ qui va nous permettre de fabriquer une nouvelle catÈgorie abÈlienne reliÈe ‡ $\mathcal{BC}$. ConsidÈrons la sous-catÈgorie pleine suivante de $D(X)$ :
\[ \mathrm{Coh}_X^- = \{ \mathcal{F} \in D(X), H^i(\mathcal{F}) = 0 ~ \mathrm{pour} ~ i\neq -1,0, H^{-1}(\mathcal{F})<0, H^0(\mathcal{F}) \geq 0 \}, \]
la notation $\mathcal{G} <0$ (resp. $\geq 0$), pour $\mathcal{G} \in \mathrm{Coh}_X$, signifiant que tous les quotients successifs de la filtration de Harder-Narsimhan de $\mathcal{G}$ sont ‡ pentes strictement nÈgatives (resp. positives). La thÈorie gÈnÈrale des \textit{paires de torsion} et du \textit{tilting} montre que $\mathrm{Coh}_X^-$ est un coeur abÈlien de $D(X)$. 

Notre rÈsultat principal, dÈmontrÈ dans la section \ref{sectionfinale}, est alors le suivant. Bien que la courbe $X$ ne vive pas au-dessus de $\mathrm{Spa}(C)$, on peut dÈfinir un morphisme $\tau$ du site des espaces perfectoÔdes sur $X$, muni de la topologie pro-Ètale, vers $\mathrm{Perf}_{C,\mathrm{pro\acute{e}t}}$, qui induit un morphisme $\tau_*$ au niveau des topos correspondants. Dans l'ÈnoncÈ qui suit, on utilise implicitement l'Èquivalence donnÈe par le thÈorËme de puretÈ de Scholze entre $\mathrm{Perf}_{C,\mathrm{pro\acute{e}t}}$ et $\mathrm{Perf}_{C^{\flat},\mathrm{pro\acute{e}t}}$. 
\begin{theoreme} \label{mainintroduction}
Le foncteur cohomologie $R^0\tau_*$ induit une Èquivalence de catÈgories abÈliennes entre $\mathrm{Coh}_X^-$ et $\mathcal{BC}$. 
\end{theoreme}

\begin{remarque}
Fontaine \cite{mailfontaine} a obtenu indÈpendamment un analogue \og Galois-Èquivariant \fg{} du thÈorËme \ref{mainintroduction}, mais avec pour dÈfinition de l'analogue galoisien de $\mathcal{BC}$ celle de \cite{FontKa}, qui est plus proche de la dÈfinition originale de Colmez que de la dÈfinition \ref{maindefintro}.
\end{remarque}

L'exactitude du foncteur $R^0\tau_*$ en restriction ‡ $\mathrm{Coh}_X^-$ dÈcoule des propriÈtÈs de la cohomologie des faisceaux cohÈrents sur la version relative de la courbe de Fargues-Fontaine, qui permettent de donner une autre dÈfinition de $\mathrm{Coh}_X^-$ : voir la section \ref{perversitÈ}. Un certain nombre de corollaires du thÈorËme sont rassemblÈs dans la section \ref{sectionfinale} : en particulier, le fait que la catÈgorie $\mathcal{BC}$ ne dÈpend que de $C^{\flat}$, le fait que les espaces de Banach-Colmez sont des diamants et une caractÈrisation cohomologique des algËbres sympathiques.

Il est facile de voir que l'image par $R^0\tau_*$ de $\mathrm{Coh}_X^-$ tombe dans $\mathcal{BC}$. Le reste de la preuve est plus difficile. Le point clÈ est de rÈussir ‡ dÈcrire les groupes d'extensions entre les faisceaux $\underline{\qp}$ et $\mathbf{G}_a$ dans la catÈgorie des faisceaux de $\qp$-espaces vectoriels sur $\mathrm{Perf}_{C,\mathrm{pro\acute{e}t}}$, en petits degrÈs et de les comparer aux groupes d'extensions analogues dans $\mathrm{Coh}_X^-$, que l'on sait calculer plus facilement. Ces calculs sont effectuÈs dans la section \ref{extproÈtales} de ce texte.

\begin{remarque} Les calculs de certains groupes d'extensions dans la catÈgorie des faisceaux de $\qp$-espaces vectoriels sur $\mathrm{Perf}_{C,\mathrm{pro\acute{e}t}}$ entre les faisceaux $\qp$ et $\mathbf{G}_a$, en petit degrÈ sont rÈminiscents de ceux de Breen \cite{breenperf} en caractÈristique $p$, mÍme s'ils sont considÈrablement plus simples (la seule chose dont nous ayons besoin est la rÈsolution partielle explicite utilisÈe par Berthelot-Breen-Messing dans \cite{bbm}) et beaucoup moins forts. Dans le cas des extensions de $\mathbf{G}_a$ par lui-mÍme, nous montrons que 
\[ \mathrm{Hom}(\mathbf{G}_a,\mathbf{G}_a) = C ~ ; ~ \mathrm{Ext}^1(\mathbf{G}_a,\mathbf{G}_a)= C.\]
Il ne fait aucun doute que l'on devrait avoir $\mathrm{Ext}^i(\mathbf{G}_a,\mathbf{G}_a)=0$ en tout degrÈ $>1$. Dans \cite{breenperf}, Breen prouve que 
\[ \mathrm{Ext}^i(\mathbf{G}_a,\mathbf{G}_a) = 0 \]
pour tout $i>0$, o˘ cette fois-ci le groupe considÈrÈ est un groupe d'extensions dans la catÈgorie des faisceaux de $\mathbf{F}_p$-espaces vectoriels sur le site des schÈmas \textit{parfaits} sur une base parfaite de caractÈristique $p$, muni de la topologie Ètale (ou fppf)\footnote{L'asymÈtrie apparente (pour $i=1$) entre les deux situations n'en est pas vraiment une : si dans la dÈfinition des espaces de Banach-Colmez, on remplace le corps $\qp$ par un corps local $E$ de caractÈristique $p$, $\mathrm{Ext}^1(\mathbf{G}_a,\mathbf{G}_a)$ s'annule aussi dans la catÈgorie des faisceaux de $E$-espaces vectoriels sur $\mathrm{Perf}_{C,\mathrm{pro\acute{e}t}}$.}. 
\end{remarque}

L'usage de la rÈsolution partielle explicite de \cite{bbm} ramËne le calcul des groupes d'extensions ‡ des calculs de cohomologie pro-Ètale, que nous faisons donc au prÈalable dans la section \ref{calculcohom}. La connaissance de ces groupes en petit degrÈ est suffisante, mais nous avons choisi de mener le calcul en tout degrÈ. Plus prÈcisÈment, nous dÈmontrons le rÈsultat suivant, d'intÈrÍt indÈpendant de l'Ètude de la catÈgorie $\mathcal{BC}$.

\begin{theoreme}
Soit $n\geq 1$ et $i\geq 0$. Notons 
\[ \O(\mathbf{A}_C^n) \overset{d_0} \longrightarrow \Omega^1(\mathbf{A}_C^n) \overset{d_1} \longrightarrow \dots \overset{d_{n-1}} \longrightarrow \Omega^n(\mathbf{A}_C^n) \]
le complexe des sections globales du complexe de de Rham de $\mathbf{A}_C^n$. Alors, d'une part, pour tout $i\geq 0$,
\[ H^i(\mathbf{A}_C^n,\mathbf{G}_a) = \Omega^i(\mathbf{A}_C^n). \]
D'autre part, $H^0(\mathbf{A}_C^n,\qp)=\qp$ et pour tout $i > 0$, on a un isomorphisme :
\[ H^i(\mathbf{A}_C^n,\qp) =\mathrm{Ker}(d_{i}) = \mathrm{Im}(d_{i-1}) \subset \Omega^i(\mathbf{A}_C^n). \]
Tous les groupes de cohomologie considÈrÈs sont des groupes de cohomologie pro-Ètale.
\end{theoreme}

\begin{remarque}
Ce rÈsultat a ÈtÈ Ègalement obtenu par Colmez et Nizio\l{} \cite{cn2} ; voir la remarque \ref{cdnsynto}.
\end{remarque}

Dans le cas du faisceau $\qp$, le lecteur notera le contraste avec la cohomologie Ètale, qui est nulle en degrÈ strictement positif. La premiËre partie du thÈorËme se dÈduit facilement des rÈsultats de \cite{Shodge} ; la seconde est nettement plus dÈlicate. Pour $i=0,1$, qui sont les seuls degrÈs nÈcessaires pour l'application ‡ la dÈmonstration de l'Èquivalence entre $\mathrm{Coh}_X^-$ et $\mathcal{BC}$, on peut tout faire ‡ l'aide de la suite exacte de Kummer. En degrÈ quelconque, nous y parvenons ‡ l'aide de la version \og faisceautique \fg{} d'une variante de la suite exacte bien connue en thÈorie de Hodge $p$-adique :
\[ 0 \to \qp \to B_{\rm cris}^{\varphi=1} \to B_{\rm dR}/B_{\rm dR}^+ \to 0. \]
Les mÍmes mÈthodes s'appliquent au disque unitÈ ouvert en dimension quelconque. Elles devraient permettre plus gÈnÈralement d'Ètudier la cohomologie pro-Ètale des espaces rigides Stein lisses dÈfinis sur une extension finie de $\qp$\footnote{Ce point de vue est esquiss\'e \`a la fin de \cite{ACLB}.}. 
\\

Terminons cette introduction en soulignant l'analogie frappante entre la catÈgorie $\mathcal{BC}$ des espaces de Banach-Colmez et celle des groupes quasi-algÈbriques unipotents (ou proalgÈbriques unipotents) de Serre (\cite{Serre}). Non seulement les dÈfinitions de ces deux catÈgories ont un air de famille, mais ces objets font leur apparition dans des contextes similaires. Ils interviennent tous deux en thÈorie du corps de classes : voir respectivement \cite{Serre} et \cite{farguescdc}. Les groupes de cohomologie plate (fppf) de $\mu_p$ sur une variÈtÈ propre et lisse sur un corps parfait de caractÈristique $p$ sont les points de groupes algÈbriques unipotents (\cite{milne}), de mÍme que la cohomologie syntomique gÈomÈtrique donne naissance ‡ des espaces de Banach-Colmez (\cite{niziol}). Il conviendrait d'expliquer et d'explorer cette analogie.

\begin{notations} On note $C^{\flat}$ le basculÈ de $C$. Il s'agit d'un corps valuÈ complet algÈbriquement clos de caractÈristique $p$, dont l'anneau des entiers $\O_{C^{\flat}}$ est
\[ \O_{C^{\flat}} := \underset{x \mapsto x^p} \varprojlim ~ \O_C/p, \]
qui s'identifie aussi (comme monoÔde multiplicatif) ‡ $\varprojlim_{x \mapsto x^p} \O_C$.

On fixe $\epsilon \in \O_{C^{\flat}}$, avec $\epsilon^{(0)}=1$, $\epsilon^{(1)} \neq 1$, un systËme compatible de racines de l'unitÈ et on note 
\[ t = \sum_{n=1}^{\infty} (-1)^{n+1} \frac{([\epsilon]-1)^n}{n} \]
le \og $2i\pi$ de Fontaine \fg{}.
\end{notations}

\subsection{Remerciements}
Mes plus vifs remerciements vont ‡ Laurent Fargues qui m'a fait dÈcouvrir les espaces de Banach-Colmez et m'a posÈ la question qui est ‡ l'origine de ce texte. Sans ses nombreuses suggestions, ses corrections et ses encouragements, ce travail n'aurait probablement pas abouti. Je remercie Peter Scholze pour deux suggestions dÈcisives et une invitation ‡ l'UniversitÈ de Bonn lors de la rÈdaction de ce texte, ainsi que Jean-Marc Fontaine pour son intÈrÍt pour mon travail. Je tiens Ègalement ‡ remercier Gabriel Dospinescu pour des conversations utiles sur la cohomologie pro-Ètale, Matthew Morrow pour ses commentaires sur \cite{BMS}, et tous les deux pour leur patience ‡ l'Ècoute de mes approximations. Bien que je n'aie finalement pas eu ‡ faire usage des techniques de \cite{breen}, je suis reconnaissant ‡ Lawrence Breen pour ses explications ‡ leur sujet. Enfin, Quentin Guignard et Joaqu\'in Rodrigues Jacinto ont eu la gentillesse de me faire part de leurs commentaires sur une version prÈliminaire de ce texte ; je les en remercie.

\section{La catÈgorie des espaces de Banach-Colmez} \label{sec1}

Dans cette section, nous dÈfinissons les espaces de Banach-Colmez comme faisceaux pro-Ètales sur la catÈgorie des espaces perfectoÔdes sur $\mathrm{Spa}(C)$ (dÈfinition \ref{defbc}). Nous expliquons ensuite pourquoi les \textit{revÍtements universels} de groupes $p$-divisibles sur $\O_C$ donnent naissance ‡ des espaces de Banach-Colmez, ce qui permet d'en fabriquer des exemples explicites (corollaire \ref{explicite}). 

\subsection{Espaces perfectoÔdes, topologie pro-Ètale et $v$-topologie}

Dans tout ce paragraphe, $K=C$ ou $K=C^{\flat}$. 
\begin{definition}
Une \textit{$K$-algËbre perfectoÔde} est une $K$-algËbre de Banach $R$ \textit{uniforme}, i.e. telle que l'ensemble $R^{\circ}$ des ÈlÈments de $R$ ‡ puissances bornÈes soit bornÈ, et telle que le Frobenius $\Phi : R^{\circ}/p \to R^{\circ}/p$ soit surjectif (en particulier, une $C^{\flat}$-algËbre perfectoÔde est simplement une $C^{\flat}$-algËbre de Banach uniforme et parfaite). 

Si $R$ est une $C$-algËbre perfectoÔde, le \textit{basculement} $R^{\flat}$ est dÈfini par
\[ R^{\flat,\circ} = \underset{\Phi} \varprojlim ~ R^{\circ}/p \quad ; \quad R^{\flat} = R^{\flat,\circ} \otimes_{\O_{C^{\flat}}} C^{\flat}. \]
C'est une $C^{\flat}$-algËbre perfectoÔde, et $(R^{\flat})^{\circ}=R^{\flat,\circ}$. On dispose d'une application continue et multiplicative $R^{\flat}=\varprojlim_{x \mapsto x^p} R \to R$ de projection sur la premiËre coordonnÈe, notÈe $f \mapsto f^{\sharp}$.  
\end{definition}

\begin{definition}
Une \textit{$K$-algËbre affinoÔde perfectoÔde} est un couple $(R,R^+)$, avec $R$ une $K$-algËbre perfectoÔde et $R^+\subset R^{\circ}$ un sous-anneau ouvert intÈgralement clos.
\end{definition}

\begin{theoreme}
Soit $(R,R^+)$ une $K$-algËbre perfectoÔde et $X=\mathrm{Spa}(R,R^+)$. Alors $\O_X$ est un faisceau et pour tout ouvert rationnel $U \subset X$, $\O_X(U)$ est encore une $K$-algËbre perfectoÔde. 

Si $(R^{\flat},R^{\flat +})$ est la basculÈe de $(R,R^+)$, l'application qui ‡ $x\in X$ associe $x^{\flat} \in X^{\flat} = \mathrm{Spa}(R^{\flat},R^{\flat +})$, dÈfini par $|f(x^{\flat})|=|f^{\sharp}(x)|$\footnote{Rappelons que $|f(x)|$ est la notation consacrÈe pour $x(f)$.}, si $f \in R$, induit un homÈmorphisme entre $X$ et $X^{\flat}$ qui identifie les ouverts rationnels. 
\end{theoreme}

\begin{definition}
Un \textit{espace perfectoÔde} sur $K$ est un espace adique sur $K$ recouvert par des ouverts isomorphes ‡ $\mathrm{Spa}(R,R^+)$, pour certaines $K$-algËbres perfectoÔdes $(R,R^+)$. On notera $\mathrm{Perf}_K$ la catÈgorie des espaces perfectoÔdes sur $K$.
\end{definition}

Introduisons maintenant deux topologies de Grothendieck sur la catÈgorie $\mathrm{Perf}_K$, la topologie pro-Ètale et la $v$-topologie.

\begin{definition}
Soit $f : Y \to X$ un morphisme entre espaces adiques analytiques\footnote{Rappelons qu'un point d'un espace adique est dit \textit{non analytique} si la valuation correspondante a un noyau ouvert. Un espace adique est dit \textit{analytique} si aucun de ses points n'est non analytique. De faÁon Èquivalente, il peut Ítre recouvert par des ouverts $\mathrm{Spa}(R,R^+)$ avec $R$ un anneau de Huber-Tate. En particulier, un espace perfectoÔde est analytique.} sur $K$. Le morphisme $f$ est dit \textit{affinoÔde pro-Ètale} si $X=\mathrm{Spa}(R,R^+)$ et $Y=\mathrm{Spa}(S,S^+)$ sont affinoÔdes et si $Y=\varprojlim Y_i \to X$ s'Ècrit comme limite projective cofiltrante de morphismes Ètales $Y_i \to X$, avec $Y_i=\mathrm{Spa}(S_i,S_i^+)$ affinoÔde. Le morphisme $f : X \to Y$ est dit \textit{pro-Ètale} s'il est affinoÔde pro-Ètale localement sur $X$ et sur $Y$.
\end{definition}

\begin{definition}
i) Le \textit{gros site pro-Ètale de $K$}, notÈ $\mathrm{Perf}_{K,\mathrm{pro\acute{e}t}}$\footnote{Dans la suite, on aura ‡ considÈrer des groupes d'extensions entre faisceaux sur le site $\mathrm{Perf}_{C,\mathrm{pro\acute{e}t}}$. Pour Èviter les problËmes de thÈorie des ensembles, on se restreindra aux espaces perfectoÔdes $X$ sur $C$, avec $\mathrm{card}(|X|) < \kappa$, $\kappa$ Ètant un cardinal inaccessible (donc non dÈnombrable). Cela suffit ‡ garantir que tout objet du site est localement sympathique (voir plus loin). MÍme convention pour les groupes de cohomologie pro-Ètale et les groupes de $v$-cohomologie d'un espace adique analytique.}, est la topologie de Grothendieck sur $\mathrm{Perf}_K$ pour laquelle une famille de morphismes $\{f_i : S_i \to S, i\in I \}$ est un recouvrement si chaque $f_i$ est pro-Ètale et si pour tout ouvert quasi-compact $U$ de $S$, il existe un ensemble fini d'indices $J\subset I$ et des ouverts quasi-compacts $U_i \subset S_i$ pour chaque $i\in J$, tels que $U=\bigcup_{i\in J} f_i(U_i)$.

ii) Soit $Y$ un espace adique analytique sur $K$. Le \textit{petit site pro-Ètale de $Y$} est la topologie de Grothendieck sur la catÈgorie des $f : X \to Y$ pro-Ètales sur $Y$, $X \in \mathrm{Perf}_K$, avec les recouvrements dÈfinis de faÁon analogue ‡ ce qui prÈcËde.  

iii) Le \textit{$v$-site de $K$}, notÈ $\mathrm{Perf}_{K,v}$, est la topologie de Grothendieck sur $\mathrm{Perf}_{K}$ pour laquelle une famille de morphismes $\{f_i : S_i \to S, i\in I \}$ est un recouvrement si pour tout ouvert quasi-compact $U$ de $S$, il existe un ensemble fini d'indices $J\subset I$ et des ouverts quasi-compacts $U_i \subset S_i$ pour chaque $i\in J$, tels que $U=\bigcup_{i\in J} f_i(U_i)$.
\end{definition}

\begin{theoreme}\label{eqscholze}
Le foncteur basculement induit des Èquivalences entre sites $\mathrm{Perf}_{C,\mathrm{pro\acute{e}t}} \simeq \mathrm{Perf}_{C^{\flat},\mathrm{pro\acute{e}t}}$ et $\mathrm{Perf}_{C,v} \simeq \mathrm{Perf}_{C^{\flat},v}$.
\end{theoreme}

\subsection{La catÈgorie $\mathcal{BC}$}
Donnons deux exemples simples de faisceaux pour la $v$-topologie (et donc pour la topologie pro-Ètale).

\begin{theoreme}
Le prÈfaisceau $\mathbf{G}_a$ qui ‡ $S \in \mathrm{Perf}_C$ associe $\O_S(S)$ est un faisceau pour la $v$-topologie (et donc aussi pour la topologie pro-Ètale).
\end{theoreme}
La preuve de ce rÈsultat est difficile (\cite[Th. 8.7]{S17}). Le faisceau $\mathbf{G}_a$ est reprÈsentÈ par $\mathbf{A}_C^1$, la droite affine adique sur $C$ (qui n'est pas un espace perfectoÔde).

\begin{proposition}
Soit $T$ un espace topologique. Le prÈfaisceau $\underline{T}$ qui envoie $S \in \mathrm{Perf}_C$ sur $\con(|S|,T)$ est un faisceau pour la $v$-topologie (et donc aussi pour la topologie pro-Ètale). Si $T$ est totalement discontinu et $S$ quasi-compact quasi-sÈparÈ, $\underline{T}(S)=\con(\pi_0(S),T)$. 
\end{proposition}
\begin{proof}
La preuve est la mÍme que celle de \cite[Lem. 4.2.12]{bs} : le point clÈ est qu'un morphisme quasi-compact quasi-sÈparÈ surjectif $f : S' \to S$ entre espaces perfectoÔdes quasi-compacts quasi-sÈparÈs est une application quotient au niveau des espaces topologiques sous-jacents (car gÈnÈralisante), i.e. si $U\subset S$, $U$ est ouvert dans $S$ si et seulement si $f^{-1}(U)$ l'est dans $S'$. 
\end{proof}

Si $T$ est supposÈ profini, en Ècrivant $T$ comme limite projective d'ensembles finis, on obtient que $\underline{T}$ est en fait reprÈsentÈ par l'espace perfectoÔde $\mathrm{Spa}(\con(T,C),\con(T,\mathcal{O}_C))$.

\begin{definition} 
On appellera \textit{$\qp$-Espace Vectoriel de dimension finie} un faisceau de la forme $\underline{V}$, o˘ $V$ est un $\qp$-espace vectoriel de dimension finie\footnote{Pour allÈger les notations, on notera le plus souvent simplement $V$ au lieu de $\underline{V}$.} et \textit{$C$-Espace Vectoriel de dimension finie} un faisceau de la forme $W \otimes_C \mathbf{G}_a$, o˘ $W$ est un $C$-espace vectoriel de dimension finie.
\end{definition}

Ce texte est consacrÈ ‡ l'Ètude de la catÈgorie abÈlienne suivante.
\begin{definition}\label{defbc}
La catÈgorie $\mathcal{BC}$ des \textit{espaces de Banach-Colmez} est la plus petite sous-catÈgorie abÈlienne stable par extensions contenant les faisceaux $\underline{\qp}$ et $\mathbf{G}_a$, de la catÈgorie des faisceaux de $\qp$-espaces vectoriels sur $\mathrm{Perf}_{C,\mathrm{pro\acute{e}t}}$. 
\end{definition}
En particulier, cette catÈgorie contient Èvidemment tous les $\qp$-Espaces Vectoriels de dimension finie et tous les $C$-Espaces Vectoriels de dimension finie. Dans la suite, on notera simplement $\qp$ au lieu de $\underline{\qp}$, pour allÈger les notations. 

\begin{remarks} \label{remcarp}
a) On obtiendrait une catÈgorie Èquivalente en remplaÁant $\mathrm{Perf}_{C,\mathrm{pro\acute{e}t}}$ par $\mathrm{Perf}_{C,v}$ (cf. la remarque \ref{vtopproet}). 

b) Cette dÈfinition est en fait Èquivalente ‡ la dÈfinition originale de Colmez, qui Ètait formulÈe diffÈremment, comme on le verra plus loin (corollaire \ref{deforiginale}).

c) On pourrait remplacer $\qp$ par un corps local $E$ de caractÈristique $p$, $C$ par $C^{\flat}$, et dÈfinir la catÈgorie des \textit{$E$-espaces de Banach-Colmez} comme plus petite sous-catÈgorie abÈlienne stable par extensions contenant les faisceaux $\underline{E}$ et $\mathbf{G}_a$, de la catÈgorie des faisceaux de $E$-espaces vectoriels sur $\mathrm{Perf}_{C^{\flat},\mathrm{pro\acute{e}t}}$ (noter que dÈsormais $\mathbf{G}_a$ est reprÈsentable par un objet de $\mathrm{Perf}_{C^{\flat}}$, ‡ savoir $\mathbf{A}_{C^{\flat}}^{1,\mathrm{perf}}$). 
\end{remarks}

\subsection{RevÍtements universels de groupes $p$-divisibles} \label{revunivpdiv}

Comment fabriquer des exemples intÈressants d'espaces de Banach-Colmez ? Une mÈthode possible, qui a l'avantage d'Ítre de nature gÈomÈtrique, est d'utiliser la thÈorie des groupes $p$-divisibles. Les rÈsultats de ce paragraphe sont dus ‡ Fargues-Fontaine (\cite{FF}) et Scholze-Weinstein (\cite{SW}) ; nous les rappelons pour la commoditÈ du lecteur. 
\\

Soit $G$ un groupe $p$-divisible sur $\O_C$. Son \textit{revÍtement universel} $\tilde{G}$ est le prÈfaisceau qui associe ‡ une $C$-algËbre perfectoÔde $R$ le $\qp$-espace vectoriel
\[ \tilde{G}(R) = \underset{\times p} \varprojlim ~ \underset{k} \varprojlim ~ \underset{n} \varinjlim ~ G[p^n](R^{\circ}/p^k). \]
Ce prÈfaisceau sera vu dans la suite comme un prÈfaisceau sur $\mathrm{Perf}_{C,\mathrm{pro\acute{e}t}}$, par restriction. 

\begin{example}
Si $G$ est un groupe $p$-divisible \textit{Ètale} sur $\O_C$, $G$ est isomorphe ‡ $T(G) \otimes \qp/\zp$, $T(G)$ Ètant le module de Tate de $G$. Dans ce cas, on a immÈdiatement $\tilde{G} = V(G)$ et $\tilde{G}$ est donc un $\qp$-Espace Vectoriel de dimension finie. 
\end{example}

\begin{proposition}
Le foncteur $G \mapsto \tilde{G}$ transforme les isogÈnies en isomorphismes. De plus, si $A$ est une $\O_C$-algËbre $p$-adique, $\tilde{G}(A) \simeq \tilde{G}(A/p)$.
\end{proposition}
\begin{proof}
La premiËre affirmation est une consÈquence immÈdiate du fait que la multiplication par $p$ est un isomorphisme de $\tilde{G}$. Pour la seconde, voir par exemple \cite[Prop. 4.5.2]{FF}.
\end{proof}

Notons $H$ le groupe $p$-divisible sur $k=\bar{\mathbf{F}}_p$ obtenu en rÈduisant $G$ modulo l'idÈal maximal de $\O_C$. On sait (\cite[Th. 5.1.4]{SW}) que $G \otimes_{\O_C} \O_C/p$ est quasi-isogËne ‡ $H \otimes_k \O_C/p$. De la proposition prÈcÈdente on dÈduit alors facilement :
\begin{lemme}
Soit $G$ un groupe $p$-divisible sur $\O_C$ et $H$ sa fibre spÈciale. Notons $\tilde{H}^{\rm can}$ le relËvement canonique\footnote{Il s'agit du foncteur associant $\tilde{H}(R \otimes_{W(k)} k)$ ‡ une $W(k)$-algËbre adique $R$. Cf. \cite[\S 4.6.3]{FF}.} de $\tilde{H}$ ‡ $W(k)$. Alors $\tilde{G} \simeq \tilde{H}^{\rm can} \otimes_{W(k)} \O_C$. 
\end{lemme}

La suite exacte connexe-Ètale de $H$ :
\[ 0 \to H^{\circ} \to H \to H^{\rm et} \to 0 \]
est scindÈe, puisque $k$ est un corps parfait. On en dÈduit immÈdiatement que $\tilde{G}$ est la somme directe d'un $\qp$-Espace Vectoriel de dimension finie et du revÍtement universel d'un groupe $p$-divisible connexe en fibre spÈciale.

On peut donc supposer dÈsormais $G$ connexe en fibre spÈciale (i.e. $H$ connexe). Dans ce cas, le morphisme de systËmes projectifs
\[ \xymatrix{
     H \ar@{=}[d] & \ar[l]^p  H \ar[d]^V & \ar[l]^p   H \ar[d]^{V^2} & \ar[l] \dots \\
     H & \ar[l]_F H^{(p^{-1})} & \ar[l]_F H^{(p^{-2})} & \ar[l]  \dots
  } \]
induit un isomorphisme
\[ \tilde{H}(R^{\circ}/p) \simeq H(R^{\flat \circ}), \] 
puisque $H$ Ètant connexe, il existe une isogÈnie $u$ et un entier $n\geq 1$ tel que $F^n=p \circ u$. Il existe un entier $d$ tel que $H$ est (non canoniquement) isomorphe ‡ $\mathrm{Spf}(k[[T_1,\dots,T_d]])$, et donc $H(R^{\flat \circ}) \simeq (R^{\flat \circ\circ})^d$. On en dÈduit que $\tilde{G}$ est reprÈsentable par l'espace perfectoÔde qui est l'espace adique fibre gÈnÈrique sur $C$ du formel $\mathrm{Spf}(\O_C[[T_1^{1/p^{\infty}},\dots,T_d^{1/p^{\infty}}]])$. Finalement, on a montrÈ dans tous les cas :

\begin{proposition}
Si $G$ est un groupe $p$-divisible sur $\O_C$, $\tilde{G}$ est reprÈsentÈ par un espace perfectoÔde (c'est donc en particulier un faisceau). 
\end{proposition}

\begin{proposition} 
Si $G$ est un groupe $p$-divisible sur $\O_C$, le faisceau $\tilde{G}$ est un espace de Banach-Colmez. 
\end{proposition}
\begin{proof}
Comme on l'a notÈ plus haut, on peut supposer $H$ connexe. Alors $\tilde{G}$ est obtenu en prenant la limite inverse sur la multiplication par $p$ du faisceau associÈ ‡ la fibre gÈnÈrique adique $\mathcal{G}_{\eta}^{\rm ad}$ du schÈma en groupes formel $\mathcal{G}$ obtenu en complÈtant formellement $G$ le long de sa section neutre. Cette fibre gÈnÈrique adique est ‡ distinguer de la fibre gÈnÈrique schÈmatique de $G$, qui est un schÈma en groupes Ètale sur $\mathrm{Spec}(C)$, donc constant de faisceau associÈ $T(G) \otimes \qp/\zp$ ainsi que de la fibre gÈnÈrique formelle $\mathcal{G} \hat{\otimes}_{\O_C} C$ de $\mathcal{G}$, qui est isomorphe ‡ $\mathrm{Lie}(G) \otimes \widehat{\mathbf{G}}_a$.

On a une suite exacte d'espaces adiques en groupes (\cite[Th. 1.2]{granalyt}) : 
\[ 0 \to T(G) \otimes \qp/\zp = \mathcal{G}_{\eta}^{\rm ad}[p^{\infty}] \to \mathcal{G}_{\eta}^{\rm ad} \longrightarrow \mathrm{Lie}(G) \otimes \mathbf{G}_a = (\mathcal{G} \hat{\otimes}_{\O_C} C)^{\rm ad} \to 0, \]
donnÈe par l'application logarithme $\log :  \mathcal{G}_{\eta}^{\rm ad} \to \mathrm{Lie}(G) \otimes \mathbf{G}_a$, qui est un revÍtement au sens de de Jong. On peut appliquer ‡ cette suite exacte le foncteur $\varprojlim_{\mathbf{N}} (\cdot)$, les morphismes de transition Ètant $(\cdot) \overset{\times p} \longrightarrow (\cdot)$. On obtient une suite exacte de faisceaux pro-Ètales :
\begin{eqnarray} 0 \to V(G) \to \tilde{G} \to \mathrm{Lie}(G) \otimes \mathbf{G}_a \to 0. \label{log} \end{eqnarray}
Cette suite exacte montre  $\tilde{G}$ est un espace de Banach-Colmez.
\end{proof}

Le lemme de Yoneda et le fait que les perfectoÔdes forment une base de la topologie pro-Ètale permettent de voir la catÈgorie dont les objets sont les revÍtements universels de groupes $p$-divisibles sur $\O_C$ (et les morphismes les morphismes d'espaces adiques en groupes) comme une sous-catÈgorie pleine de la catÈgorie des faisceaux abÈliens sur $\mathrm{Perf}_{C,\mathrm{pro\acute{e}t}}$.

\begin{definition}
On note $\mathcal{BC}^{\rm rep}$ la sous-catÈgorie pleine de $\mathcal{BC}$ formÈe des revÍtements universels de groupes $p$-divisibles.
\end{definition}

La thÈorie des groupes $p$-divisibles permet de dÈcrire explicitement les objets de $\mathcal{BC}^{\rm rep}$ : 

Soit $R$ une $C$-algËbre perfectoÔde. Le complÈtÈ $p$-adique $A_{\rm cris}(R)$ de l'enveloppe ‡ puissances divisÈes de la surjection $W(R^{\flat \circ}) \to R^{\circ}/p$ est l'Èpaississement ‡ puissances divisÈes $p$-adiquement complet universel de $R^{\circ}/p$. La construction de $A_{\rm cris}(R)$ est fonctorielle en $R$. On note $B_{\rm cris}^+(R)=A_{\rm cris}(R)[1/p]$ et simplement $B_{\rm cris}^+$ si $R=C$. Si $\Gamma$ est un groupe $p$-divisible sur $R^{\circ}/p$, on notera $\mathbf{M}(\Gamma)$ l'Èvaluation du cristal de DieudonnÈ de $\Gamma$ sur $A_{\rm cris}(R)$. Si $\Gamma$ provient par extension des scalaires de $k$ ‡ $R^{\circ}/p$ d'un groupe $p$-divisible $H$, $\mathbf{M}(\Gamma)$ est simplement $M(H) \otimes_L A_{\rm cris}(R)$, $M(H)$ Ètant le module de DieudonnÈ usuel de $H$. On a le thÈorËme suivant (\cite[Th. A]{SW}) :

\begin{theoreme}\label{scholzeweinstein}
Si $R$ est une $C$-algËbre perfectoÔde, le foncteur $\Gamma \mapsto \mathbf{M}(\Gamma)[1/p]$ de la catÈgorie des groupes $p$-divisibles sur $R^{\circ}/p$ ‡ isogÈnie prËs dans la catÈgorie des $B_{\rm cris}^+(R)$-modules finis projectifs avec Frobenius est pleinement fidËle.
\end{theoreme}

On en dÈduit la

\begin{proposition}
Soit $R$ une $C$-algËbre affinoÔde perfectoÔde. On a des isomorphismes fonctoriels en $R$
\[ \tilde{G}(R) = (B_{\rm cris}^+(R) \otimes_{\breve{\q}_p} M(H))^{\varphi=p}. \]
\end{proposition}
\begin{proof}
Il suffit de remarquer que 
\begin{align*}
\tilde{G}(R) =\tilde{H}(R^{\circ}/p) &= \mathrm{Hom}_{R^{\circ}/p}(\qp/\zp,H)[1/p] \\
&= \mathrm{Hom}_{B_{\rm cris}^+(R),\varphi}(\mathbf{M}(\qp/\zp)[1/p],\mathbf{M}(H)[1/p])
\end{align*}
(la derniËre ÈgalitÈ venant du thÈorËme \ref{scholzeweinstein}) et d'utiliser la description de $\mathbf{M}(H)$. 
\end{proof}

\begin{remarque}\label{utile}
On a observÈ plus haut que si $H$ est connexe
\[ \tilde{G}(R) = H(R^{\flat \circ}) \]
pour toute $C$-algËbre perfectoÔde $R$. Autrement dit, via l'Èquivalence de Scholze $\mathrm{Perf}_{C,\mathrm{pro\acute{e}t}}\simeq \mathrm{Perf}_{C^{\flat},\mathrm{pro\acute{e}t}}$, le faisceau $\tilde{G}$ correspond ‡ la fibre gÈnÈrique adique de $H \otimes \O_C^{\flat}$. Par pleine fidÈlitÈ du foncteur de DieudonnÈ (utilisÈe cette fois-ci sur $\O_{C^{\flat}}$), si $G$ et $G'$ sont deux groupes $p$-divisibles sur $\O_C$ connexes en fibre spÈciale :
\[ \mathrm{Hom}(\tilde{G},\tilde{G'})=\mathrm{Hom}(H \otimes \O_{C^{\flat}},H' \otimes \O_{C^{\flat}})=\mathrm{Hom}_{B_{\rm cris}^+,\varphi}(B_{\rm cris}^+ \otimes_L M(H),B_{\rm cris}^+ \otimes_L M(H')), \]
les deux premiers $\mathrm{Hom}$ Ètant des $\mathrm{Hom}$ comme faisceaux.
%Si $H=\qp/\zp$, $\tilde{G}$ correspond au faisceau constant $\qp$ sur $\mathrm{Perf}_{\O_{C^{\flat}}}$ et donc on a encore
%\begin{align*}
%\mathrm{Hom}(\tilde{G},\tilde{G'})=\mathrm{Hom}(\qp,H' \otimes \O_{C^{\flat}}) &=\mathrm{Hom}(\qp/\zp,H' \otimes \O_{C^{\flat}})[1/p] \\
%&= \mathrm{Hom}_{B_{\rm cris}^+,\varphi}(B_{\rm cris}^+ \otimes_L M(H),B_{\rm cris}^+ \otimes_L M(H'))  
%\end{align*}
%d'aprËs le thÈorËme \ref{scholzeweinstein}\todo[size=\tiny]{les autres cas}.
Cette remarque sera utile plus loin.
\end{remarque}

\begin{example}
Soit $G= \mu_{p^{\infty}}$. Alors $\tilde{G}(R)=B_{\rm cris}^+(R)^{\varphi=p}$. Si $R=C$, la suite obtenue en prenant les $C$-points de la suite exacte \eqref{log} pour $G$ reste exacte et c'est la suite exacte\footnote{Il serait plus habituel d'Ècrire $\qp(1)$ au lieu de $\qp$, mais nous voyons simplement ici cette suite exacte comme suite des $C$-points d'espaces de Banach-Colmez, sans nous prÈoccuper de l'action de Galois.} 
\[ 0 \to \qp \to (B_{\rm cris}^+)^{\varphi=p} \overset{\theta} \longrightarrow C \to 0, \]
souvent appelÈe \textit{suite exacte fondamentale de la thÈorie de Hodge $p$-adique}, qui apparaÓt un peu partout en thÈorie de Hodge $p$-adique et est ‡ l'origine de la thÈorie des espaces de Banach-Colmez.
\end{example}

\begin{corollaire}\label{explicite}
Tout objet de $\mathcal{BC}^{\rm rep}$ est somme directe de faisceaux $\mathbf{U}_{\lambda}:=(B_{\rm cris}^+(\cdot))^{\varphi^h=p^d}$, avec $0 \leq \lambda=\frac{d}{h} \leq 1$.
\end{corollaire}
\begin{proof}
C'est un corollaire direct de la classification de DieudonnÈ-Manin qui affirme que la catÈgorie des isocristaux associÈs aux groupes $p$-divisibles sur $k$ est semi-simple, avec un unique objet simple de pente $\lambda$ pour tout $\lambda \in \q \cap [0,1]$, et de la proposition prÈcÈdente.
\end{proof}  

Ce corollaire donne une description agrÈable de la sous-catÈgorie $\mathcal{BC}^{\rm rep}$. Toutefois ‡ ce stade, la structure de la catÈgorie $\mathcal{BC}$ tout entiËre reste encore mystÈrieuse : elle ne sera ÈlucidÈe que dans la partie \ref{sectionfinale}. Pour mieux comprendre la catÈgorie $\mathcal{BC}$, il nous faut commencer par analyser les morphismes et extensions entre $\qp$-Espaces Vectoriels de dimension finie et $C$-Espaces Vectoriels de dimension finie. 

\section{Quelques calculs de cohomologie pro-Ètale} \label{calculcohom}

Pour mener ‡ bien le calcul des groupes d'extensions de la section \ref{extproÈtales}, il nous faut au prÈalable Ítre capable de dÈcrire certains groupes de cohomologie pro-Ètale. Ces calculs ont un intÈrÍt en soi indÈpendamment du problËme que nous avons en vue et les mÈthodes utilisÈes permettent de donner des rÈsultats plus gÈnÈraux et plus prÈcis. Nous les effectuons dans le cas qui nous intÈresse : celui de l'espace affine, pour les faisceaux $\qp$ et $\mathbf{G}_a$. Les mÍmes techniques permettent de calculer la cohomologie du disque unitÈ ouvert en toute dimension.  
\\

\textit{Sauf mention du contraire, tous les groupes de cohomologie considÈrÈs sont des groupes de cohomologie pro-Ètale.}
\\

Nous allons dÈmontrer les rÈsultats suivants\footnote{Dans la suite nous ignorons systÈmatiquement l'action de Galois et donc les twists ‡ la Tate dans l'ÈnoncÈ des rÈsultats, bien qu'il soit facile de les suivre ‡ la trace.}. 
\begin{theoreme} \label{cohga}
Soit $n\geq 1$ et $i\geq 0$. On a :
\[ H^i(\mathbf{A}_C^n,\mathbf{G}_a) = \Omega^i(\mathbf{A}_C^n). \]
\end{theoreme}

\begin{theoreme} \label{cohespaffine}
Notons 
\[ \O(\mathbf{A}_C^n) \overset{d_0} \longrightarrow \Omega^1(\mathbf{A}_C^n) \overset{d_1} \longrightarrow \dots \overset{d_{n-1}} \longrightarrow \Omega^n(\mathbf{A}_C^n) \]
le complexe des sections globales du complexe de de Rham de $\mathbf{A}_C^n$. Alors $H^0(\mathbf{A}_C^n,\qp)=\qp$ et pour tout $i > 0$, on a un isomorphisme :
\[ H^i(\mathbf{A}_C^n,\qp) =\mathrm{Ker}(d_{i}) = \mathrm{Im}(d_{i-1}) \subset \Omega^i(\mathbf{A}_C^n). \]
\end{theoreme}

\begin{remarque}
Les cohomologies Ètale et pro-Ètale du faisceau constant $\z/p^k$ sur l'espace affine sont les mÍmes, d'aprËs la proposition \ref{competproet} ci-dessous. Or Berkovich \cite{berko} a prouvÈ que pour tout $i>0$, $H_{\mathrm{\acute{e}t}}^i(\mathbf{A}_C^n,\z/p^k)=0$. On en dÈduit facilement que $H^i(\mathbf{A}_C^n,\zp)=0$ pour tout $i>0$. Cela ne contredit Èvidemment pas le thÈorËme \ref{cohespaffine}, puisque l'espace affine n'est pas quasi-compact.
\end{remarque}

\begin{proof}[DÈmonstration du thÈorËme \ref{cohga}] 
La cohomologie de $\mathbf{G}_a$ se calcule aisÈment, gr‚ce au rÈsultat remarquable suivant\footnote{Qui est, notamment, ‡ l'origine de la suite spectrale de Hodge-Tate.} (\cite[Prop. 3.23]{CDM}, o˘ le faisceau que nous notons $\mathbf{G}_a$ est appelÈ $\widehat{\O}$) :

\begin{proposition} \label{scholzega}
Soit $n\geq 1$ et $\nu' : \widetilde{\mathbf{A}}_{C,\mathrm{pro\acute{e}t}}^n \to \widetilde{\mathbf{A}}_{C,\mathrm{\acute{e}t}}^n$ le morphisme du topos pro-Ètale vers le topos Ètale de $\mathbf{A}_C^n$. On a 
\[ R^i \nu'_* \mathbf{G}_{a} = \Omega_{\mathbf{A}_C^n}^i, \]
pour tout $i\geq 0$ (en particulier ces faisceaux sont nuls si $i>n$).
\end{proposition}
Comme $\mathbf{A}_C^n$ est Stein, donc n'a pas de cohomologie cohÈrente en degrÈ positif, on en dÈduit que pour tout $i$,
\[ H^i(\mathbf{A}_C^n,\mathbf{G}_a) = \Omega^i(\mathbf{A}_C^n), \]
gr‚ce ‡ la suite spectrale de Leray pour le morphisme $\nu$.
\end{proof}

La dÈmonstration du thÈorËme \ref{cohespaffine} est plus difficile et le reste de cette section lui est consacrÈ. Nous commenÁons par le cas de la dimension $1$ ‡ l'aide de la suite exacte de Kummer et de thÈorËmes d'annulation de Berkovich\footnote{Dans la suite, nous utiliserons sans plus de commentaire le fait que la cohomologie Ètale d'un espace analytique Hausdorff au sens de Berkovich est la mÍme que celle de la variÈtÈ rigide ou de l'espace adique associÈ. Voir \cite[Ch. 8]{huber}.}. Puis nous traitons le cas gÈnÈral ‡ l'aide de la version faisceautique de la \textit{suite exacte fondamentale} en thÈorie de Hodge $p$-adique. Notons que nous n'utiliserons dans la suite du texte que le calcul de la cohomologie en degrÈs $0$ et $1$, pour lequel la suite exacte de Kummer suffit. \textit{Le lecteur qui le souhaite peut donc sauter en premiËre lecture le paragraphe \ref{dimgenerale}}.

\subsection{Cohomologie de la droite affine}  

Rappelons tout d'abord la proposition suivante, qui servira constamment (\cite[Prop. 13.2.2]{ega3} et \cite[Rem. 13.2.4]{ega3}).

\begin{proposition} \label{mittag}
Soit $I$ un ensemble ordonnÈ filtrant ayant une partie cofinale dÈnombrable et $(A_i)_{i\in I}$ un systËme projectif de groupes abÈliens. On note $f_{ij} : A_j \to A_i$ pour $j\geq i$. Si $(A_i)_i$ vÈrifie la condition de Mittag-Leffler,
\[ R^1 \underset{i} \varprojlim A_i = 0. \]
La conclusion reste valable si l'on suppose que chacun des groupes $A_i$ est muni d'une structure d'espace mÈtrique complet compatible ‡ la structure de groupe, que les $f_{ij}$ sont uniformÈment continues et que pour tout $i$, il existe $j\geq i$ tel que pour tout $k\geq j$, $f_{ik}(A_k)$ est dense dans $f_{ij}(A_j)$. 
\end{proposition}

Un autre fait utile est le suivant.
\begin{proposition} \label{replete}
Soit $I$ un ensemble ordonnÈ filtrant ayant une partie cofinale dÈnombrable et $(\mathcal{F}_i)_{i\in I}$ un systËme projectif de faisceaux abÈliens pour la topologie pro-Ètale sur un espace adique analytique $X$, avec morphismes de transition surjectifs. Alors
\[ \forall k>0, ~ R^k ~ \underset{i} \varprojlim ~ \mathcal{F}_i = 0. \]
\end{proposition}
\begin{proof}
Voir \cite[Prop. 3.1.10]{bs}.
\end{proof}

On utilisera Ègalement le rÈsultat de comparaison suivant.

\begin{proposition} \label{competproet}
Soit $X$ un espace adique analytique sur $C$ et $\mathcal{F}$ un faisceau Ètale de groupes abÈliens sur $X$. Notons $\nu : \widetilde{X}_{\mathrm{pro\acute{e}t}} \to \widetilde{X}_{\mathrm{\acute{e}t}}$ le morphisme de topos. Alors pour tout $i \geq 0$,
\[ H_{\mathrm{pro\acute{e}t}}^i(X,\nu^* \mathcal{F}) = H_{\mathrm{\acute{e}t}}^i(X,\mathcal{F}). \]
\end{proposition}
\begin{proof}
Ceci est dÈmontrÈ dans \cite[Prop. 14.8]{S17}, avec $X$ remplacÈ par $X^{\diamond}$, le diamant associÈ ‡ l'espace adique analytique $X$. Or la cohomologie pro-Ètale de $X$ est la mÍme que celle de $X^{\diamond}$, par dÈfinition des sites pro-Ètales et l'Èquivalence $\mathrm{Perf}_{C,\mathrm{pro\acute{e}t}} \simeq \mathrm{Perf}_{C^{\flat},\mathrm{pro\acute{e}t}}$. La cohomologie Ètale de $X$ est aussi la mÍme que celle de $X^{\diamond}$ car les sites Ètales de $X$ et $X^{\diamond}$ sont les mÍmes : en effet, les morphismes Ètales descendent pour la topologie pro-Ètale\footnote{Et mÍme pour la $v$-topologie, cf. \cite[Prop 9.7]{S17}.} (voir la deuxiËme partie de la preuve de \cite[Prop. 9.7]{S17}). 
\end{proof}

\begin{proposition}\label{cohodisque}
On note $D$ le disque unitÈ ouvert de dimension $1$ sur $C$. On a $H^i(D,\qp)=0$ pour tout $i>1$. De plus, $H^1(D,\qp)$ s'identifie ‡ l'espace $\O(D)_0$ des fonctions rigides analytiques sur $D$ nulles en zÈro. 
\end{proposition}
\begin{proof}
Pour tout $n>1$, notons $\bar{B}_n= \mathrm{Spa}(C\langle p^{-1/n} T \rangle)$ la boule fermÈe de rayon $p^{-1/n}$.

Montrons dans un premier temps que $H_{\rm \acute{e}t}^i(\bar{B}_n,\qp)$\footnote{Si $X$ est un espace adique, pour nous $H_{\rm \acute{e}t}^i(X,\zp) = \varprojlim H_{\rm \acute{e}t}^i(X,\z/p^k)$ et $H_{\rm \acute{e}t}^i(X,\qp)=H_{\rm \acute{e}t}^i(X,\zp)[1/p]$, \textit{par dÈfinition}.} est nul pour $i>1$. D'aprËs \cite[Th. 4.2.6]{berk}, $H_{\rm \acute{e}t}^i(\bar{B}_n,\z/p^k)=0$ si $i>2$. On a donc $H_{\rm \acute{e}t}^i(\bar{B}_n,\qp)=0$ si $i>2$.  La suite exacte de Kummer donne une suite exacte longue
\begin{align*} 0 \to \O(\bar{B}_n)^*/(\O(\bar{B}_n)^*)^{p^k} \to H_{\mathrm{\acute{e}t}}^1(\bar{B}_n,\z/p^k) & \to H_{\mathrm{\acute{e}t}}^1(\bar{B}_n,\mathbf{G}_m) \\
& \to H_{\mathrm{\acute{e}t}}^1(\bar{B}_n,\mathbf{G}_m) \to H_{\mathrm{\acute{e}t}}^2(\bar{B}_n,\z/p^k) \to H_{\mathrm{\acute{e}t}}^2(\bar{B}_n,\mathbf{G}_m). \end{align*}
Or d'aprËs \cite[Lem. 6.1.2]{berk}, $H_{\mathrm{\acute{e}t}}^2(\bar{B}_n,\mathbf{G}_m)=0$ ; comme on a aussi $\mathrm{Pic}(\bar{B}_n)=0$ (\cite[Satz 1]{lutk}), on obtient 
\[ H_{\mathrm{\acute{e}t}}^2(\bar{B}_n,\z/p^k)=0. \]
On en dÈduit que $H_{\rm \acute{e}t}^2(\bar{B}_n,\qp)=0$.

Calculons maintenant $H^1(D,\qp)$. On regarde la suite exacte\footnote{Dont l'existence dÈcoule de la proposition \ref{replete}, mais se vÈrifie facilement dans ce cas en utilisant \cite[Prop. 13.3.1]{ega3}.} :
\[ 0 \to R^1 \underset{n} \varprojlim ~ H^{0}(\bar{B}_n,\qp) \to H^1(D,\qp) \to \underset{n} \varprojlim ~ H^1(\bar{B}_n,\qp) \to 0. \]
Le terme de gauche est nul (proposition \ref{mittag}). Or $H^1(\bar{B}_n,\qp)= H^1(\bar{B}_n,\zp)[1/p]$ par quasi-compacitÈ de $\bar{B}_n$ (\cite[Lem. 21.17.1]{stacks}) et on a une suite exacte (proposition \ref{replete}) :
\[ 0 \to R^1 \underset{k} \varprojlim ~ H^{0}(\bar{B}_n,\z/p^k) \to H^1(\bar{B}_n,\zp) \to \underset{k} \varprojlim ~ H^1(\bar{B}_n,\z/p^k) \to 0, \]
o˘ ‡ nouveau le terme de gauche est nul. La cohomologie du faisceau $\z/p^k$ sur $\bar{B}_n$ se calcule indiffÈremment pour la topologie Ètale ou pro-Ètale, d'aprËs la proposition \ref{competproet}. La suite exacte de Kummer donne donc que pour tout $n$,
\[ H_{\mathrm{\acute{e}t}}^1(\bar{B}_n,\zp)=\underset{k} \varprojlim ~ \O(\bar{B}_n)^*/(\O(\bar{B}_n)^*)^{p^k}. \]
On a 
\[ \O(\bar{B}_n)^*= C^* \cdot \{ 1+ \sum_{i\geq 1} a_i X^i, \forall i, |a_i|p^{-i/n} < 1 \}. \]
Notons $M_n=\{ 1+ \sum_{i\geq 1} a_i X^i, \forall i, |a_i|p^{-i/n} < 1 \}$, de sorte que l'on a aussi $H_{\mathrm{\acute{e}t}}^1(\bar{B}_n,\zp)= \varprojlim ~ M_n/M_n^{p^k}$.  

Si $f \in M_n$, $\log \circ f$ est bien dÈfinie et est un ÈlÈment de l'espace $\O(\bar{B}_n)_0$ des fonctions rigides analytiques sur $\bar{B}_n$ nulles en zÈro. Notons $\Phi_n(f)$ la restriction de cet ÈlÈment ‡ $\bar{B}_{n-1}$. Nous affirmons que l'application $\Phi_n$ s'Ètend ‡ $\varprojlim M_n/M_n^{p^k}$. En effet, soit $f\in M_n$. Alors $\Phi_n(f) \in \O(\bar{B}_{n-1})_0^{\leq r_n}$, o˘ $r_n=\sup \{ |z|, z \in \log(B(1,p^{-1/(n(n-1))})\}$ et o˘ $\O(\cdot)_0^{\leq r}$ pour $r>0$ dÈsigne l'ensemble des fonctions bornÈes par $r$. Si $f \in M_n^{p^k}$, $\Phi_n(f) \in p^k \O(\bar{B}_{n-1})_0^{\leq r_n}$. On en dÈduit que $\Phi_n$ induit une flËche de $\varprojlim ~ M_n/M_n^{p^k}$ vers le complÈtÈ $p$-adique de $\O(\bar{B}_{n-1})_0^{\leq r_n}$, qui est $\O(\bar{B}_{n-1})_0^{\leq r_n}$ lui-mÍme, et donc en particulier une flËche de $\varprojlim ~ M_n/M_n^{p^k}$ vers $\O(\bar{B}_{n-1})_0$. On a donc dÈfini pour chaque $n$ un morphisme 
\[ \Phi_n : H^1(\bar{B}_n,\qp) = (\underset{k} \varprojlim ~ M_n/M_n^{p^k})[1/p] \to \O(\bar{B}_{n-1})_0. \]
Ces flËches sont Èvidemment compatibles quand $n$ varie et donnent donc un morphisme continu 
\[ \Phi : H^1(D,\qp) \to \underset{n} \varprojlim ~ \O(\bar{B}_{n-1})_0 = \O(D)_0. \] 
C'est un isomorphisme : pour le voir, construisons la bijection rÈciproque. Soit $g \in \O(D)_0$, vue comme une suite $(g_n)_n$ d'ÈlÈments de $\O(\bar{B}_{n-1})_0$. Pour tout $n$, $g_n$ est bornÈe, donc il existe $k_n \geq 0$ tel que $p^{k_n} g_n$ soit ‡ valeurs dans la boule centrÈe en $0$ de rayon $p^{-1/(p-1)}$. Alors $\exp(g_n)$ est bien dÈfini ; c'est une fonction inversible sur $\bar{B}_{n-1}$, que l'on peut en particulier voir comment un ÈlÈment de $H^1(\bar{B}_{n-1},\qp)$. On pose alors $f_n = p^{-k_n} \exp(g_n) \in H^1(\bar{B}_{n-1},\qp)$ (pour la structure de $\qp$-espace vectoriel de $H^1(\bar{B}_{n-1},\qp)$). Alors $f=(f_n)_n$ est un antÈcÈdent de $g$ par $\Phi$.  

Pour conclure la preuve de la proposition, il ne reste plus qu'‡ montrer que les $H^i(D,\qp)$ sont nuls pour $i>1$. Soit $B_n$ la boule ouverte de rayon $p^{-1/n}$. On a une suite exacte
\[ 0 \to R^1 \underset{n} \varprojlim ~ H^{1}(B_n,\qp) \to H^2(D,\qp) \to \underset{n} \varprojlim ~ H^2(B_n,\qp) \to 0. \]
Le terme de gauche s'annule, car les flËches de restriction $\O(B_{n+1})_0 \to \O(B_n)_0$ sont d'image dense (proposition \ref{mittag}). De plus dans la limite inverse de droite, on peut remplacer les groupes de cohomologie pro-Ètale par des groupes de cohomologie Ètale. En effet, on peut remplacer les $B_n$ par le systËme cofinal des $\bar{B}_n$. Alors, comme on l'a vu ci-dessus, $H^i(\bar{B}_n,\qp)=H_{\rm \acute{e}t}^i(\bar{B}_n,\qp)$ pour tout $i$ (facile). On en dÈduit avec ce qu'on a dit plus haut que le terme de droite est Ègalement nul et donc $H^2(D,\qp)=0$. Enfin pour $i\geq 3$, on utilise la suite exacte
\[  0 \to R^1 \underset{n} \varprojlim ~ H^{i-1}(\bar{B}_n,\qp) \to H^i(D,\qp) \to \underset{n} \varprojlim ~ H^i(\bar{B}_n,\qp) \to 0 \] 
et le fait que les termes de gauche et de droite sont nuls (toujours parce qu'on peut remplacer cohomologie pro-Ètale par cohomologie Ètale et par le dÈbut de la dÈmonstration).
\end{proof}

\begin{remarque}
MÍme en dimension $1$, il semble plus dÈlicat de dÈcrire \textit{explicitement} le $H^1$ Ètale ‡ coefficients $\qp$ de la boule ouverte ou fermÈe, ou le $H^1$ pro-Ètale du disque fermÈ.
\end{remarque}

On en dÈduit le cas $n=1$ du thÈorËme \ref{cohespaffine}.
\begin{corollaire}\label{cohodroiteaff}
On a $H^i(\mathbf{A}_C^1,\qp)=0$ pour tout $i>1$. De plus, $H^0(\mathbf{A}_C^1,\qp)=\qp$ et $H^1(\mathbf{A}_C^1,\qp)$ s'identifie ‡ l'espace $\O(\mathbf{A}_C^1)_0$ des fonctions rigides analytiques sur $\mathbf{A}_C^1$, nulles en $0$.
\end{corollaire}
\begin{proof}
Notons pour tout $m>0$, $D_m$ la boule ouverte de rayon $m$. L'existence de la suite exacte
\[ 0 \to R^1 \underset{m} \varprojlim ~ H^{i-1}(D_m,\qp) \to H^i(\mathbf{A}_C^1,\qp) \to \underset{m} \varprojlim ~ H^i(D_m,\qp) \to 0 \]
et la proposition \ref{cohodisque} donnent immÈdiatement le rÈsultat (pour $i=2$, on utilise une fois de plus la proposition \ref{mittag}). 
\end{proof}

Notons qu'exactement la mÍme preuve donne Ègalement le rÈsultat suivant, en dimension quelconque. C'est de celui-ci que nous ferons usage dans la section \ref{extproÈtales}.
\begin{proposition}
Soit $n\geq 1$. On a $H^0(\mathbf{A}_C^n,\qp)=\qp$ et $H^1(\mathbf{A}_C^n,\qp)$ s'identifie ‡ l'espace des fonctions rigides analytiques sur $\mathbf{A}_C^n$, nulles en $0$.
\end{proposition}

\subsection{Cohomologie de l'espace affine de dimension arbitraire} \label{dimgenerale}
Nous allons maintenant dÈmontrer le thÈorËme \ref{cohespaffine} sans restriction sur la dimension. L'idÈe est d'exploiter la suite exacte de faisceaux pro-Ètales\footnote{Cette idÈe (en remplaÁant $\mathbb{B}$ par $\mathbb{B}_{\rm cris}$) nous a ÈtÈ suggÈrÈe par Gabriel Dospinescu.} :
\begin{eqnarray} 0 \to \qp \to \mathbb{B}[1/t]^{\varphi=1} \to \mathbb{B}_{\mathrm{dR}}/\mathbb{B}_{\mathrm{dR}}^+ \to 0. \label{fonda} \end{eqnarray}
La dÈfinition des faisceaux de pÈriodes qui apparaissent dans cette suite exacte et quelques unes de leurs propriÈtÈs sont rappelÈes dans l'appendice \ref{appendice} ; le point clÈ est que cohomologie de ces faisceaux est plus accessible que celle de $\qp$. 
\\

CommenÁons par l'analyse de la cohomologie de $\mathbb{B}_{\mathrm{dR}}/\mathbb{B}_{\mathrm{dR}}^+$.

Rappelons tout d'abord quelques rÈsultats sur la cohomologie des variÈtÈs rigides lisses Stein (\cite{GK1}). 
\begin{proposition} \label{stein}
Soit $X$ un espace Stein lisse sur un corps $p$-adique $K$, de dimension $d$. Les groupes de cohomologie de de Rham de $X$ (au sens de \cite{GK1}) sont les groupes de cohomologie du complexe
\[ \O(X) \to \Omega^1(X) \to \dots \to \Omega^d(X). \]
Les diffÈrentielles sont strictes et les $H_{\rm dR}^i(X)$ ont donc une topologie naturelle, qui en fait des $K$-espaces de FrÈchet. La topologie du dual topologique de ces $K$-espaces de FrÈchet est la topologie localement convexe la plus fine.
\end{proposition}
La derniËre assertion traduit le fait que la cohomologie de de Rham des affinoÔdes surconvergents est de dimension finie. 

\begin{lemme} \label{commutetenseur}
Soit $K$ un corps $p$-adique (c'est-‡-dire une extension finie de $\qp$), $W$ un $K$-espace de Banach et $\mathcal{C}^{\bullet}$ un complexe strict de $K$-espaces de FrÈchet. Pour tout $i$, 
\[ H^i(\mathcal{C}^{\bullet})  \widehat{\otimes}_K W \simeq H^i(\mathcal{C}^{\bullet}  \widehat{\otimes}_K W). \]
\end{lemme}
\begin{proof}
On sait qu'il existe un ensemble $I$ tel que $W\simeq \ell_0(I,K)$, car $K$ est de valuation discrËte. Comme $\mathcal{C}^{\bullet}$ un complexe de FrÈchets, d'aprËs \cite[\S 17]{schneider}, on a pour tout $i$
\[ \mathcal{C}^i \widehat{\otimes}_K W \simeq \ell_0(I,\mathcal{C}^i). \]
On en dÈduit immÈdiatement que pour tout $i$
\[ H^i(\mathcal{C}^{\bullet} ~ \widehat{\otimes}_K W) \simeq \ell_0(I, H^i(\mathcal{C}^{\bullet})), \]
d'o˘ le lemme.
\end{proof}

On en dÈduit la
\begin{proposition} \label{basechange}
Soit $X$ une variÈtÈ rigide lisse et Stein sur un corps $p$-adique $K$. Pour toute $K$-algËbre de Banach $W$ et tout $i$,
\[ H_{\rm dR}^i(X_W) \simeq H_{\rm dR}^i(X) \widehat{\otimes}_K W. \]
\end{proposition}
\begin{proof}
La proposition \ref{stein} permet d'appliquer le lemme aux complexes des sections globales du complexe de de Rham de $X$ et ‡ $W$. \end{proof}

\begin{proposition} \label{poincare}
Soit $X$ une variÈtÈ rigide lisse sur un corps $p$-adique $K$, de dimension $n$. On a une suite exacte de faisceaux pro-Ètales sur $X$ :
\[ 0 \to \mathbb{B}_{\mathrm{dR},X} \to \O \mathbb{B}_{\mathrm{dR},X} \to \O \mathbb{B}_{\mathrm{dR},X} \otimes_{\O_{X}} \Omega_{X}^1 \to \dots \to \O \mathbb{B}_{\mathrm{dR},X} \otimes_{\O_{X}}  \Omega_{X}^n \to 0. \]
ainsi que pour tout $r \in \z$ une suite exacte de faisceaux pro-Ètales sur $X$ :
\[ 0 \to \mathrm{Fil}^r \mathbb{B}_{\mathrm{dR},X} \to \mathrm{Fil}^r \O \mathbb{B}_{\mathrm{dR},X} \to \mathrm{Fil}^{r-1} \O \mathbb{B}_{\mathrm{dR},X} \otimes_{\O_{X}} \Omega_{X}^1 \to \dots \to \mathrm{Fil}^{r-n} \O \mathbb{B}_{\mathrm{dR},X} \otimes_{\O_X}  \Omega_{X}^n \to 0. \]
\end{proposition}
\begin{proof}
Voir \cite[Cor. 6.13]{Shodge}. Nous attirons l'attention du lecteur sur le fait que la notation $\O_X$, $\Omega_X^i$ dÈsigne ici les faisceaux pro-Ètales $\nu^* \O_X$, $\nu^* \Omega_X^i$, $\nu$ Ètant le morphisme de sites de $X_{\mathrm{pro\acute{e}t}}$ vers $X_{\mathrm{\acute{e}t}}$.
\end{proof}

Nous aurons Ègalement besoin du fait suivant. On dispose de morphisme de topos :
\[ \widetilde{X}_{\mathrm{pro\acute{e}t}} \overset{\nu'}\longrightarrow \widetilde{X}_{C,\mathrm{\acute{e}t}} \overset{\lambda}\longrightarrow \widetilde{X}_{\mathrm{\acute{e}t}},  \]
dont la composÈe est le morphisme de topos $\nu : \widetilde{X}_{\mathrm{pro\acute{e}t}} \to \widetilde{X}_{\mathrm{\acute{e}t}}$. Le morphisme $\lambda$ est induit par le morphisme de sites $X_{C,\mathrm{\acute{e}t}} \to X_{\mathrm{\acute{e}t}}$ tel que si $U \to X$ est Ètale, $\lambda^*(U)=U_C$. Le thÈorËme d'Elkik (\cite{elkik}) donne une Èquivalence :
\[ \underset{K'|K}\varprojlim \widetilde{X}_{K',\mathrm{\acute{e}t}} \simeq \widetilde{X}_{C,\mathrm{\acute{e}t}}, \]
$K'$ parcourant les extensions de degrÈ fini de $K$.
%Soit $U\to X_C$ Ètale quasi-compact. Il existe alors une extension finie $K'$ de $K$ et $V \to X_{K'}$ Ètale, tel que $U= V \widehat{\otimes}_{K'} C$. Si $\mathcal{F}$ est un faisceau Ètale sur $X$, on a alors par dÈfinition
%\[ (\lambda^{-1} \mathcal{F})(Y) = \underset{K''|K'} \varinjlim \mathcal{F}(V \otimes_{K'} K''), \]
%$K''$ parcourant les extensions de degrÈ fini de $K'$. 
Si $i<j$ sont deux entiers et $\mathcal{F}$ un fibrÈ vectoriel sur $X$, $\mathcal{F} \widehat{\otimes}_K \mathrm{Fil}^i B_{\rm dR}/\mathrm{Fil}^j B_{\rm dR}$\footnote{Soit $\mathcal{G}$ un faisceau cohÈrent sur un espace rigide sur $K$ et $W$ un $K$-espace de Banach. Le faisceau $\mathcal{G} \widehat{\otimes}_K W$ est le faisceau dont les sections sur un ouvert quasi-compact $U$ sont $\mathcal{G}(U) \widehat{\otimes}_K W$ (noter que $\mathcal{G}(U)$ a une structure naturelle de $K$-espace de Banach). Il s'agit bien d'un faisceau, puisque $\hat{\otimes}_K W$ prÈserve les suites exactes, cf. le lemme \ref{commutetenseur}.} est un faisceau de $\O_{X_{K'}}$-modules pour toute extension finie $K'$ de $K$, car $B_{\rm dR}$ est une $\bar{K}$-algËbre. On peut donc par l'Èquivalence d'Elkik voir $\mathcal{F} \widehat{\otimes}_K \mathrm{Fil}^i B_{\rm dR}/\mathrm{Fil}^j B_{\rm dR}$ comme un faisceau Ètale sur $X_C$, que l'on notera encore $\mathcal{F} \widehat{\otimes}_K \mathrm{Fil}^i B_{\rm dR}/\mathrm{Fil}^j B_{\rm dR}$ pour simplifier, en particulier dans l'ÈnoncÈ suivant.

\begin{proposition} \label{rmuÈtoile}
Soit $\mathcal{F}$ un fibrÈ vectoriel sur $X$. Alors, pour tout $i<j$,
\[ R\nu'_* (\mathrm{Fil}^i \O\mathbb{B}_{\mathrm{dR},X}/\mathrm{Fil}^j \O\mathbb{B}_{\mathrm{dR},X} \otimes_{\O_X} \mathcal{F}) = \mathcal{F} \widehat{\otimes}_K \mathrm{Fil}^i B_{\mathrm{dR}}/ \mathrm{Fil}^j B_{\rm dR}. \]
\end{proposition}
\begin{proof}
Nous affirmons que la flËche naturelle
\[ \mathcal{F} \widehat{\otimes}_K \mathrm{Fil}^i B_{\mathrm{dR}}/ \mathrm{Fil}^j B_{\rm dR} \to R\nu'_*(\mathrm{Fil}^i \O\mathbb{B}_{\mathrm{dR},X}/\mathrm{Fil}^j \O\mathbb{B}_{\mathrm{dR},X} \otimes_{\O_X} \mathcal{F}), \]
est un quasi-isomorphisme filtrÈ, pour les filtrations naturelles des deux cÙtÈs. Il suffit de le tester sur les graduÈs.

On a, par la formule de projection, pour tout $k$ :
\[ R \nu'_* (\mathrm{gr}^k \O\mathbb{B}_{\mathrm{dR},X} \otimes_{\O_X} \mathcal{F}) = R \nu'_*  \mathrm{gr}^k \O\mathbb{B}_{\mathrm{dR},X} \otimes_{\lambda^{-1} \O_X} \lambda^{-1} \mathcal{F}. \]
Or, d'aprËs \cite[Prop. 6.16 (i)]{Shodge},
\[ R \nu'_* \mathrm{gr}^k \O\mathbb{B}_{\mathrm{dR},X} = \nu'_* \mathrm{gr}^k \O\mathbb{B}_{\mathrm{dR},X} = \O_{X_C}(k). \]
DÈcrivons le faisceau $\lambda^{-1} \mathcal{F} \otimes_{\lambda^{-1} \O_X} \O_{X_C}$. Soit $U$ un ouvert Ètale de $X_C$. Quitte ‡ rÈtrÈcir $U$, on peut supposer que $U\to X_C$ provient par extension des scalaires d'un ouvert Ètale $V\to X_{K'}$, avec $K'/K$ finie. Alors on a
\[ \lambda^{-1} \mathcal{F} (U) = \underset{K''/K' ~ \mathrm{finie}} \varinjlim \mathcal{F}(V_{K"}). \]
D'o˘ :
\[ (\lambda^{-1} \mathcal{F} \otimes_{\lambda^{-1} \O_X} \O_{X_C})(U) = \underset{K''/K ~ \mathrm{finie}} \varinjlim \mathcal{F}(V_{K''}) \otimes_{V_{K''}} U = \mathcal{F}_C(U). \]
Par consÈquent, on a bien
\[ R \nu'_* (\mathrm{gr}^k \O\mathbb{B}_{\mathrm{dR},X} \otimes_{\O_X} \mathcal{F}) \simeq \lambda^{-1} \mathcal{F} \widehat{\otimes}_K C(k). \]
\end{proof}

\begin{proposition} \label{bdrplusstein}
Soit $n\geq 1$. On a $H^0(\mathbf{A}_C^n,\mathbb{B}_{\rm dR}^+)= B_{\rm dR}^+$ et pour tout $i>0$, 
\[ H^i(\mathbf{A}_C^n, \mathbb{B}_{\rm dR}^+) = \mathrm{Ker}(d_i), \]
avec les notations du thÈorËme \ref{cohespaffine}. 
\end{proposition}
\begin{proof}
Le dÈbut de la dÈmonstration s'applique ‡ n'importe quel espace rigide lisse Stein de dimension $n$ dÈfini sur une extension finie $K$ de $\qp$. Soit $k\geq 1$, et $i>0$. ConsidÈrons la suite spectrale de complexes filtrÈs
\[ E_1^{p,i-p} = H^{i}(X_C, \mathrm{gr}^p(\mathbb{B}_{\rm dR}^+/t^k)) \Longrightarrow H^i(X_C, \mathbb{B}_{\rm dR}^+/t^k). \]
On va supposer $k>n$ et $k>2$, ce qui est loisible, puisque l'on prendra ‡ la fin la limite sur $k \to +\infty$. Calculons les termes de la premiËre page. Si $p<0$ ou $p\geq k$, $E_1^{p,i-p}=0$. Sinon,
\[ \mathrm{gr}^p(\mathbb{B}_{\rm dR}^+/t^k) = \widehat{\O}_X(p). \]
Or, le lemme de PoincarÈ \ref{poincare} donne une rÈsolution :
\[ 0 \to \widehat{\O}_X \to \mathrm{gr}^0 \O \mathbb{B}_{\rm dR} \to \mathrm{gr}^{-1} \O \mathbb{B}_{\rm dR} \otimes_{\O_X} \Omega_X^1 \to \dots \to \mathrm{gr}^{-n} \O \mathbb{B}_{\rm dR} \otimes_{\O_X} \Omega_X^n \to 0. \]
Comme on l'a notÈ au cours de la preuve de la proposition \ref{bdrstein},
\[ R\nu'_* ( \mathrm{gr}^{-k} \O \mathbb{B}_{\rm dR} \otimes_{\O_X} \Omega_X^k) = \Omega_{X_C}^k(-k). \]
Les diffÈrentielles dans la rÈsolution obtenue en appliquant $R\nu'_*$ ‡ la rÈsolution prÈcÈdente vont donc de $\Omega_{X_C}^k(-k)$ vers $\Omega_{X_C}^{k+1}(-k-1)$ et sont $C$-linÈaires et compatibles ‡ l'action de Galois, donc forcÈment nulles. On en dÈduit que pour tout $p\geq 0$,
\[ R\nu'_* \widehat{\O}_X(p) = \bigoplus_{k=0}^n \Omega_{X_C}^k(p-k)[-k] \]
et donc que si $0 \leq p <k$,
\[ E_1^{p,i-p} = H^i(X_{C,\mathrm{\acute{e}t}}, R\nu'_* \widehat{\O}_X(p)) = \Omega^i(X_C)(p-i), \]
puisque $X_C$ est Stein.

La diffÈrentielle $d_1$ envoie $E_1^{p,q}= \Omega^{p+q}(X_C)(-q)$ vers $E_1^{p+1,q}= \Omega^{p+1+q}(X_C)(-q)$ et s'identifie ‡ $d_{p+q}(-q)$. Comme on a : 
\[ E_2^{p,i-p} =  \mathrm{Ker} (E_1^{p,i-p} \to E_1^{p+1,i-p}) / \mathrm{Im} (E_1^{p-1,i-p} \to E_1^{p,i-p}), \]
on en dÈduit que si $0<p<k-1$,
\[ E_2^{p,i-p} = H_{\rm dR}^i(X_C)(p-i) = H_{\rm dR}^i(X)  \widehat{\otimes}_K C(p-i) \]
(la derniËre ÈgalitÈ venant de la proposition \ref{basechange}), tandis que pour $p=0$,
\[ E_2^{0,i} = \mathrm{Ker}(d_i)(-i) \]
et pour $p=k-1$,
\[ E_2^{k-1,i-k+1} = \mathrm{Coker}(d_i)(k-i-1). \] 
Prenons maintenant $K=\qp$ et $X=\mathbf{A}_{\qp}^n$. Sa cohomologie de de Rham en degrÈ positif est triviale. 

Montrons par rÈcurrence sur $r\geq 2$ que $E_2^{p,i-p} = E_r^{p,i-p}$ pour tout $p$ et pour tout $i>0$. Supposons le rÈsultat connu pour un $r\geq 2$. Il suffit pour obtenir le rÈsultat pour $r+1$ de montrer que toutes les diffÈrentielles ‡ la $r$-Ëme page sont nulles. Alors $d_r$ envoie $E_r^{p,i-p}=E_2^{p,i-p}$ vers $E_r^{p+r,i-p-r+1}=E_2^{p+r,i-p-r+1}$. C'est donc par les calculs prÈcÈdents la flËche nulle, sauf Èventuellement si $p=0$ et $r=k-1$ ; dans ce cas, c'est une flËche de $\mathrm{Ker}(d_i)(-i)$ vers $\mathrm{Coker}(d_{i+1})(k-i-2)$. Comme cette flËche est compatible ‡ l'action de Galois, si elle Ètait non nulle, on aurait $-i=k-i-2$, i.e. $k=2$, ce qu'on a exclu. On en dÈduit finalement que la suite spectrale dÈgÈnËre ‡ la deuxiËme page. Pour tout $i>0$, on a donc une extension :
\[ 0 \to \mathrm{Coker}(d_i)(k-i-1) \to H^i(\mathbf{A}_C^n,\mathbb{B}_{\rm dR}^+/t^k) \to \mathrm{Ker}(d_i)(-i) \to 0. \]
Quand on passe de $k$ ‡ $k+1$, la flËche $H^i(\mathbf{A}_C^n,\mathbb{B}_{\rm dR}^+/t^{k+1}) \to H^i(\mathbf{A}_C^n,\mathbb{B}_{\rm dR}^+/t^k)$ envoie le sous-espace $\mathrm{Coker}(d_i)(k-i)$ sur zÈro, par Galois Èquivariance. Comme d'aprËs la proposition \ref{replete}, 
\[  \mathbb{B}_{\rm dR}^+ = R ~ \underset{k}\varprojlim \mathbb{B}_{\rm dR}^+/t^k, \]
on a en dÈfinitive, si $i>0$ et que l'on oublie l'action de Galois :
\[ H^i(\mathbf{A}_C^n,\mathbb{B}_{\rm dR}^+) = \mathrm{Ker}(d_i). \]
Il reste ‡ calculer la cohomologie en degrÈ $0$. La rÈsolution ci-dessus et la proposition \ref{rmuÈtoile} donnent que pour tout $k\geq 0$,
\[ H^0(\mathbf{A}_C^n,\mathbb{B}_{\rm dR}^+/t^k) = \mathrm{Ker} \left( \O(\mathbf{A}_{\qp}^n) \widehat{\otimes}_{\qp} B_{\rm dR}^+/t^k \to \Omega^1(\mathbf{A}_{\qp}^n) \widehat{\otimes}_{\qp} t^{-1} B_{\rm dR}^+/t^{k-1} B_{\rm dR}^+ \right). \]
C'est donc une extension
\[ 0 \to B_{\rm dR}^+/t^k \to H^0(\mathbf{A}_C^n,\mathbb{B}_{\rm dR}^+/t^k) \to \O(\mathbf{A}_{\qp}^n)/\qp \widehat{\otimes}_{\qp} C(k-1) \to 0. \]
Quand on passe de $k$ ‡ $k+1$, la flËche naturelle de $H^0(\mathbf{A}_C^n,\mathbb{B}_{\rm dR}^+/t^{k+1})$ vers $H^0(\mathbf{A}_C^n,\mathbb{B}_{\rm dR}^+/t^k)$ est la flËche Èvidente sur le terme de gauche et le morphisme nul sur celui de droite. En prenant la limite inverse sur $k$, on rÈcupËre donc finalement :
\[ H^0(\mathbf{A}_C^n,\mathbb{B}_{\rm dR}^+) = B_{\rm dR}^+. \]
\end{proof}

\begin{remarque} \label{rnuprimebdrplus}
PlutÙt que d'utiliser la suite spectrale d'un complexe filtrÈ, on pourrait utiliser le quasi-isomorphisme, valable pour toute variÈtÈ rigide lisse sur une extension finie $K$ de $\qp$, de dimension $n$ :
\[ R\nu'_* \mathbb{B}_{\mathrm{dR},X}^+ = \O_X \widehat{\otimes}_K B_{\rm dR}^+ \to \Omega_X^1 \widehat{\otimes}_K t^{-1} B_{\rm dR}^+ \to \dots \to \Omega_X^n \widehat{\otimes}_K t^{-n}B_{\rm dR}^+ \]
(la diffÈrentielle du complexe de droite Ètant donnÈe par la diffÈrentielle du complexe de de Rham de $X$), qui se montre en reprenant les arguments de la proposition \ref{rmuÈtoile} et dont on reparlera plus bas (remarque \ref{lecasdeladim1}). 
\end{remarque}

\begin{proposition}\label{bdrstein}
Soit $n\geq 1$. On a $H^0(\mathbf{A}_C^n,\mathbb{B}_{\rm dR}) = B_{\rm dR}$ et pour tout $i>0$, $H^i(\mathbf{A}_C^n,\mathbb{B}_{\rm dR}) =0$. 
\end{proposition}
\begin{proof}
Pour $k\geq 1$, notons $\bar{U}_k$ la boule fermÈe de rayon $p^k$ et $U_k$ la boule ouverte de rayon $p^k$. Observons que 
\[ R\Gamma(\mathbf{A}_C^n, \mathbb{B}_{\rm dR}) = R ~ \underset{k} \varprojlim ~ R\Gamma(\bar{U}_k, \mathbb{B}_{dR}) = R ~ \underset{k} \varprojlim (R\Gamma(\bar{U}_k, \mathbb{B}_{dR}^+)[1/t]), \]
par quasi-compacitÈ de $\bar{U}_k$. Pour tout $k$, la flËche 
\[ R\Gamma(\bar{U}_{k+1}, \mathbb{B}_{dR}^+)[1/t] \to R\Gamma(\bar{U}_k, \mathbb{B}_{\rm dR}^+)[1/t] \]
se factorise ‡ travers $R\Gamma(U_{k+1},\mathbb{B}_{\rm dR}^+)[1/t]$. Par consÈquent, l'isomorphisme prÈcÈdent peut se rÈÈcrire :
\[ R\Gamma(\mathbf{A}_C^n, \mathbb{B}_{\rm dR}) = R ~ \underset{k} \varprojlim (R\Gamma(U_k, \mathbb{B}_{dR}^+)[1/t]). \]
Or la proposition \ref{bdrplusstein}, ou plutÙt sa preuve (qui s'adapte au cas du disque ouvert, puisque sa cohomologie de de Rham en degrÈ strictement positif est elle aussi triviale), montre que pour tout $i>0$ et tout $k>0$, $H^i(U_k, \mathbb{B}_{\rm dR}^+)$ est annulÈ par $t$ et que $H^0(U_k,\mathbb{B}_{\rm dR}^+)= B_{\rm dR}^+$. On en dÈduit l'ÈnoncÈ cherchÈ.
\end{proof}

D'o˘ finalement :
\begin{proposition} \label{bdraffine}
Le groupe $H^0(\mathbf{A}_C^n,\mathbb{B}_{\rm dR}/\mathbb{B}_{\rm dR}^+)$ est une extension de $\mathrm{Ker}(d_1)=\O(\mathbf{A}_C^n)/C$ par $B_{\rm dR}/B_{\rm dR}^+$ et pour tout $i>0$, $H^i(\mathbf{A}_C^n,\mathbb{B}_{\rm dR}/\mathbb{B}_{\rm dR}^+)= \mathrm{Ker}(d_{i+1})=\mathrm{Im}(d_i)$. 
\end{proposition}

\begin{remarque} \label{lecasdeladim1}
La mÈthode utilisÈe permettrait de dÈcrire plus gÈnÈralement la cohomologie de $\mathbb{B}_{\mathrm{dR},X}^+$ et $\mathbb{B}_{\mathrm{dR},X}$ pour $X$ un espace Stein lisse sur un corps $p$-adique. La cohomologie de $\mathbb{B}_{\rm dR}^+$ se calcule ‡ l'aide du quasi-isomorphisme :
\[ R\nu'_* \mathbb{B}_{\mathrm{dR},X}^+ = \O_X \widehat{\otimes}_K B_{\rm dR}^+ \to \Omega_X^1 \widehat{\otimes}_K t^{-1} B_{\rm dR}^+ \to \dots \to \Omega_X^n \widehat{\otimes}_K t^{-n}B_{\rm dR}^+, \]
ÈvoquÈ dans la remarque \ref{rnuprimebdrplus}. On dÈduit de cet isomorphisme l'existence d'un triangle distinguÈ :
\[ \Omega_X^{\bullet} \widehat{\otimes}_K B_{\rm dR}^+ \to R\nu'_* \mathbb{B}_{\mathrm{dR},X}^+ \to (0 \to \Omega_X^1 \widehat{\otimes}_K t^{-1} B_{\rm dR}^+/t \to \dots \to \Omega_X^n \widehat{\otimes}_K t^{-n}B_{\rm dR}^+/t). \]
Le complexe de droite se dÈvisse lui-mÍme ‡ nouveau comme extension de
\[ 0 \to 0 \to \Omega_X^2 \widehat{\otimes}_K t^{-2} B_{\rm dR}^+/t^{-1} \to \dots \to \Omega_X^n \widehat{\otimes}_K t^{-n}B_{\rm dR}^+/t^{-1} \]
par $\Omega_X^{\bullet} \widehat{\otimes}_K C(-1)$. L'hypercohomologie de ces complexes se calcule facilement ‡ l'aide de la proposition \ref{basechange}. Pour le faisceau $\mathbb{B}_{\rm dR}$, sa cohomologie devrait pouvoir se calculer comme dans la preuve de la proposition \ref{bdrstein}, en choisissant un recouvrement affinoÔde admissible $(U_k)_k$ de $X$ et une prÈsentation surconvergente de chaque $U_k$ (l'existence d'une telle prÈsentation est garantie par \cite[Th. 7]{elkik}) : de cette faÁon, on peut prÈsenter $X$ comme limite inverse d'espaces Stein $(V_k)$, avec pour tout $k$, $V_k \subset U_k$. 

Voici le rÈsultat que l'on obtient pour $X$ Stein lisse de dimension $1$, sous l'hypothËse additionnelle que $X$ est connexe. On a :
\[ H^0(X_C,\mathbb{B}_{\rm dR}^+)=B_{\rm dR}^+ ~ ; ~ H^i(X_C,\mathbb{B}_{\rm dR}^+)=0, ~ \mathrm{si}~ i>1 \]
et une extension
\[ 0 \to H_{\rm dR}^1(X) \widehat{\otimes}_K B_{\rm dR}^+ \to H^1(X_C,\mathbb{B}_{\rm dR}^+) \to \Omega^1(X) \widehat{\otimes}_K t^{-1} B_{\rm dR}^+/B_{\rm dR}^+ \to 0. \]
(on utilise le fait que $H_{\rm dR}^i(X)=0$ si $i>2$, et aussi pour $i=2$ puisque par dualitÈ de PoincarÈ $H_{\rm dR}^2(X)=H_{\mathrm{dR},c}^0(X)^*=0$). On a aussi :
\[ H^0(X_C,\mathbb{B}_{\rm dR})=H_{\rm dR}^1(X) \widehat{\otimes}_K B_{\rm dR} ~ ; ~ H^i(X_C,\mathbb{B}_{\rm dR})=0, ~ \mathrm{si}~ i>1 \]
Via ces identifications, la flËche naturelle $H^1(X_C,\mathbb{B}_{\rm dR}^+) \to H^1(X_C,\mathbb{B}_{\rm dR})$ se dÈcrit comme suit : sur le sous-espace $H_{\rm dR}^1(X) \widehat{\otimes}_K B_{\rm dR}^+$, c'est la flËche Èvidente $H_{\rm dR}^1(X) \widehat{\otimes}_K B_{\rm dR}^+ \to H_{\rm dR}^1(X) \widehat{\otimes}_K B_{\rm dR}$ ; sur le quotient $\Omega^1(X) \widehat{\otimes}_K t^{-1} B_{\rm dR}^+/B_{\rm dR}^+$, c'est la composÈe de la projection $\Omega^1(X) \widehat{\otimes}_K t^{-1} B_{\rm dR}^+/B_{\rm dR}^+ \to H_{\rm dR}^1(X) \widehat{\otimes}_K  t^{-1} B_{\rm dR}^+/B_{\rm dR}^+$ avec l'inclusion $H_{\rm dR}^1(X) \widehat{\otimes}_K  B_{\rm dR}/B_{\rm dR}^+$.

Par consÈquent, le groupe $H^0(X_C,\mathbb{B}_{\rm dR}/\mathbb{B}_{\rm dR}^+)$ est une extension de $\O(X_C)/C$ par $B_{\rm dR}/B_{\rm dR}^+$. On a $H^1(X_C,\mathbb{B}_{\rm dR}/\mathbb{B}_{\rm dR}^+) = H_{\rm dR}^1(X) \widehat{\otimes}_K B_{\rm dR}/t^{-1}B_{\rm dR}^+$ et $H^i(X_C,\mathbb{B}_{\rm dR}/\mathbb{B}_{\rm dR}^+)=0$ si $i>1$. 
\end{remarque}

L'analyse de la cohomologie de $\mathbb{B}[1/t]^{\varphi=1}$ est plus subtile. On va prouver le rÈsultat suivant.

\begin{proposition} \label{bcrisaffine}
Soit $n\geq 1$. On note $D^n$ le disque unitÈ ouvert de dimension $n$. Pour tout $i>0$,
\[ H^i(D_C^n, \mathbb{B}[1/t]^{\varphi=1}) = 0. \]
De mÍme, pour tout $i>0$,
\[ H^i(\mathbf{A}_C^n, \mathbb{B}[1/t]^{\varphi=1}) = 0. \]
\end{proposition}

Dans tout ce paragraphe, $I$ dÈsigne un sous-intervalle compact de $]0,1[$ ‡ extrÈmitÈs des nombres rationnels. Nous allons commencer par dÈcrire la cohomologie du faisceau $\mathbb{B}_I[1/t]$. Modulo une hypothËse formulÈe ci-dessous (preuve de la proposition \ref{couronne}) et appelÈe $(*)$\footnote{Et que nous avons depuis d\'emontr\'ee : cf. \cite[Prop. 3.11]{ACLB}.}, nous obtiendrons des rÈsultats plus fins, sans inverser $t$. Ces rÈsultats conditionnels n'interviennent pas dans la dÈmonstration de la proposition \ref{bcrisaffine} ; nous les mentionnons seulement car ils mettent en Èvidence l'importance du foncteur dÈcalage $L\eta_t$. 
\\

Rappelons briËvement pour commencer ce que sont les foncteurs dÈcalage de Berthelot-Ogus. Soit $A$ un anneau et $f\in A$, non diviseur de zÈro. Les foncteurs dÈcalage ont ÈtÈ introduits pour la premiËre fois dans \cite{BO}, sur une suggestion de Deligne.

\begin{definition}
Soit $\delta : \z \to \mathbf{N}$. Si $K^{\bullet}$ est un complexe de $A$-modules tel que $K^i$ est sans $f$-torsion pour tout $i$, on dÈfinit un nouveau complexe $\eta_{\delta,f} K^{\bullet}$ par
\[ (\eta_{\delta,f} K^{\bullet})^j = \{ x \in f^{\delta(j)} K^j, dx \in f^{\delta(j+1)} K^{j+1} \}. \]
Si $\delta=\mathrm{Id}$, $\eta_{\mathrm{Id},f} K^{\bullet}$ est simplement $\eta_f K^{\bullet}$.
\end{definition}

\begin{proposition} \label{leta}
Si $\delta$ est croissante, $\eta_{\delta,f}$ transforme les quasi-isomorphismes en des quasi-isomorphismes, donc s'Ètend en un foncteur notÈ $L\eta_{\delta,f}$ entre catÈgories dÈrivÈes.

%Si $f$ et $g$ sont dans $A$, non diviseurs de zÈro, $g$ divisant $f$, et $\delta, \epsilon : \z \to \mathbf{N}$ deux fonctions avec $\delta$ et $\epsilon +\delta$ croissantes, la composÈe $\eta_{\epsilon,g} \eta_{\delta,f}$ transforme les quasi-isomorphismes en des quasi-isomorphismes, donc s'Ètend en un foncteur notÈ $L\eta_{\epsilon,g} \eta_{\delta,f}$ entre catÈgories dÈrivÈes.
\end{proposition}
\begin{proof}
Voir \cite[Prop. 8.19]{BO}.
%On laisse au lecteur le soin de vÈrifier que la mÍme preuve fonctionne exactement de la mÍme maniËre dans le second cas.
\end{proof}
On aura besoin de la proposition suivante, qui explicite l'action des foncteurs dÈcalage sur les complexes de Koszul.
\begin{proposition} \label{koszul}
Soit $A$ un anneau, $g_1,\dots,g_n \in A$ et $f\in A$ non diviseurs de zÈro. Soit $M$ un $A$-module sans $f$-torsion. Si $h_1,\dots, h_d$ sont des endomorphismes de $M$ qui commutent, on note $K_M(h_1,\dots,h_d)$ le complexe de Koszul
\[ M \to \bigoplus_{1 \leq i \leq d} M \to \bigoplus_{1 \leq i_1<i_2 \leq d} M \to \dots \to \bigoplus_{1\leq i_1<\dots < i_k \leq d} M \to \dots \]
o˘ la diffÈrentielle de $M$ en position $i_1< \dots < i_k$ vers $M$ en position $j_1<\dots, < j_{k+1}$ est non nulle seulement si $\{i_1,\dots,i_k \} \subset \{j_1, \dots,j_{k+1}\}$ et vaut dans ce cas $(-1)^{m-1} h_m$, $m$ Ètant l'unique indice entre $1$ et $k+1$ tel que $j_m \notin \{i_1,\dots,i_k\}$. 

$\mathrm{(i)}$ Si $f$ divise tous les $g_i$,
\[ \eta_f K_M(g_1,\dots,g_n) = K_M(g_1/f,\dots,g_n/f). \]

$\mathrm{(ii)}$ S'il existe un $i$ tel que $g_i$ divise $f$, $\eta_f K_M(g_1,\dots,g_n)$ est acyclique.

\end{proposition}
\begin{proof}
Voir \cite[Lem. 7.9]{BMS}.
\end{proof}

Ces dÈfinitions s'Ètendent ‡ un cadre plus gÈnÈral (\cite[\S 6.1]{BMS}) : si $(T,\O_T)$ est un topos annelÈ, $f\in \O_T$ engendrant un idÈal inversible et $K^{\bullet}$ un complexe de $\O_T$-modules sans $f$-torsion, on peut dÈfinir $\eta_f K^{\bullet}$ comme prÈcÈdemment et on montre que le foncteur $\eta_f$ s'Ètend ‡ la catÈgorie dÈrivÈe $D(\O_T)$ des $\O_T$-modules.

Soit $Z$ un espace rigide sur $C$. SpÈcialisant au cas o˘ $T$ est le topos des faisceaux Ètales de $B_I$-modules sur $Z$, $\O_T=B_I$ et $f=t$, on obtient un complexe $L\eta_t R\nu'_* \mathbb{B}_{I,Z}$ dans la catÈgorie dÈrivÈe des faisceaux Ètales de $B_I$-modules sur $Z$. 
\\

Avant de calculer la cohomologie du disque ouvert ou de l'espace affine, traitons le cas d'une couronne ouverte. Soit $r, r'$ deux nombres rationnels. Notons 
\[ \bar{\mathcal{C}}_{r,r'} = \mathrm{Spa}(\qp \langle p^rT_1,\dots,p^{r}T_n,p^{r'}T_1^{-1},\dots, p^{r'}T_n^{-1} \rangle) \]
%\[ X_r = \mathrm{Spa}(C\langle T_1,\dots,T_n, X_1,\dots, X_n, Y_1, \dots, Y_n \rangle /(X_1-p^r T_1,\dots, X_n-p^rT_n, p^{-r} T_1Y_1-1,\dots, p^{-r}T_nY_n-1)) \]
la couronne fermÈe de rayons $p^{-r'}$ et $p^r$. On notera simplement $\underline{T}$ pour $T_1,\dots,T_n$ et de mÍme pour les autres variables qui apparaissent. On note $\mathcal{C}_{r,r'}$ la couronne ouverte de rayons $p^{-r'}$ et $p^r$ ; c'est une variÈtÈ Stein.

\begin{lemme} \label{birinfty}
Notons
\[ (R_{\infty},R_{\infty}^+) = (C\langle (p^r\underline{T})^{1/p^{\infty}},(p^{r'}\underline{T}^{-1})^{1/p^{\infty}} \rangle, \O_C\langle (p^r\underline{T})^{1/p^{\infty}},(p^{r'}\underline{T}^{-1})^{1/p^{\infty}} \rangle). \]
Alors
\[ B_I(R_{\infty},R_{\infty}^+) = B_I \langle ([p^{\flat}]^r\underline{X})^{1/p^{\infty}}, ([p^{\flat}]^{r'}\underline{X}^{-1})^{1/p^{\infty}} \rangle, \]
avec pour tout $i$, $X_i = [T_i]$.
\end{lemme}
\begin{proof}
Pour toute algËbre affinoÔde perfectoÔde $(S,S^+)$, on a par dÈfinition
\[ B_I(S,S^+)= W(S^+) \left \langle \frac{[\alpha]}{p}, \frac{p}{[\beta]} \right \rangle, \]
si $I=[a,b]$ et $|\alpha|=a$, $|\beta|=b$. On a donc : 
\[ B_I(R_{\infty},R_{\infty}^+)= W(R_{\infty}^+) \left \langle \frac{[\alpha]}{p}, \frac{p}{[\beta]} \right \rangle. \]
Or 
\[ W(R_{\infty}^+) = A_{\rm inf} \langle ([p^{\flat}]^r\underline{X})^{1/p^{\infty}}, ([p^{\flat}]^{r'}\underline{X}^{-1})^{1/p^{\infty}} \rangle, \]
comme on le voit immÈdiatement en utilisant l'adjonction entre vecteurs de Witt et basculement. D'o˘
\[ B_I(R_{\infty},R_{\infty}^+) = A_{\rm inf} \langle ([p^{\flat}]^r\underline{X})^{1/p^{\infty}}, ([p^{\flat}]^{r'}\underline{X}^{-1})^{1/p^{\infty}} \rangle \left \langle \frac{[\alpha]}{p}, \frac{p}{[\beta]} \right \rangle = B_I \langle ([p^{\flat}]^r\underline{X})^{1/p^{\infty}}, ([p^{\flat}]^{r'}\underline{X}^{-1})^{1/p^{\infty}} \rangle, \]
la derniËre ÈgalitÈ Ètant obtenue en prenant $(S,S^+)=(C,\O_C)$ dans la formule ci-dessus.
\end{proof}

\begin{proposition}  \label{couronne}
On a des quasi-isomorphismes : 
\[ R\Gamma(\bar{\mathcal{C}}_{r,r',C}, \mathbb{B}_I[1/t])  \simeq \Omega^{\bullet}(\bar{\mathcal{C}}_{r,r'}) ~ \widehat{\otimes}_{\qp} B_I[1/t] . \]
En outre, si l'on admet l'hypothËse $(*)$ ci-dessous, on a mÍme
\[ R\Gamma(\bar{\mathcal{C}}_{r,r',C,\mathrm{\acute{e}t}}, L\eta_t R\nu'_* \mathbb{B}_I)  \simeq \Omega^{\bullet}(\bar{\mathcal{C}}_{r,r'}) ~ \widehat{\otimes}_{\qp} B_I. \]
\end{proposition}
\begin{proof}
Tout d'abord, comme $\mathbb{B}_I \simeq \mathbb{B}_{\varphi^k(I)}$ pour tout entier $k\in \z$ et comme (pour la deuxiËme assertion) $\varphi(t)=pt$, avec $p$ inversible dans $B_I$, on peut supposer $I \subset [p^{-1},1[$. On sait alors que $t$ et $[\epsilon]-1$ diffËrent par une unitÈ de $B_I$. 

Notons $R=C\langle p^r\underline{T}, p^{r'}\underline{T}^{-1} \rangle$ et
\[ R_{\infty} = C\langle (p^r\underline{T})^{1/p^{\infty}},(p^{r'}\underline{T}^{-1})^{1/p^{\infty}} \rangle, \]
comme dans le lemme. Notons $\tilde{C}_{r,r',C}= \mathrm{Spa}(R_{\infty},R_{\infty}^+)$. On sait que
\[ H^0(\tilde{C}_{r,r',C},\mathbb{B}_I)= B_I(R_{\infty},R_{\infty}^+) \]
et que 
\[ H^i(\tilde{C}_{r,r',C},\mathbb{B}_I)=0 \]
si $i>0$, d'aprËs la proposition \ref{descriptionbi}. En outre, la proposition \ref{debile} dit que les sections de $\mathbb{B}_I$ sur le produit fibrÈ de $\tilde{\mathcal{C}}_{r,r',C}$ $k$-fois avec lui-mÍme au-dessus de $\bar{\mathcal{C}}_{r,r',C}$ sont
\[ \mathcal{C}^0(\zp^{k-1}, B_I(R_{\infty},R_{\infty}^+)). \]
La suite spectrale de Cartan-Leray pour le recouvrement pro-Ètale $\tilde{C}_{r,r',C} \to \bar{C}_{r,r',C}$ dÈgÈnËre donc et identifie le complexe de cohomologie de $\mathbb{B}_I$ sur $\bar{\mathcal{C}}_{r,r',C}$ au complexe de cohomologie continue du groupe $\zp^n$ agissant sur $B_I(R_{\infty},R_{\infty}^+)$ ; autrement dit, la flËche
\[ R\Gamma_{\rm cont}(\zp^n,B_I(R_{\infty},R_{\infty}^+) ) \to R\Gamma(\bar{C}_{r,r',C}, \mathbb{B}_I) \]
est un quasi-isomorphisme. On a donc Ègalement, en appliquant $L\eta_t$ des deux cÙtÈs, un quasi-isomorphisme
\[  L\eta_t R\Gamma_{\rm cont}(\zp^n,B_I(R_{\infty},R_{\infty}^+) ) \to L\eta_t R\Gamma(\bar{C}_{r,r',C,\mathrm{\acute{e}t}}, R\nu'_* \mathbb{B}_I). \]

Le rÈsultat suivant devrait Ítre vrai mais nous ne l'avons pas dÈmontrÈ\footnote{Il l'est d\'esormais : cf. \cite[Prop. 3.11]{ACLB}.}. Nous le formulons donc comme une hypothËse.
\\

\textit{\textbf{HypothËse $(*)$.}}
Soit $S$ un espace affinoÔde lisse sur $\mathrm{Spa}(C,\O_C)$. La flËche naturelle
\[ L\eta_t R\Gamma(S,\mathbb{B}_I) \to R\Gamma(S_{\mathrm{\acute{e}t}}, L\eta_t R\nu'_* \mathbb{B}_I) \]
est un quasi-isomorphisme.
\\

Si l'on admet cette hypothËse, on a donc un quasi-isomorphisme
\[  L\eta_t R\Gamma_{\rm cont}(\zp^n,B_I(R_{\infty},R_{\infty}^+) ) \to  R\Gamma(\bar{C}_{r,r',C,\mathrm{\acute{e}t}}, L\eta_tR\nu'_* \mathbb{B}_I). \] 

Pour obtenir le premier quasi-isomorphisme de l'ÈnoncÈ de la proposition, il ne reste donc plus qu'‡ calculer le membre de gauche, ce que nous allons faire ‡ l'aide des complexes de Koszul. On note $(\gamma_1,\dots,\gamma_n)$ le systËme de gÈnÈrateurs canonique de $\zp^n$. Le lemme \ref{birinfty} affirme que
\[ B_I(R_{\infty},R_{\infty}^+) = B_I \langle ([p^{\flat}]^r\underline{X})^{1/p^{\infty}}, ([p^{\flat}]^{r'}\underline{X}^{-1})^{1/p^{\infty}} \rangle, \]
avec pour tout $i$, $X_i = [T_i]$. On peut donc dÈcomposer :
\[ B_I(R_{\infty},R_{\infty}^+) = B_I(R_{\infty},R_{\infty}^+)^{\rm int} \oplus B_I(R_{\infty},R_{\infty}^+)^{\rm nonint}, \]
avec $B_I(R_{\infty},R_{\infty}^+)^{\rm int}= B_I \langle [p^{\flat}]^r\underline{X}, [p^{\flat}]^{r'} \underline{X}^{-1}\rangle$ et $B_I(R_{\infty},R_{\infty}^+)^{\rm nonint}$ le sous-$B_I \langle [p^{\flat}]^r\underline{X}, [p^{\flat}]^{r'} \underline{X}^{-1}\rangle$-module complet de $B_I(R_{\infty},R_{\infty}^+)$ engendrÈ par les monÙmes ‡ coefficients non entiers.

Pour tout $i\geq 0$, le groupe de cohomologie $H_{\rm cont}^i(\zp^n,B_I(R_{\infty},R_{\infty}^+)^{\rm nonint})$ est annulÈ par $t$, et donc
\[ L\eta_t R\Gamma_{\rm cont}(\zp^n,B_I(R_{\infty},R_{\infty}^+)^{\rm non int} ) = 0. \]
La preuve de cette assertion est tout ‡ fait analogue ‡ celle de \cite[Lem. 9.6]{BMS} : on montre que la multiplication par $t$ sur $R\Gamma_{\rm cont}(\zp^n,B_I(R_{\infty},R_{\infty}^+)^{\rm nonint})$ est homotope ‡ zÈro. En Ècrivant ce complexe de cohomologie comme complexe de Koszul, on se ramËne ‡ fabriquer l'homotopie pour le complexe
\[ B_I(R_{\infty},R_{\infty}^+)^{\rm int}. ([p^{\flat}]^{\rho_i} X_i)^{a(i)} \prod_{j \neq i} ([p^{\flat}]^{\rho_j}X_j)^{a(j)} \overset{\gamma_i-1}\longrightarrow B_I(R_{\infty},R_{\infty}^+)^{\rm int}. ([p^{\flat}]^{\rho_i} X_i)^{a(i)} \prod_{j \neq i} ([p^{\flat}]^{\rho_j}X_j)^{a(j)}, \]
avec $1\leq i \leq n$, $a(i)=m/p^r$, $r\geq 1$ et $m \in \z \backslash p\z$, qui est quasi-isomorphe au complexe
\[ B_I(R_{\infty},R_{\infty}^+)^{\rm int} \overset{\gamma_i [\epsilon^{m/p^r}]-1}\longrightarrow B_I(R_{\infty},R_{\infty}^+)^{\rm int}. \] 
On conclut pour ce complexe en utilisant le fait que $t$ divise $\gamma_i^{p^{r-1}}-1$ (ici on utilise l'hypothËse sur $I$). Nous renvoyons ‡ loc. cit. pour les dÈtails. 

%Le complexe de Koszul de $B_I(R_{\infty},R_{\infty}^+)^{\rm nonint}$ s'Ècrit comme somme directe complÈtÈe de complexes de la forme
%\[ K_{B_I ([p^{\flat}]^{\rho_i} X_i)^{a(i)} \prod_{j \neq i} ([p^{\flat}]^{\rho_j}X_j)^{a(j)}}(\gamma_1-1,\dots,\gamma_n-1) \]
%avec pour tout $k=1,\dots,n$, $a(k) \in  \mathbf{Z}[1/p]$ et $i$ est choisi tel que $a(i) \notin \z$ soit de valuation $p$-adique minimale parmi les $a(k)$, et $\rho_k=r$ ou $r'$ selon le signe de $a(k)$. Ce complexe est isomorphe au complexe
%\[ K_{B_I }([\epsilon^{a(1)}]-1,\dots,[\epsilon^{a(n)}]-1). \]
%Comme $[\epsilon^{a(i)}]-1=[\epsilon^{\{a(i)\}}]-1$ ($\{x\}$ dÈsigne ici la partie fractionnaire de $x$), cet ÈlÈment divise $t$, et on en dÈduit que le complexe $L\eta_t K_{B_I(R_{\infty},R_{\infty}^+)^{\rm nonint}}(\gamma_1-1,\dots,\gamma_n-1)$ est acyclique, avec la proposition \ref{koszul} (i). 

Au contraire, dans le complexe de Koszul de $B_I(R_{\infty},R_{\infty}^+)^{\rm int}$, toutes les flËches sont divisibles par $t$, donc on a d'aprËs la proposition \ref{koszul} (ii) :
\[ L\eta_t K_{B_I(R_{\infty},R_{\infty}^+)^{\rm int}}(\gamma_1-1,\dots,\gamma_n-1)=K_{B_I \langle [p^{\flat}]^r\underline{X}, [p^{\flat}]^r \underline{X}^{-1}\rangle}(\frac{\gamma_1-1}{t},\dots,\frac{\gamma_n-1}{t}) . \]
On sait qu'‡ une unitÈ de $B_I$ prËs, pour chaque $i$, $\frac{\gamma_i-1}{t}$ et $X_i d/dX_i$ coÔncident (\cite[Lem. 12.3]{BMS}). Le complexe considÈrÈ est donc quasi-isomorphe au complexe 
\[  K_{B_I \langle [p^{\flat}]^r\underline{X},[p^{\flat}]^{r'} \underline{X}^{-1}\rangle} \left(X_1\frac{d}{dX_1},\dots, X_n\frac{d}{dX_n} \right), \]
qui est lui-mÍme quasi-isomorphe au complexe de de Rham
\[ \Omega^{\bullet}(\bar{\mathcal{C}}_{r,r'}) ~ \widehat{\otimes}_{\qp} B_I. \]
en utilisant la base $d\log(X_1), \dots,d\log(X_n)$. Ceci prouve le deuxiËme quasi-isomorphisme de l'ÈnoncÈ, modulo l'hypothËse $(*)$.

Enfin, notons que $R\Gamma(\bar{C}_{r,r',C,\mathrm{\acute{e}t}}, R\nu'_* \mathbb{B}_I)$ et $L\eta_t R\Gamma(\bar{C}_{r,r',C}, R\nu'_* \mathbb{B}_I)$ sont tous deux isomorphes ‡ $R\Gamma(\bar{C}_{r,r',C,\mathrm{\acute{e}t}}, R\nu'_* \mathbb{B}_I[1/t])$ aprËs inversion de $t$ (par quasi-compacitÈ de $\bar{\mathcal{C}}_{r,r'}$) ; par consÈquent le premier quasi-isomorphisme de l'ÈnoncÈ est vrai de faÁon inconditionnelle. 
\end{proof}

\begin{remarque}
La partie de la preuve ci-dessous qui consiste ‡ exprimer la cohomologie de $\mathbb{B}_I$ comme cohomologie d'un complexe de Koszul est standard depuis Faltings, et valable si l'on remplace l'algËbre $R$ qui y apparaÓt par une algËbre affinoÔde \og petite \fg{} au sens de Faltings. NÈanmoins pour une telle algËbre $R$, il ne semble pas Èvident de dÈcrire le complexe de Koszul de $\mathbb{B}_I(R_{\infty})$ ($R_{\infty}$ Ètant l'extension perfectoÔde pro-Ètale de $R$ dÈterminÈe par le choix d'une carte) en termes du complexe de de Rham de $\mathrm{Spa}(R)$. Cela est probablement possible aprËs choix d'un modËle formel convenable de $R$, mais s˚rement dÈlicat ‡ mettre en oeuvre, surtout si l'on veut des quasi-isomorphismes fonctoriels. 

En outre, le complexe de de Rham d'un affinoÔde n'est pas trËs sympathique ; il vaut mieux travailler avec des affinoÔdes surconvergents. Ce point de vue est explor\'e dans \cite{ACLB}.
\end{remarque}

Passons au cas du disque. On note $\bar{D}=\mathrm{Spa}(\qp \langle \underline{T} \rangle)$ le disque unitÈ fermÈ de dimension $n$ et $\tilde{D}_C=\mathrm{Spa}(C \langle \underline{T}^{1/p^{\infty}} \rangle)$ le disque unitÈ perfectoÔde sur $C$.

\begin{proposition} \label{boule}
On a :
\[ R\Gamma(\bar{D}_C, \mathbb{B}_I[1/t]) = \Omega_{\bar{D}}^{\bullet} ~ \widehat{\otimes}_{\qp} B_I[1/t]. \] 
En outre, si l'hypothËse $(*)$ est valide,
\[ R\Gamma(\bar{D}_{C,\mathrm{\acute{e}t}}, L\eta_t R\nu'_* \mathbb{B}_I) = \Omega_{\bar{D}}^{\bullet} ~ \widehat{\otimes}_{\qp} B_I. \]
\end{proposition}

\begin{proof}
Le morphisme $\tilde{D}_C \to \bar{D}_C$ n'est pas un recouvrement pro-Ètale de $\bar{D}_C$, bien s˚r, mais c'en est un recouvrement quasi-pro-Ètale (c'est-‡-dire localement pro-Ètale pour la topologie pro-Ètale), et cela nous suffira, puisqu'un faisceau pro-Ètale est automatiquement un faisceau pour la topologie quasi-pro-Ètale. La cohomologie de $\mathbb{B}_I$ sur $\tilde{D}_C$ est nulle (proposition \ref{descriptionbi}) ; on peut donc calculer la cohomologie de $\mathbb{B}_I$ sur $\bar{D}_C$ comme cohomologie du complexe de Cech pour le recouvrement $\tilde{D}_C \to \bar{D}_C$. Pour le dÈcrire, on utilise le lemme suivant.

\begin{lemme} 
Pour $i=1,\dots,n$, on note $\mathrm{ev}_i$ l'application d'Èvaluation en $T_i=0$. Pour tout $k\geq 1$, le produit fibrÈ $k$-fois $\tilde{D}_C \times_{\bar{D}_C} \dots \times_{\bar{D}_C} \tilde{D}_C$ est affinoÔde perfectoÔde, d'algËbre affinoÔde $(A_{n,k},A_{n,k}^+)$, avec $A_{n,k}=A_{n,k}^+[1/p]$ et $A_{n,k}^+$ est l'ensemble des $f \in \mathcal{C}^0(\zp^{k-1}, \O_C \langle \underline{T}^{1/p^{\infty}} \rangle)$ telles que pour tout $0 \leq i_1,\dots,i_j \leq n$, et tout $x=(x_1,\dots,x_{k-1}), y=(y_1,\dots,y_{k-1}) \in \zp^{k-1}$ avec $x_l=y_l$ dËs que $l\notin \{i_1,\dots,i_j \}$, alors
\[ \mathrm{ev}_{i_1} \circ \dots \circ \mathrm{ev}_{i_j} \circ f (x) = \mathrm{ev}_{i_1} \circ \dots \circ \mathrm{ev}_{i_j} \circ f(y). \]
\end{lemme}

\begin{proof}
Tout d'abord, il est clair que l'algËbre affinoÔde $(A_{n,k},A_{n,k}^+)$ est une algËbre affinoÔde perfectoÔde. 

Pour allÈger les notations, nous ne traitons que le cas $n=1$, $k=2$, mais le cas gÈnÈral est parfaitement similaire. Montrons que $\tilde{D} \times_{\bar{D}} \tilde{D}$ et $\mathrm{Spa}(A_{1,2})$ sont isomorphes. On peut dÈfinir deux morphismes d'algËbres continus $\lambda, \mu : C \langle \underline{T}^{1/p^{\infty}} \rangle \to A_{1,2}$ : $\lambda$ envoie $f \in C \langle T^{1/p^{\infty}} \rangle$ sur la fonction constante Ègale ‡ $f$ sur $\zp$ ; le second envoie $f \in C \langle T^{1/p^{\infty}} \rangle$ sur la fonction 
\[ \mu(f) : \zp \to C \langle T^{1/p^{\infty}} \rangle ~ ; ~ x \mapsto x \cdot f, \]
$x \cdot f$ dÈsignant l'ÈlÈment de $C \langle T^{1/p^{\infty}} \rangle$ obtenu en remplaÁant dans l'Ècriture de $f$ chaque terme $T^{a/p^m}$ par $\zeta_{p^m}^{ax} T^{a/p^m}$. L'image de ce morphisme est incluse dans $A_{1,2}$. De plus si $f\in C \langle T \rangle$, l'image de $f$ par ces deux morphismes est la mÍme. En d'autres termes, les deux morphismes d'espaces adiques $\lambda^*, \mu^* : \mathrm{Spa}(A_{1,2}) \to \tilde{D}$ correspondants sont les mÍmes aprËs composition avec le morphisme $\tilde{D} \to \bar{D}$. On a donc un morphisme $\mathrm{Spa}(A_{1,2}) \to \tilde{D} \times_{\bar{D}} \tilde{D}$. 

Pour montrer que c'est un isomorphisme, nous utiliserons le critËre suivant (\cite[Lem. 5.4]{S17}) : un morphisme $f : X\to Y$ quasi-compact quasi-sÈparÈ entre espaces perfectoÔdes est un isomorphisme si et seulement s'il induit une bijection entre les espaces topologiques sous-jacents et des isomorphismes entre corps rÈsiduels en tous les points de rang $1$. Comme le $v$-recouvrement $\tilde{D} \to \bar{D}$ donne en restriction au disque privÈ de l'origine un recouvrement pro-Ètale $\tilde{D} \backslash \{0\} \to \bar{D} \backslash \{0\}$ de groupe $\zp$, on sait dÈj‡ que le morphisme $\mathrm{Spa}(A_{1,2}) \to \tilde{D} \times_{\bar{D}} \tilde{D}$ considÈrÈ est un isomorphisme en dehors de l'unique point de $\tilde{D} \times_{\bar{D}} \tilde{D}$ au-dessus de l'origine. Par le critËre ci-dessus, il suffit pour conclure de montrer que la prÈ-image de ce point est un seul point, de corps rÈsiduel $C$. Soit donc $K$ un corps contenant $C$ et $\alpha : A_{1,2} \to K$ un morphisme d'algËbres tel que
\[ \alpha \circ \lambda =\alpha \circ \mu : C \langle T^{1/p^{\infty}} \rangle \to K ~ ; ~ f \mapsto f(0). \]
On peut Ècrire tout ÈlÈment de $A_{1,2}$ comme somme d'une constante et d'un ÈlÈment de $\mathcal{C}^0(\zp,C \langle T^{1/p^{\infty}} \rangle)$ dont le terme constant est nul. L'image de ce dernier est nulle. En effet, on peut approcher une fonction localement constante sur $\zp$ ‡ valeurs dans $C \langle T^{1/p^{\infty}} \rangle$, sans terme constant. L'identitÈ
\[ \forall x \in \zp, ~ \mathbf{1}_{i+p^l \zp} (x) = \frac{1}{p^l} \sum_{\zeta \in \mu_{p^l}} \zeta^{x-i} \]
si $i \in \z$ et $l>0$  montre qu'il suffit de vÈrifier que l'image de la fonction $x \mapsto \zeta_{p^l}^x T^{a/p^m}$, $a\neq 0$, $l, m \in \mathbf{N}$, est nulle. Or on peut toujours Ècrire cette fonction comme un multiple de $\mu(T^{1/p^k})$, pour $k$ assez grand, par une fonction de $\zp$ dans $C\langle T^{1/p^{\infty}} \rangle$ sans terme constant. Comme l'image de $\mu(T^{1/p^k})$ est nulle, c'est gagnÈ.
\end{proof}

On a une suite exacte
\begin{eqnarray} 0 \to  \mathcal{C}^0(\zp^{k-1}, \cap_{i=1}^n \mathrm{Ker}(\mathrm{ev}_i)) \to A_{n,k} \to A_{n-1,k-1}^n \to 0, \label{devissage} \end{eqnarray}
la deuxiËme flËche Ètant $f \mapsto \oplus ~ \mathrm{ev}_i \circ f$ et $A_{-1,k}$ Ètant $C$ par convention pour tout $k$.

Montrons que le sous-complexe du complexe de Cech de $\mathbb{B}_I$ (o˘ l'on suppose comme prÈcÈdemment que $I\subset [p^{-1},1[$) pour le recouvrement $\tilde{D}_C \to \bar{D}_C$ dont le $k$-Ëme terme est donnÈ par :
\[ \mathbb{B}_I (\mathcal{C}^0(\zp^{k-1}, \cap_{i=1}^n \mathrm{Ker}(\mathrm{ev}_i))) \]
est isomorphe \textit{aprËs application du foncteur dÈcalage $L\eta_t$} au sous-complexe du complexe $\Omega_{\bar{D}}^{\bullet} ~\widehat{\otimes}_{\qp} B_I$ dont le $k$-Ëme terme est formÈ des $\sum f_{i_1,\dots,i_{k-1}} dT_{i_1} \wedge \dots \wedge dT_{i_{k-1}}$ tels que pour chaque $(i_1,\dots,i_{k-1})$, $f_{i_1,\dots,i_{k-1}}$ est divisible par $T_1\dots T_n/(T_{i_1}\dots T_{i_{k-1}})$. 

Le $k$-Ëme terme de ce complexe se rÈÈcrit :
\[ \mathcal{C}^0(\zp^{k-1}, \cap_{i=1}^n \mathrm{Ker}(\mathrm{ev}'_i)), \]
o˘ cette fois-ci, $\mathrm{ev}'_i : B_I \langle \underline{X}^{1/p^{\infty}} \rangle \to B_I$ est l'Èvaluation en $X_i=0$, d'aprËs la proposition \ref{debile} et le lemme \ref{birinfty}. 

C'est donc le complexe standard des cochaÓnes continues pour l'action du groupe $\zp^n$ sur $\cap_{i=1}^n \mathrm{Ker}(\mathrm{ev}'_i)$. Il est isomorphe au complexe de Koszul
\[ K_{\cap_{i=1}^n \mathrm{Ker}(\mathrm{ev}'_i)}(\gamma_1-1,\dots,\gamma_n-1). \]
%Ce complexe s'Ècrit comme somme directe complÈtÈe de complexes de la forme :
%\[ K_{B_I \cdot \prod_{ i=1}^n X_i^{a(i)}}(\gamma_1-1,\dots,\gamma_n-1) \]
%avec pour tout $i=1,\dots,n$, $a(i) \in  \mathbf{N}[1/p] \backslash \{0\}$. Ce complexe est isomorphe au complexe
%\[ K_{B_I }([\epsilon^{a(1)}]-1,\dots,[\epsilon^{a(n)}]-1). \]
La suite de l'argument est semblable au raisonnement employÈ pour prouver \ref{couronne}. On dÈcompose :
\[ \cap_{i=1}^n \mathrm{Ker}(\mathrm{ev}'_i) = (\cap_{i=1}^n \mathrm{Ker}(\mathrm{ev}'_i))^{\rm int} \oplus (\cap_{i=1}^n \mathrm{Ker}(\mathrm{ev}'_i))^{\rm non int}, \]
avec $(\cap_{i=1}^n \mathrm{Ker}(\mathrm{ev}'_i))^{\rm int}=X_1\dots X_n B_I \langle \underline{X} \rangle$, puis l'on montre exactement de la mÍme maniËre que la multiplication par $t$ est homotope ‡ zÈro sur $K_{(\cap_{i=1}^n \mathrm{Ker}(\mathrm{ev}'_i))^{\rm nonint}}(\gamma_1-1,\dots,\gamma_n-1)$ et que
\[ L\eta_t K_{(\cap_{i=1}^n \mathrm{Ker}(\mathrm{ev}'_i))^{\rm int}}(\gamma_1-1,\dots,\gamma_n-1) = K_{(\cap_{i=1}^n \mathrm{Ker}(\mathrm{ev}'_i))^{\rm int}}(\frac{\gamma_1-1}{t},\dots,\frac{\gamma_n-1}{t}). \]
Il ne reste donc ‡ la fin que
\[ K_{X_1\dots X_n B_I \langle \underline{X} \rangle}(X_1 \frac{d}{dX_1},\dots, X_n \frac{d}{dX_n}). \] 
Un petit calcul montre que ce complexe est isomorphe au sous-complexe du complexe de de Rham mentionnÈ ci-dessus, via la flËche qui envoie $X_1\dots X_n f_{i_1,\dots,i_{k-1}} dX_{i_1} \wedge \dots \wedge dX_{i_{k-1}}$ en degrÈ $k$ sur $X_1\dots X_n X_{i_1}^{-1}\dots X_{i_{k-1}}^{-1} f_{i_1,\dots,i_{k-1}} dX_{i_1} \wedge \dots \wedge dX_{i_{k-1}}$ en degrÈ $k$. 

Finalement, on obtient par rÈcurrence sur $n$ avec la suite exacte \eqref{devissage} que le complexe de Cech de $\mathbb{B}_I$ pour le recouvrement $\tilde{D}_C \to \bar{D}_C$ devient quasi-isomorphe, aprËs qu'on lui a appliquÈ $L\eta_t$, ‡ $\Omega_{\bar{D}}^{\bullet} \hat{\otimes}_{\qp} B_I$. Cela conclut la preuve des quasi-isomorphismes de l'ÈnoncÈ, modulo l'hypothËse $(*)$ pour le second.
\end{proof}

ArmÈ des propositions \ref{couronne} et \ref{boule}, nous pouvons enfin calculer la cohomologie du faisceau $L\eta_t R\nu'_* \mathbb{B}_I$ pour les couronnes et les disques ouverts, ainsi que pour l'espace affine.

\begin{corollaire} \label{couronnebis}
Soit $r, r'$ deux rationnels et $i\geq 0$ (Èventuellement $r$ ou $r'$ est $+\infty$). On a
\[ H^i(\mathcal{C}_{r,r',C}, \mathbb{B}_I[1/t])= H_{\rm dR}^i(\mathcal{C}_{r,r'}) \widehat{\otimes}_{\qp} B_I[1/t] = \bigwedge^i B_I[1/t]^n. \]
Si l'hypothËse $(*)$ est vÈrifiÈe, on a mÍme :
\[ H^i(\mathcal{C}_{r,r',C,\mathrm{\acute{e}t}}, L\eta_t R\nu'_* \mathbb{B}_I)= H_{\rm dR}^i(\mathcal{C}_{r,r'}) \widehat{\otimes}_{\qp} B_I = \bigwedge^i B_I^n. \]
\end{corollaire}
\begin{proof}
Les deux formules se dÈduisent de la mÍme faÁon de la proposition \ref{couronne}. Nous expliquons l'argument pour le complexe de faisceaux $L\eta_t R\nu'_* \mathbb{B}_I$. En appliquant le mÍme argument en inversant $t$ ‡ chaque Ètape, on obtiendrait de mÍme le rÈsultat pour $\mathbb{B}_I[1/t]$. 

On choisit des rÈels $r_n, r'_n$ pour tout $n$ de sorte que les couronnes fermÈes $\bar{\mathcal{C}}_{r_n,r'_n}$ forment un recouvrement admissible de la couronne $\mathcal{C}_{r,r'}$. Si $\iota_n$ dÈnote l'inclusion $\bar{\mathcal{C}}_{r_n,r'_n,C} \to \mathcal{C}_{r,r',C}$, notons $\mathcal{F}_n$ le complexe $R\iota_{n,*} L\eta_t R\nu'_* \mathbb{B}_{I,\bar{\mathcal{C}}_{r_n,r'_n,C}}$, de faÁon que
\[ L\eta_t R\nu'_* \mathbb{B}_{I,\mathcal{C}_{r,r',C}} = \underset{n} \varprojlim ~ \mathcal{F}_n. \]
Nous affirmons que pour tout $i$,
\[ H^i(\mathcal{C}_{r,r',C,\mathrm{\acute{e}t}}, \underset{n}\varprojlim ~ \mathcal{F}_n) = \underset{n}\varprojlim ~H^i(\mathcal{C}_{r,r',C,\mathrm{\acute{e}t}}, \mathcal{F}_n) = \underset{n}\varprojlim ~ H^i(\mathcal{C}_{r_n,r'_n,C,\mathrm{\acute{e}t}}, L\eta_t R\nu'_* \mathbb{B}_I) \]
(la derniËre ÈgalitÈ est une consÈquence directe de la dÈfinition de $\mathcal{F}_n$). Pour cela on applique \cite[Prop. 13.3.1]{ega3}, ou plutÙt \cite[Rem. 13.3.2 (ii)]{ega3}, dans sa version topologique. Les conditions (i)-(iii) de loc. cit sont trivialement vÈrifiÈes en prenant comme base les ouverts Ètales quasi-compacts\footnote{Dans le cas du faisceau pro-Ètale $\mathbb{B}_I[1/t]$, on pourrait aussi directement appliquer la proposition \ref{replete}.}. Il faut seulement s'assurer que le systËme projectif formÈ des groupes de cohomologie des $\mathcal{F}_n$ vÈrifie Mittag-Leffler (topologique). Pour chaque $n$, choisissons des suites $r_{n,k}, r'_{n,k}$ de rÈels de sorte que les couronnes fermÈes $\bar{\mathcal{C}}_{r_{n,k},r'_{n,k},C}$ soient strictement emboÓtÈes d'intersection $\bar{\mathcal{C}}_{r_n,r'_n,C}$, et considÈrons
\[ \underset{k}\varinjlim ~H^i(\bar{\mathcal{C}}_{r_{n,k},r'_{n,k},C,\mathrm{\acute{e}t}},L\eta_t R\nu'_* \mathbb{B}_I). \]
On a 
\[ R~ \underset{n}\varprojlim ~ H^i(\mathcal{C}_{r_n,r'_n,C,\mathrm{\acute{e}t}}, L\eta_t R\nu'_* \mathbb{B}_I) = R ~ \underset{n} \varprojlim ~(\underset{k}\varinjlim ~H^i(\bar{\mathcal{C}}_{r_{n,k},r'_{n,k},C,\mathrm{\acute{e}t}},L\eta_t R\nu'_* \mathbb{B}_I)). \]
Il suffit donc de considÈrer le membre de droite. Or on a pour tout $n$, d'aprËs la proposition prÈcÈdente :
\[  \underset{k}\varinjlim ~ H^i(\bar{\mathcal{C}}_{r_{n,k},r'_{n,k},C,\mathrm{\acute{e}t}},L\eta_t R\nu'_* \mathbb{B}_I)) \simeq \underset{k} \varinjlim ~ H^i(\Omega^{\bullet}(\bar{\mathcal{C}}_{r_{n,k},r'_{n,k}}) \widehat{\otimes}_{\qp} B_I). \]
Notons que le membre de droite peut aussi se rÈÈcrire 
\[ \underset{k}\varinjlim ~ H^i(\Omega^{\bullet}(\mathcal{C}_{r_{n,k},r'_{n,k}}) \widehat{\otimes}_{\qp} B_I). \]
Comme une couronne ouverte est Stein, on peut appliquer le lemme \ref{commutetenseur} avec $W=B_I$ comme dans la preuve de la proposition \ref{basechange} ; puisque le foncteur limite inductive est exact, on obtient
finalement que pour tout $i$ :
\[ \underset{k}\varinjlim ~H^i(\bar{\mathcal{C}}_{r_{n,k},r'_{n,k},C,\mathrm{\acute{e}t}},L\eta_t R\nu'_* \mathbb{B}_I)) = \underset{k}\varinjlim ~ (H_{\rm dR}^i(\mathcal{C}_{r_{n,k},r'_{n,k}}) \widehat{\otimes}_{\qp} B_I). \]
Le systËme projectif quand $n$ varie des groupes de droite est le mÍme que le systËme projectif formÈ des $H_{\rm dR}^i(\mathcal{C}_{r_n,r'_n})^{\dagger} \widehat{\otimes}_{\qp} B_I$, et ces groupes vÈrifient Mittag-Leffler. Finalement, on a donc pour tout $i$ :
\[ H^i(\mathcal{C}_{r,r',C,\mathrm{\acute{e}t}},L\eta_t R\nu'_*\mathbb{B}_{I,\mathcal{C}_{r,r'}}) = \underset{n} \varprojlim ~ H_{\rm dR}^i(\mathcal{C}_{r_n,r'_n}) \widehat{\otimes}_{\qp} B_I, \]
qui est $H_{\rm dR}^i(\mathcal{C}_{r,r'}) \widehat{\otimes}_{\qp} B_I$.
\end{proof}

Exactement les mÍmes arguments, en utilisant cette fois la proposition \ref{boule}, donnent le

\begin{corollaire}  \label{boulebis}
Soit $X$ le disque ouvert de dimension $n$ et de rayon $r$ (Èventuellement $r=+\infty$, i.e. $X=\mathbf{A}^n$). On a $H^0(X_C,\mathbb{B}_I[1/t])=B_I[1/t]$ et $H^i(X_C, \mathbb{B}_I[1/t])=0$ si $i>0$. Si l'hypothËse $(*)$ est vraie, on a $H^0(X_{C,\mathrm{\acute{e}t}},L\eta_t R\nu'_* \mathbb{B}_I)=B_I$ et $H^i(X_{C,\mathrm{\acute{e}t}}, L\eta_t R\nu'_* \mathbb{B}_I)=0$ si $i>0$.
\end{corollaire} 

\begin{proof}[DÈmonstration de la proposition \ref{bcrisaffine}]
L'inclusion naturelle
\[ \mathbb{B}[1/t]^{\varphi=1} \to (\underset{I}\varprojlim ~ \mathbb{B}_I[1/t])^{\varphi=1}, \]
$I$ dÈcrivant les sous-intervalles compacts de $]0,1[$, est un isomorphisme. On a :
\[ R \varprojlim_I \mathbb{B}_I[1/t] = \varprojlim_I \mathbb{B}_I[1/t], \]
en vertu de \cite[Lem. 3.18]{Shodge} (appliqu\'e en choisissant comme base d'ouverts les affino\"ides perfecto\"ides). On dÈduit donc du corollaire \ref{boulebis} que la cohomologie de $\varprojlim_I \mathbb{B}_I[1/t]$ en degrÈ positif est nulle. Par consÈquent, la cohomologie de $\mathbb{B}[1/t]^{\varphi=1}$ est nulle en degrÈ $>1$ et est en degrÈ $1$ le conoyau de $1-\varphi$ sur $B[1/t]$, c'est-‡-dire zÈro.\end{proof}

On peut dÈsormais conclure la preuve du thÈorËme \ref{cohespaffine}.

\begin{proof}[DÈmonstration du thÈorËme \ref{cohespaffine}]
C'est maintenant immÈdiat : on applique pour finir la suite exacte longue qui se dÈduit de la \og suite exacte fondamentale faisceautique \fg{} \eqref{fonda} et les propositions \ref{bdraffine} et \ref{bcrisaffine} donnent que pour tout $i$,
\[ H^i(\mathbf{A}_C^n,\qp) = \mathrm{Ker}(d_{i})=\mathrm{Im}(d_{i-1}). \]
\end{proof}

\begin{remarque} \label{cdnsynto}
Colmez, Dospinescu et Nizio\l{} ont annoncÈ (\cite{cdn}) une description de la cohomologie pro-Ètale de $\qp$ sur un un espace Stein lisse sur un corps $p$-adique, \textit{ayant un modËle formel $p$-adique strictement semi-stable}, en termes de la cohomologie de Hyodo-Kato de la fibre spÈciale de ce modËle et de la cohomologie de de Rham. Leur dÈmonstration repose sur un thÈorËme de comparaison entre cohomologies Ètale et syntomique. Dans le cas particulier de l'espace affine, leur argument est prÈsentÈ dans \cite{cn2}.

Notre dÈmonstration est diffÈrente puisqu'elle ne fait pas appel ‡ la thÈorie syntomique. Toutefois, les deux approches sont liÈes : en effet, la relation entre cohomologies Ètale et syntomique peut s'interprÈter en termes du foncteur dÈcalage $L\eta_t$, comme nous le montrons dans \cite[\S 2.8]{lbsynto}.
\end{remarque}

\section{Groupes d'extensions de certains faisceaux pro-Ètales}\label{extproÈtales}
L'objectif de cette section est la dÈmonstration du rÈsultat suivant.
\begin{theoreme}\label{breen}
Dans la catÈgorie abÈlienne des faisceaux de groupes abÈliens sur $\mathrm{Perf}_{C,v}$, on a
\begin{center}
   \begin{tabular}{| l | c | r | }
     \hline
     $\mathrm{Hom}$ & $\mathbf{G}_a$ & $\qp$ \\ \hline
     $\mathbf{G}_a$ & $C$ & $0$ \\ \hline
     $\qp$ & $C$ & $\qp$ \\
     \hline
   \end{tabular}
  \quad 
   \begin{tabular}{| l | c | r | }
     \hline
     $\mathrm{Ext}^1$ & $\mathbf{G}_a$ & $\qp$ \\ \hline
     $\mathbf{G}_a$ & $C$ & $C$ \\ \hline
     $\qp$ & $0$ & $0$ \\
     \hline
   \end{tabular}
   \quad 
   \begin{tabular}{| l | c | r | }
     \hline
     $\mathrm{Ext}^2$ & $\mathbf{G}_a$ & $\qp$ \\ \hline
     $\mathbf{G}_a$ & $0$ & $0$ \\ \hline
     $\qp$ & $0$ & $0$ \\
     \hline
   \end{tabular}
    \end{center}
    (ces tableaux se lisant de gauche ‡ droite). 
\end{theoreme}
La stratÈgie gÈnÈrale pour ce faire est la suivante :

Soit $G$ un groupe abÈlien dans un topos $\mathcal{T}$. Il existe une rÈsolution canonique de $G$ de longueur infinie par un complexe dont chaque composante est une somme de termes de la forme $\z[G^i \times \z^j]$, dont on trouvera une construction dans le style cubique dans \cite{maclane} ou dans le cadre simplicial dans \cite{illusie}. Soit $n>0$ un entier. Si on ne s'intÈresse qu'aux calculs des groupes $\mathrm{Ext}^i$ avec $i<n$, il suffit de se donner une rÈsolution partielle de $G$ de longueur $n$, i.e. un complexe $C$ d'objets de ce topos, concentrÈ en degrÈs nÈgatifs, muni d'un morphisme d'augmentation $C \to G$ ($G$ vu comme complexe concentrÈ en degrÈ $0$) induisant un isomorphisme $H^0(C) \simeq G$ et tel que $H^i(C)=0$ si $i\neq 0, -n$. Le triangle distinguÈ
\[ H^{-n}(C)[n] \to C \to G \overset{+1} \longrightarrow \]
montre que si $G'$ est un autre groupe abÈlien du topos $\mathcal{T}$, en appliquant ‡ ce triangle le foncteur $R\mathrm{Hom}(\cdot,G')$ on obtient des isomorphismes 
\[ \mathrm{Ext}_{\mathcal{T}}^i(G,G') \simeq \mathrm{Ext}_{\mathcal{T}}^i(C,G') \]
pour tout $0 \leq i <n$. 

Pour $n=3$, ce complexe est construit dans \cite{bbm} en tronquant le complexe d'Eilenberg-MacLane stabilisÈ de \cite{maclane} et \cite{illusie}, et s'explicite comme suit :
\[ \z[G^4] \times \z[G^3] \times \z[G^3] \times \z[G^2] \times \z[G] \overset{\partial_3} \longrightarrow \z[G^3] \times \z[G^2] \overset{\partial_2} \longrightarrow \z[G^2] \overset{\partial_1} \longrightarrow \z[G]. \]
L'augmentation vers $G$ est le morphisme canonique $\epsilon : \z[G] \to G$. Les flËches sont dÈfinies par :
\[ \partial_1 [x,y] = [x+y] -[x]-[y], \]
\[ \partial_2 [x,y,z] = [x+y,z]-[y,z] -[x,y+z] +[x,y], \]
\[ \partial_2[x,y] = [x,y]-[y,x], \]
\[ \partial_3[x,y,z,w] = -[y,z,w]+[x+y,z,w]-[x,y+z,w] +[x,y,z+w] -[x,y,z], \]
\[ \partial_3 [x,y,z] = -[y,z] +[x+y,z] -[x,z] -[x,y,z]+[x,z,y] -[z,x,y] \]
pour le premier facteur $\z[G^3]$,
\[ \partial_3[x,y,z]=-[x,z]+[x,y+z]-[x,y]+[x,y,z]-[y,x,z]+[y,z,x] \]
pour le second facteur $\z[G^3]$, 
\[ \partial_3[x,y]=[x,y]+[y,x], \]
\[ \partial_3[x] = [x,x]. \]
Si $G'$ est un autre groupe abÈlien du topos $\mathcal{T}$, on a une suite spectrale (complexe filtrÈ)
\[ E_1^{pq} = \mathrm{Ext}_{\mathcal{T}}^q(C^{-p},G') \Rightarrow \mathrm{Ext}_{\mathcal{T}}^{p+q}(C,G'). \]
Le point est maintenant que les groupes $E_1^{pq}$ se calculent plus facilement, puisque si $G$ et $G'$ sont deux groupes abÈliens du topos des faisceaux sur un espace muni d'une topologie de Grothendieck et que $G$ est reprÈsentable, alors pour tout $i\geq 0$ et tout $k\geq 1$,
\[ \mathrm{Ext}_{\mathcal{T}}^i(\z[G],G') = H_{\mathcal{T}}^i(G,G'). \]

Dans la suite, le topos $\mathcal{T}$ sera le topos $\z-\mathrm{Sh}$ des faisceaux de groupes abÈliens sur $\mathrm{Perf}_{C,\mathrm{pro\acute{e}t}}$ ; $G$ et $G'$ seront soit le faisceau constant $\qp$ soit $\mathbf{G}_a$.

Pour calculer les groupes d'extensions entre $G$ et $G'$, on est donc ramenÈ ‡ calculer des groupes de cohomologie pro-Ètale et ‡ analyser des morphismes entre ces groupes de cohomologie. 

\begin{remarque} \label{vtopproet}
On pourrait en fait, comme on l'a indiquÈ dans la remarque \ref{remcarp}, remplacer partout la topologie pro-Ètale par la $v$-topologie. En effet, d'aprËs \cite[Prop. 14.7]{S17}, on sait que si $X$ est un espace adique analytique sur $C$ et $X$ le diamant sur $\mathrm{Spa}(C^{\flat})$ associÈ, $\mathcal{F}$ un faisceau pro-Ètale de groupes abÈliens sur $X$ et $\mu : X_v^{\diamond} \to X_{\mathrm{pro\acute{e}t}}^{\diamond}$ le morphisme de sites, alors pour tout $i\geq 0$,
\[ H_v^i(X^{\diamond},\mu^* \mathcal{F}) = H_{\mathrm{pro\acute{e}t}}^i(X,\mathcal{F}), \]
et donc les calculs de la section prÈcÈdente peuvent aussi Ítre vus comme des calculs de $v$-cohomologie.
\end{remarque}

\subsection{Le cas $G=\mathbf{G}_a$, $G'=\qp$}

On veut calculer les groupes $\mathrm{Ext}^i(\mathbf{G}_a,\qp)$ pour $i=0, 1, 2$. On va montrer que pour $i=1$ ce groupe est un $C$-espace vectoriel de dimension $1$ et qu'il est nul pour $i=0, 2$. 

Pour $i=0$, on a une flËche  $\mathrm{Hom}(C^{0},\qp) \to \mathrm{Hom}(C^{-1},\qp)$, qui est
\[ H^0(\mathbf{A}_C^1,\qp)=\qp \to H^0(\mathbf{A}_C^2,\qp)=\qp \quad ; \quad x \mapsto x-x-x=-x, \]
donc Èvidemment injective. Par consÈquent, $\mathrm{Hom}(\mathbf{G}_a,\qp)=0$. 
\\

Pour $i=2$, il y a trois flËches ‡ analyser :
\[ \mathrm{Hom}(C^{-2},\qp) \to \mathrm{Hom}(C^{-3},\qp), \]
\[ \mathrm{Ext}^1(C^{-1},\qp) \to \mathrm{Ext}^1(C^{-2},\qp), \]
\[ \mathrm{Ext}^2(C^{0},\qp) \to \mathrm{Ext}^2(C^{-1},\qp). \]
La derniËre flËche a un noyau nul, puisque $\mathrm{Ext}^2(C^0,\qp)=H^2(\mathbf{A}_C^1,\qp)=0$ d'aprËs la proposition \ref{cohodroiteaff}. La premiËre flËche envoie $(x,y) \in \qp^2 = H^0(\mathbf{A}_C^3,\qp) \oplus H^0(\mathbf{A}_C^2,\qp)$ sur $(-x,-2y,2(x-y),2y,y)$, donc est injective. Il reste ‡ Ètudier la deuxiËme flËche. Elle va de $H^1(\mathbf{A}_C^2,\qp)\simeq \O(\mathbf{A}_C^2)_0$ vers $H^1(\mathbf{A}_C^3,\qp) \oplus H^1(\mathbf{A}_C^2,\qp) \simeq \O(\mathbf{A}_C^3)_0 \oplus \O(\mathbf{A}_C^2)_0$ et envoie $f(X,Y)$ sur $(f(X+Y,Z)-f(X,Y+Z)+f(X,Y)-f(Y,Z),f(X,Y)-f(Y,X))$. Soit $f$ dans le noyau. Ecrivons $f=\sum_{q\geq 1} f_q$, o˘ pour tout $q$, $f_q$ est un polynÙme homogËne en $X, Y$ de degrÈ $q$. D'aprËs \cite[Lem. 3]{lazard}, $f_q$ est un multiple scalaire de $C_q$ pour tout $q\geq 2$. Autrement dit, $f$ est dans l'image de la flËche 
\[ \mathrm{Ext}^1(C^0,\qp) \to \mathrm{Ext}^1(C^{-1},\qp) \quad ; \quad f(X) \mapsto f(X+Y)-f(X)-f(Y). \]
Tous les termes en degrÈ $2$ de la deuxiËme page de la suite spectrale sont nuls et on a donc bien en dÈfinitive $\mathrm{Ext}^2(\mathbf{G}_a,\qp)=0$.
\\

Pour $i=1$, il s'agit d'analyser les flËches
\[ \mathrm{Hom}(C^{-1},\qp) \to \mathrm{Hom}(C^{-2},\qp), \]
\[ \mathrm{Ext}^1(C^0,\qp) \to \mathrm{Ext}^1(C^{-1},\qp). \]
La premiËre flËche est la flËche nulle de $H^0(\mathbf{A}_C^2,\qp)=\qp$ vers $H^0(\mathbf{A}_C^3,\qp)\oplus H^0(\mathbf{A}_C^2,\qp)=\qp^2$. On a aussi une flËche $\mathrm{Hom}(C^{0},\qp) \to \mathrm{Hom}(C^{-1},\qp)$, qui est
\[ H^0(\mathbf{A}_C^1,\qp)=\qp \to H^0(\mathbf{A}_C^2,\qp)=\qp \quad ; \quad x \mapsto x-x-x=-x, \]
et est donc surjective. Donc le noyau de $\mathrm{Hom}(C^{-1},\qp) \to \mathrm{Hom}(C^{-2},\qp)$ est inclus dans l'image de $\mathrm{Hom}(C^{0},\qp) \to \mathrm{Hom}(C^{-1},\qp)$. Donc $E_2^{1,0}=0$.   

La seconde flËche va de $\mathrm{Ext}^1(C^0,\qp)=H^1(\mathbf{A}_C^1,\qp)$ vers $\mathrm{Ext}^1(C^{-1},\qp)=H^1(\mathbf{A}_C^2,\qp)$. Elle envoie $f(X) \in \O(\mathbf{A}_C^1)_0$ sur $f(X+Y) - f(X)-f(Y)$. Si $f$ est dans le noyau, $f$ est donc nÈcessairement un multiple scalaire de l'identitÈ $\mathrm{Id}$. Donc $E_2^{0,1}=C$.

A la $r$-Ëme page de la suite spectrale, la diffÈrentielle $d_r$ envoie $E_r^{p,q}$ vers $E_r^{p+r,q-r+1}$. Donc
\[ E_r^{0,1} = \mathrm{Ker}(E_r^{0,1} \to E_r^{r,2-r}) / \mathrm{Im}(E_r^{-r,r} \to E_r^{0,1}) = \mathrm{Ker}(E_r^{0,1} \to E_r^{r,2-r}). \]
Or on vient de voir que pour $r=2$, $E_2^{k,2-k}=0$ pour tout $k$. Il en est donc de mÍme pour tout $r\geq 2$ et on a ainsi pour tout $r$, $E_r^{0,1}=E_2^{0,1}$. On a donc finalement $\mathrm{Ext}^1(\mathbf{G}_a,\qp) \simeq C$. Si $G$ et $G'$ sont deux faisceaux abÈliens avec $G'$ reprÈsentable par $Z$, une extension 
\[ 0 \to G \to F \to G' \to 0 \] 
dÈtermine une classe dans $H^1(Z,G)$ qui est l'image de $\mathrm{Id} \in \mathrm{Hom}(G',G')=G'(Z)=H^0(Z,G')$ par la flËche de connexion dÈduite de la suite exacte prÈcÈdente. Ici, l'isomorphisme $\O(\mathbf{A}_C^1)/C \simeq H^1(\mathbf{A}_C^1,\qp)$ utilisÈ est le morphisme de connexion dÈduit de la suite exacte de faisceaux
\[ 0 \to \qp \to \mathbb{B}[1/t]^{\varphi=1} \to \mathbb{B}_{\rm dR}/t\mathbb{B}_{\rm dR}^+ \to 0. \]
En prenant le tirÈ en arriËre de cette extension le long de $\mathbb{B}_{\rm dR}^+/t\mathbb{B}_{\rm dR}^+ \to \mathbb{B}_{\rm dR}/t\mathbb{B}_{\rm dR}^+$, on obtient une extension
\[ 0 \to \qp \to \mathbb{B}[1/t]^{\varphi=1} \to \mathbb{B}_{\rm dR}^+/t\mathbb{B}_{\rm dR}^+ \simeq \mathbf{G}_a \to 0, \]
qui est celle donnÈe par l'application $\theta$ de Fontaine. L'identitÈ de $\mathbf{G}_a $ dans lui-mÍme correspond ‡ la fonction identitÈ dans $H^0(\mathbf{A}_C^1,\mathbf{G}_a)=\O(\mathbf{A}_C^1)$. La flËche composÈe $H^0(\mathbf{A}_C^1,\mathbf{G}_a) \to H^0(\mathbf{A}_C^1,\mathbb{B}_{\rm dR}/t\mathbb{B}_{\rm dR}^+) \to H^1(\mathbf{A}_C^1,\qp)$ est la flËche quotient $\O(\mathbf{A}_C^1) \to \O(\mathbf{A}_C^1)/C$. On en dÈduit que l'ÈlÈment $\mathrm{Id} \in \mathrm{Ext}^1(\mathbf{G}_a,\qp)$ exhibÈ ci-dessus correspond ‡ l'extension $\mathbb{B}^{\varphi=p}$ de $\mathbf{G}_a$ par $\qp$ donnÈe par l'application $\theta$ de Fontaine.

\subsection{Le cas $G=G'=\mathbf{G}_a$}

Pour $i=0$, on a une flËche  $\mathrm{Hom}(C^{0},\mathbf{G}_a) \to \mathrm{Hom}(C^{-1},\mathbf{G}_a)$, qui s'identifie ‡
\[ H^0(\mathbf{A}_C^1,\mathbf{G}_a)=\O(\mathbf{A}_C^1) \to H^0(\mathbf{A}_C^2,\mathbf{G}_a)=\O(\mathbf{A}_C^2) \quad ; \quad f(X) \mapsto f(X+Y)-f(X)-f(Y). \]
Son noyau est de dimension $1$, engendrÈ par $f(X)=X$. Ainsi $\mathrm{Hom}(\mathbf{G}_a,\mathbf{G}_a)=C$. 
\\

Pour $i=2$, on regarde 
\[ \mathrm{Hom}(C^{-2},\mathbf{G}_a) \to \mathrm{Hom}(C^{-3},\mathbf{G}_a), \]
\[ \mathrm{Ext}^1(C^{-1},\mathbf{G}_a) \to \mathrm{Ext}^1(C^{-2},\mathbf{G}_a), \]
\[ \mathrm{Ext}^2(C^{0},\mathbf{G}_a) \to \mathrm{Ext}^2(C^{-1},\mathbf{G}_a). \]
Le dernier terme est nul. Analysons le premier : la flËche va de $\O(\mathbf{A}_C^3) \times \O(\mathbf{A}_C^2)$ vers $\O(\mathbf{A}_C^4) \times \O(\mathbf{A}_C^3) \times \O(\mathbf{A}_C^3) \times \O(\mathbf{A}_C^2) \times \O(\mathbf{A}_C^1)$. Les applications considÈrÈes (sauf la derniËre) prÈservent le degrÈ. Soit $(f,g)$ dans le noyau. On Ècrit $f(x,y,z)=\sum_{k,j,l} a_{k,j,l} x^k y^j z^l$ et $g(x,y)=\sum_{k,j} b_{k,l} x^k y^l$. L'annulation sur la composante $\mathbf{A}_C^4$ donne :
\[ {k+j \choose k} a_{k+j,l,m} - {j+l \choose j} a_{k,j+l,m}  + {l+m \choose l} a_{k,j,l+m} = 0, \]
pour $k, m \neq 0$ (en particulier $a_{k,0,l}=0$ pour $k,l>0$). Pour $m\neq 0$, on a aussi
\[ {j+l \choose j} a_{0,j+l,m} - {l+m \choose l} a_{0,j,l+m} = 0 \]
et de mÍme en Èchangeant les rÙles de $k$ et $m$. Enfin $a_{0,j,0}=0$. 

Si $a_{k,1,l}= (k+1)c_{k+1,l} - (l+1)c_{k,l+1}$ pour $k, l>0$, on vÈrifie immÈdiatement par rÈcurrence avec la relation ci-dessus que
\[ a_{k,j,l} = {k+j \choose j} c_{k+j,l} - {j+l \choose j} c_{k,j+l}, \]
pour tout $j > 0$. On pose aussi $j c_{0,j}=a_{0,j-1,1}$ pour $j>0$. Une rÈcurrence immÈdiate ‡ l'aide de la relation ci-dessus donne $a_{0,j,l} = {j+l \choose j} c_{0,j+l}$ pour tout $l>0$. Enfin on pose $c_{1,k}-c_{k,1}=b_{1,k}$. 

Si les $c_{j,k}$ sont ainsi dÈfinis, alors la premiËre flËche vers $\mathbf{A}_C^3$ donne $b_{j,k}=c_{j,k}-c_{k,j}$ pour $j >1$. L'autre flËche donne la mÍme ÈgalitÈ pour $k>1$. En fin de compte, si l'on pose $h(x,y)=\sum_{k,j} c_{k,j} x^k y^j$, on a 
\[ f(x,y,z)= h(x+y,z)-h(y,z)-h(x,y+z)+h(x,y) ~ ; ~ g(x,y)=h(x,y)-h(y,x). \]
En d'autres termes, $(f,g)$ est dans l'image de la flËche $\mathrm{Hom}(C^{-1},\mathbf{G}_a) \to \mathrm{Hom}(C^{-2},\mathbf{G}_a)$. 

Enfin on regarde le terme du milieu. Soit $\omega \in H^1(\mathbf{A}_C^2,\mathbf{G}_a)=\Omega^1(\mathbf{A}_C^2)$ dans le noyau, que l'on Ècrit $\omega=f(x,y)dx+g(x,y)dy$. On a 
\[ f(x+y,z)-f(x,y+z)+f(x,y) = 0, \]
\[ f(x+y,z)-f(y,z)-g(x,y+z)+g(x,y) = 0, \]
\[ g(x+y,z)-g(y,z)-g(x,y+z)= 0. \]
On vÈrifie aisÈment l'existence de $h \in \O(\mathbf{A}_C^1)$ telle que $f(x,y)=h(x+y)-h(x)$ et $g(x,y)=g(x+y)-g(x)$. En dÈfinitive, $\mathrm{Ext}^2(\mathbf{G}_a,\mathbf{G}_a)=0$. 
\\

Pour $i=1$, il s'agit d'analyser les flËches
\[ \mathrm{Hom}(C^{-1},\mathbf{G}_a) \to \mathrm{Hom}(C^{-2},\mathbf{G}_a), \]
\[ \mathrm{Ext}^1(C^0,\mathbf{G}_a) \to \mathrm{Ext}^1(C^{-1},\mathbf{G}_a). \]
La premiËre va de $H^0(\mathbf{A}_C^2,\mathbf{G}_a)=\O(\mathbf{A}_C^2)$ dans $H^0(\mathbf{A}_C^3,\mathbf{G}_a) \oplus H^0(\mathbf{A}_C^2,\mathbf{G}_a)=\O(\mathbf{A}_C^3) \oplus \O(\mathbf{A}_C^2)$ et envoie $f(X,Y)$ sur $(f(X+Y,Z)-f(Y,Z)-f(X,Y+Z)+f(X,Y),f(X,Y)-f(Y,X))$. Son noyau, toujours d'aprËs \cite[Lem. 3]{lazard}, est inclus dans l'image de la flËche $\mathrm{Hom}(C^{0},\mathbf{G}_a) \to \mathrm{Hom}(C^{-1},\mathbf{G}_a)$ ci-dessus. 

La deuxiËme flËche va de $H^1(\mathbf{A}_C^1,\mathbf{G}_a)=\Omega^1(\mathbf{A}_C^1)$ vers $H^1(\mathbf{A}_C^2,\mathbf{G}_a)=\Omega^1(\mathbf{A}_C^2)$. Elle envoie $\omega$ sur $m^*\omega -p_1^*\omega-p_2^*\omega$ (on a notÈ $p_1, p_2$ les deux morphismes de projection de $\mathbf{A}_C^1 \times \mathbf{A}_C^1$ vers $\mathbf{A}_C^1$ et $m : \mathbf{A}_C^1 \times \mathbf{A}_C^1 \to \mathbf{A}_C^1$ le morphisme de multiplication). Son noyau est l'espace des formes diffÈrentielles invariantes par translation sur $\mathbf{A}_C^1$, donc est de dimension $1$. 

Le cas $i=2$ dÈj‡ traitÈ permet comme ci-dessus de voir que $E_r^{0,1}=E_2^{0,1}$ pour tout $r\geq 2$ : ainsi, $\mathrm{Ext}^1(\mathbf{G}_a,\mathbf{G}_a)=C$. Identifions la classe correspondant ‡ la forme diffÈrentielle invariante $dT$. L'identification $R^1\nu_* \mathbf{G}_{a,X} \simeq \Omega_X^1$ est le morphisme de connexion dÈduit de la suite exacte
\[ 0 \to \mathbf{G}_{a,X}(1) \to \mathrm{gr}^1 \O \mathbb{B}_{\rm dR} \to \Omega_X^1 \otimes_{\O_X} \mathbf{G}_{a,X} \to 0 \]
(le premier graduÈ du lemme de PoincarÈ). Si $X=\mathbf{A}_C^1$, $f \mapsto fdT$ donne un isomorphisme $\O_X \simeq \Omega_X^1$. On a donc que l'extension correspondant ‡ $dT$ est $\mathrm{gr}^1 \O \mathbb{B}_{\rm dR}$. 

Or, comme les flËches du lemme de PoincarÈ sont strictes exactes, on a une suite exacte
\[ 0 \to \mathbb{B}_{\rm dR}^+/\mathrm{Fil}^2 \to \O \mathbb{B}_{\rm dR}/\mathrm{Fil}^2 \to \mathrm{gr}^0 \O \mathbb{B}_{\rm dR} \otimes \Omega^1 \] 
qui donne un isomorphisme $\mathbb{B}_{\rm dR}^+/\mathrm{Fil}^2 \simeq \mathrm{gr}^1 \O \mathbb{B}_{\rm dR}$.

%\begin{remarque}
%La premiËre ÈgalitÈ dit exactement que $y \mapsto f(\cdot,y)$ est un $1$-cocycle de $C$ ‡ valeurs dans $\O(\mathbf{A}_C^1)$ muni de l'action par translation de $C$, donc il doit y avoir un moyen plus conceptuel de le voir...
%\end{remarque}
 
%Or on a une suite exacte d'espaces de Banach-Colmez :
%\[ 0 \to \O \to B_{\rm dR}^+/t^2B_{\rm dR}^+ \to \O \to 0. \]
%Si $B_{\rm dR}^+/t^2B_{\rm dR}^+$ Ètait trivial comme $\mathbf{G}_a$-torseur, on aurait en particulier une section continue de $\theta : B_{\rm dR}^+(C)/t^2 B_{\rm dR}^+(C) \to C$ ; mais une telle section n'existe pas (cf. \cite[Rem. 1.5.8 a)]{per adiques}). Donc l'image de la flËche $\mathrm{Ext}^1(\O,\O) \to H^1(\mathbf{G}_a,\mathbf{G}_a)$ est exactement de dimension $1$. Si l'on montre qu'elle est injective, on saura que $\mathcal{F}$ est isomorphe soit ‡ $\O \times \O$ soit ‡ $B_{\rm dR}^+/t^2B_{\rm dR}^+$, donc est un Banach-Colmez.

\subsection{Le cas $G=\qp$ et $G'=\qp$ ou $G'=\mathbf{G}_a$}
Notons tout d'abord que le groupe $\mathrm{Ext}^i(\qp,G')$ s'insËre dans une suite exacte :
\[ 0 \to R^1 \underset{\times p} \varprojlim ~ \mathrm{Ext}^{i-1}(\zp,G') \to \mathrm{Ext}^i(\qp,G') \to \underset{\times p} \varprojlim ~ \mathrm{Ext}^i(\zp,G') \to 0. \]
Notons $$X=\mathrm{Spa}(\con(\zp,C),\con(\zp,\O_C)).$$ Cet espace adique reprÈsente le faisceau pro-Ètale $\zp$. 
\begin{proposition}
Pour tout $n\geq 1$, on a $H^0(X^n,\qp)=\con(\zp^n,\qp)$ et $H^i(X^n,\qp)=0$ pour tout $i>0$. On a $H^0(X^n,\mathbf{G}_a)=\con(\zp^n,C)$ et $H^i(X^n,\mathbf{G}_a)=0$ si $i>0$. 
\end{proposition}
\begin{proof} 
Soit $i\geq 1$ et $k\geq 0$. Comme $X^n$ est profini sur $C$ (donc strictement totalement discontinu), $H^i(X^n,\z/p^k)=H_{\mathrm{\acute{e}t}}^i(X^n,\z/p^k)=0$. On en dÈduit avec la suite exacte
\[ 0 \to R^1\varprojlim H^{i-1}(X^n,\z/p^k) \to H^i(X^n,\zp) \to \varprojlim H^i(X^n,\z/p^k) \to 0 \]
que $H^i(X^n,\zp)=0$ pour tout $i\geq 1$ (pour $i=1$, on vÈrifie ‡ la main que $R^1\varprojlim$ s'annule). Comme $H^i(X^n\qp)=H^i(X^n,\zp)[1/p]$ par quasi-compacitÈ de $X^n$, on a le rÈsultat. La deuxiËme phrase est une consÈquence du fait que $$X^n=\mathrm{Spa}(\con(\zp^n,C),\con(\zp^n,\O_C))$$ est affinoÔde perfectoÔde, cf. \cite[Th. 6.5 (ii)]{KL}. 
\end{proof}
Supposons d'abord $G'=\mathbf{G}_a$. ConsidÈrons le complexe 
\[ C \to \con(\zp,C) \to \con(\zp^2,C)\to \con(\zp^3,C) \oplus \con(\zp^2,C) \]
\[ \to \con(\zp^4,C) \oplus \con(\zp^3,C)^2 \oplus \con(\zp^2,C) \oplus \con(\zp,C), \]
la premiËre flËche Ètant $a \mapsto (x \mapsto ax)$ et les flËches suivantes Ètant les flËches Èvidentes dÈduites de la dÈfinition du complexe $C$. On va montrer que ce complexe est exact.
%\[ a \mapsto ax ~ ; ~  f(x) \mapsto f(x+y)-f(x)-f(y) \]
%\[ f(x,y) \mapsto (f(x+y,z)-f(y,z)-f(x,y+z)+f(x,y),f(x,y)-f(y,x)). \] 
%\[ (f(x,y,z),g(x,y)) \]

Si $f$ est dans le noyau de la premËre flËche, $f$ est un multiple scalaire de l'identitÈ sur $\z$ donc sur $\zp$ par densitÈ. Par consÈquent, $\mathrm{Hom}(\zp,\mathbf{G}_a)=C$. VÈrifions l'exactitude en $\con(\zp^2,C)$ : comme $\con(\zp^2,C)=\con(\zp,C) ~ \widehat{\otimes} ~ \con(\zp,C)$ (puisque $\zp$ est compact), les fonctions $(x,y) \mapsto {x \choose k} {y \choose l}$, avec $k, l \in \mathbf{N}$, forment une base orthonormale de $\con(\zp^2,C)$ (dÈveloppement de Mahler d'une fonction continue sur $\zp$). Soit $f=\sum_{k,l \geq 0} a_{k,l} {x \choose k} {y \choose l}$, avec $a_{k,l} \to 0$ quand $k+l \to 0$, vÈrifiant 
\[ f(x+y,z)-f(y,z)-f(x,y+z)+f(x,y)=0 \quad ; \quad f(x,y)-f(y,x)=0. \]
A l'aide de ces deux ÈgalitÈs et de la relation ${x+y \choose k} = \sum_{j=0}^k {x \choose j} {y \choose k-j}$, on obtient immÈdiatement que $a_{k,l}$ ne dÈpend que de la somme $k+l$ si $k, l>0$, et que $a_{k,0}=a_{0,l}=0$ si $k,l>0$. Posons alors $b_n=a_{k,l}$ si $n>1$ avec $k,l$ non nuls tels que $k+l=n$, $b_0=a_{0,0}$ et $b_1$ quelconque. On vÈrifie que
\[ f(x,y)=g(x+y)-g(x)-g(y), \]
avec $g(x)=\sum_n b_n {x \choose n}$. Notons que si l'on Ècrit plutÙt $f$ sous la forme $f= \sum_{k,l \geq 0} a'_{k,l} k!{x \choose k} l!{y \choose l}$, sans se prÈocupper du comportement asymptotique des coefficients, comme $k!{x+y \choose k} = \sum_{j=0}^k {k \choose j} j!{x \choose j} (k-j)!{y \choose k-j}$, les relations obtenues sur les coefficients $a'_{k,l}$ sont exactement les mÍmes que celles obtenues ci-dessus dans le calcul de $\mathrm{Ext}^1(\mathbf{G}_a,\mathbf{G}_a)$. Cette observation permet de montrer aussi l'exactitude en $\con(\zp^3,C) \oplus \con(\zp^2,C)$ : il suffit de reprendre le mÍme calcul que celui montrant l'annulation de $\mathrm{Ext}^2(\mathbf{G}_a,\mathbf{G}_a)$ ci-dessus (la convergence ne pose pas de problËme).

Si $G'=\qp$, il faut considÈrer le complexe
\[ \qp \to \con(\zp,\qp) \to \con(\zp^2,\qp)  \to \con(\zp^3,\qp) \oplus \con(\zp^2,\qp) \]
\[ \to \con(\zp^4,\qp) \oplus \con(\zp^3,\qp)^2 \oplus \con(\zp^2,\qp) \oplus \con(\zp,\qp), \]
avec les mÍmes flËches que ci-dessus, dont on montre exactement de la mÍme maniËre qu'il est exact.

\section{Faisceaux cohÈrents sur la courbe de Fargues-Fontaine} \label{faisceauxsurlacourbe}

\subsection{Rappels sur la courbe de Fargues-Fontaine} \label{rappelscourbe}
On fixe une extension finie $E$ de $\qp$ d'uniformisante $\pi$ et de corps rÈsiduel $\mathbf{F}_q$. Soit $S=\Spa(R,R^+)$ un espace affinoÔde perfectoÔde de caractÈristique $p$. On dÈfinit l'espace 
\[ Y_{S,E} = \Spa(W_{\O_E}(R^{\circ}),W_{\O_E}(R^+)) \backslash V(\pi[\varpi]), \]
$\varpi$ Ètant une pseudo-uniformisante (ÈlÈment inversible topologiquement nilpotent) de $R$. Il s'agit d'un espace adique (i.e. le prÈfaisceau structural est un faisceau), qui est l'analogue en caractÈristique mixte du disque unitÈ ouvert ÈpointÈ sur la base $S$, si l'on pense aux vecteurs de Witt comme ‡ des \og fonctions holomorphes de la variable $\pi$ \fg{}. On peut Ècrire 
\[ Y_{S,E}=\bigcup_{n,m \geq 1} Y_{S,E,n,m} = \bigcup_{n,m\geq 1} \mathrm{Spa}(W_{\O_E}(R^{\circ}) \langle \frac{[\varpi]^n}{\pi}, \frac{\pi^m}{[\varpi]} \rangle, W_{\O_E}(R^{+}) \langle \frac{[\varpi]^n}{\pi}, \frac{\pi^m}{[\varpi]} \rangle). \]
L'espace $Y_{S,E}$ est muni d'un Frobenius $\varphi$ qui agit sur les fonctions par
\[ \varphi(\sum_n [x_n]\pi^n) = \sum_n [x_n^q]\pi^n. \]
L'anneau $\O(Y_{S,E})$ des fonctions sur $Y_{S,E}$ est le complÈtÈ $B_E(R)$ de
\[ B_E^b(R):= \left\{ \sum_{n \gg -\infty} [x_n]\pi^n, x_n \in R, (x_n) ~ \mathrm{born\acute{e}e}  \right\} = W_{\O_E}(R^{\circ})[1/\pi,1/[\varpi]], \]
l'anneau des \og fonctions mÈromorphes le long des diviseurs $\pi=0$ et $[\varpi]=0$ \fg{}, relativement aux normes $\|.\|_{\rho}$, $0< \rho < 1$, dÈfinies par 
\[ \|\sum_n [x_n]\pi^n \|_{\rho} = \sup_n \|x_n\| \rho^n, \]
$\|.\|$ Ètant une norme sur $R$ multiplicative pour les puissances qui fait de $R^{\circ}$ la boule unitÈ de $R$. 

A l'intÈrieur de $B_E^b(R)$ vit
\[ B_E^{b +}(R) = \left\{ \sum_{n \gg -\infty} [x_n]\pi^n, x_n \in R, \|x_n\| \leq 1 \right\} = W_{\O_E}(R^{\circ})[1/\pi], \]
dont l'adhÈrence dans $B_E(R)$ est notÈe $B_E^+(R)$.

Si $f \in B_E(R)$, $\|\varphi(f)\|_{\rho}=\|f\|_{\rho^{1/q}}^q$. L'action de $\varphi$ sur $Y_{S,E}$ est donc proprement discontinue et on peut passer au quotient. L'espace adique $X_{S,E}:=Y_{S,E}/\varphi^{\z}$ ainsi obtenu est \textit{la courbe de Fargues-Fontaine relative sur $S$, pour le corps local $E$}. Si $E=\mathbf{Q}_{p^h}$ est l'extension non ramifiÈe de degrÈ $h$ de $\qp$, on Ècrira $Y_{S,h}$ et $X_{S,h}$ au lieu de $Y_{S,E}$ et $X_{S,E}$ ; si $h=1$ on oubliera carrÈment l'indice $1$. 

\begin{remarque}
Les anneaux $B_E(R)$ et $B_E^+(R)$ sont ÈtudiÈs en dÈtail dans \cite{FF} lorsque $R=C^{\flat}$ et dans \cite{KL} dans le cas gÈnÈral, avec des notations diffÈrentes : l'anneau $B(R)$ (resp. $B^+(R)$) y est notÈ $\tilde{\mathcal{R}}_R$ (resp. $\tilde{\mathcal{R}}_R^+$).
\end{remarque} 

Si $S=\Spa(C^{\flat})$, on notera simplement $Y_E$ et $X_E$ au lieu de $Y_{S,E}$ et $X_{S,E}$ (par exemple, en accord avec la convention prÈcÈdente, $X$ dÈsigne la courbe associÈe au choix $S=\mathrm{Spa}(C^{\flat})$ et $E=\qp$). Il s'agit alors de la courbe de Fargues-Fontaine originale, ÈtudiÈe en dÈtail dans \cite{FF}. La courbe adique $X_E$ admet un pendant schÈmatique $X_E^{\rm sch}:=\mathrm{Proj}(P_E)$, o˘ $P_E$ est l'algËbre graduÈe :
\[ P_E= \bigoplus_{d\geq 0} \O(Y_E)^{\varphi=\pi^d} =\bigoplus_{d\geq 0} B_E^{\varphi=\pi^d} \]
(on a notÈ $B_E=B_E(C)$). Le schÈma $X_E^{\rm sch}$ est rÈgulier noethÈrien de dimension $1$ sur $E$, mais trËs loin d'Ítre de type fini : le corps rÈsiduel de $X_E^{\rm sch}$ en un point fermÈ $x$ est un corps perfectoÔde de caractÈristique $0$ dont le basculÈ est isomorphe ‡ $C^{\flat}$, donc en particulier de dimension infinie sur $E$ ! Toutefois, tout se passe comme si $X_E$ Ètait la courbe adique associÈe ‡ $X_E^{\rm sch}$. En particulier, on a une Èquivalence de type GAGA entre faisceaux cohÈrents sur $X_E$ et sur $X_E^{\rm sch}$ (\cite[Th. 8.7.7]{KL}). En outre, bien que $X_E^{\rm sch}$ ne soit pas de type fini, Fargues et Fontaine prouvent que $X_E$ est une courbe \textit{complËte}, au sens o˘ si $f \in E(X_E^{\rm sch})$, 
\[ \sum_{x \in |X_E^{\rm sch}|} v_x(f)=0 \]
($v_x$ est la valuation sur l'anneau local de $X_E^{\rm sch}$ en le point fermÈ $x$, qui est un anneau de valuation discrËte puisque $X_E^{\rm sch}$ est rÈgulier de dimension $1$). Cela permet de dÈfinir le degrÈ d'un fibrÈ en droites et, en utilisant le dÈterminant, le \textit{degrÈ} d'un fibrÈ de rang quelconque sur $X_E^{\rm sch}$. Si l'on dÈfinit la \textit{pente} d'un fibrÈ vectoriel comme le quotient de son degrÈ par son rang, tout fibrÈ $\mathcal{F}$ sur $X_E^{\rm sch}$ admet une unique filtration croissante (la \textit{filtration de Harder-Narasimhan}) 
\[ 0 = \mathcal{F}_0 \subset \mathcal{F}_1 \subset \dots \subset \mathcal{F}_n =\mathcal{F}, \]
telle que chaque $\mathcal{F}_{i+1}/\mathcal{F}_i$ soit semi-stable de pente $\mu_{i+1}$ (tout sous-fibrÈ est de pente $\leq \mu_{i+1}$) avec $\mu_1>\dots > \mu_n$. Si l'on dÈcrËte en outre qu'un faisceau cohÈrent de torsion est de pente $+\infty$, le formalisme de Harder-Narasimhan s'Ètend aux faisceaux cohÈrents sur $X_E^{\rm sch}$, puisque, $X_E^{\rm sch}$ Ètant rÈgulier de dimension $1$, tout faisceau cohÈrent sur $X_E^{\rm sch}$ est somme d'un fibrÈ et d'un faisceau de torsion. Ces rÈsultats se tranposent via l'Èquivalence GAGA ‡ $X_E$.
\\

Soit $(D,\varphi)$ un isocristal sur l'extension non ramifiÈe maximale $\breve{E}$ de $E$. On peut associer ‡ $(D,\varphi)$ un fibrÈ vectoriel sur $X_E$, notÈ $\mathcal{E}(D)$, dont la rÈalisation gÈomÈtrique est le quotient $Y \times_{\varphi} D$, $\varphi$ agissant diagonalement. Le thÈorËme GAGA pour la courbe donne un fibrÈ vectoriel $\mathcal{E}(D)^{\rm sch}$ sur $X_E^{\rm sch}$ qui est en fait le faisceau quasi-cohÈrent associÈ au $P_E$-module graduÈ
\[ \bigoplus_{d\geq 0} (B_E \otimes_{\breve{E}} D)^{\varphi=\pi^d} \]
(qui est donc en fait un fibrÈ). La catÈgorie $\varphi-\mathrm{Mod}_{\breve{E}}$ des isocristaux sur $\breve{E}$ est semi-simple, avec pour chaque rationnel $\lambda$ un unique objet simple de pente $\lambda$. On a donc pour chaque $\lambda \in \q$ un fibrÈ $\O_{X_E}(\lambda)$ sur $X_E$ dÈfini comme l'image par $\mathcal{E}$ de l'unique isocristal simple de pente $-\lambda$. C'est le fibrÈ associÈ au diviseur $\lambda.\infty$ si $\lambda$ est entier. Pour tout $\lambda$, $\O_{X_E}(\lambda)$ est de pente $\lambda$\footnote{Ce qui explique la convention de signe.}. Dans la suite de ce texte, nous nous intÈresserons aux fibrÈs sur $X$ mais le fait suivant est utile pour les arguments : si $\lambda=d/h$, le fibrÈ de rang $h$ $\O_X(\lambda)$ est le poussÈ en avant du fibrÈ en droites $\O_{X_h}(d)$ sur $X_h$.

Le thÈorËme principal de \cite{FF} concernant les fibrÈs sur $X_E$ (ou $X_E^{\rm sch}$) est le suivant.
\begin{theoreme} \label{classification}
Le foncteur $\mathcal{E} : D \mapsto \mathcal{E}(D)$ de $\varphi-\mathrm{Mod}_{\breve{E}}$ vers la catÈgorie des fibrÈs vectoriels sur $X_E$ est essentiellement surjectif. En d'autres termes, tout fibrÈ sur la courbe se dÈcompose (non canoniquement) comme somme de fibrÈs $\O_{X_E}(\lambda)$, $\lambda \in \q$, et les $\lambda$ qui apparaissent sont uniquement dÈterminÈs ‡ permutation prËs.
\end{theoreme}

\begin{remarque} a) En particulier, la filtration de Harder-Narasimhan des faisceaux cohÈrents sur $X$ est (non canoniquement) scindÈe. Cela se traduit par le fait que $\mathrm{Ext}^1(\O_{X}(\lambda),\O_{X}(\mu))=\bigoplus H^1(X,\O_X(\mu-\lambda))=0$ si $\lambda \leq \mu$. Les objets semi-stables de pente $\lambda$ sont somme directe de copies de l'unique objet stable de pente $\lambda$, $\O(\lambda)$.

b) Le foncteur $\mathcal{E}$ est trËs loin d'Ítre pleinement fidËle. A titre d'exemple, le $\qp$-espace vectoriel $\mathrm{Hom}(\O_X,\O_X(1))=B^{\varphi=p}$ est de dimension infinie, alors que le groupe correspondant dans la catÈgorie des isocristaux est nul.
\end{remarque}

\subsection{Paires de torsion et coeurs abÈliens} \label{generalites}

CommenÁons par une dÈfinition gÈnÈrale.
\begin{definition}
Soit $\mathcal{A}$ une catÈgorie abÈlienne. Une \textit{paire de torsion} de $\mathcal{A}$ est la donnÈe de deux sous-catÈgories pleines $\mathcal{T}$ et $\mathcal{T}'$ de $\mathcal{A}$, telles que $\mathrm{Hom}_{\mathcal{A}}(T,T')=0$ pour tout $T \in \mathcal{T}$ et $T' \in \mathcal{T}'$ et telles que pour tout objet $A \in \mathcal{A}$, il existe une suite exacte
\[ 0 \to T \to A \to T' \to 0 \]
avec $T \in \mathcal{T}$ et $T' \in \mathcal{T}'$.
\end{definition}

Soit $D$ la catÈgorie dÈrivÈe bornÈe de $\mathcal{A}$. On suppose que l'on s'est donnÈ une paire de torsion $(\mathcal{T},\mathcal{T}')$. Le fait suivant est standard (voir, par exemple, \cite[Ch. 1]{TorsionH}).
\begin{proposition}
Soit $k \in \z$. La sous-catÈgorie pleine
\[ \mathcal{A}_{\mathcal{T},\mathcal{T}',k} := \{ A \in D, H^i(A)=0 ~ \mathrm{pour} ~ i \neq k-1, k ~ ; ~ H^{k-1}(A) \in \mathcal{T}', H^k(A) \in \mathcal{T} \} \]
est un coeur de $D$. 
\end{proposition}
En d'autres termes, il existe une $t$-structure bornÈe sur $D$ dont le coeur est $\mathcal{A}_{\mathcal{T},\mathcal{T}',k}$. 
\\

L'exemple fondamental sera pour nous le suivant. Soit $\mathcal{A}$ une catÈgorie abÈlienne munie de deux fonctions rang et degrÈ, telle que les objets de $\mathcal{A}$ admettent une filtration de Harder-Narasimhan relativement ‡ la fonction pente associÈe. Si $m \in \mathbf{R}$, on note $\mathcal{T}_m^-$ (resp. $\mathcal{T}_m^{'-}$) comme la sous-catÈgorie pleine de $\mathcal{A}$ formÈe des objets dont tous les quotients de la filtration de Harder-Narasimhan sont de pente $\geq m$ (resp. $< m$). C'est une paire de torsion de la catÈgorie dÈrivÈe bornÈe de $\mathcal{A}$. On dÈfinit de mÍme une paire de torsion $(\mathcal{T}_m^+,\mathcal{T}_m^{'+})$ en remplaÁant $\geq m$ par $>m$ et $<m$ par $\leq m$. Pour allÈger les notations, on notera dans la suite :
\[ \mathcal{A}^-:=\mathcal{A}_{\mathcal{T}_0^-,\mathcal{T}_0^{'-},0} \quad ; \quad \mathcal{A}^+=\mathcal{A}_{\mathcal{T}_0^+,\mathcal{T}_0^{'+},1}. \]
Si l'on Ètend additivement ‡ la catÈgorie dÈrivÈe les fonctions rang et degrÈ, la catÈgorie $\mathcal{A}^{-}$ est munie de fonctions degrÈ, rang et pente dÈfinies par 
\[ \mathrm{deg}^{-}=-\mathrm{rg} \quad ; \quad \mathrm{rg}^{-}=\mathrm{deg} \quad ; \quad \mu^{-}=-\mathrm{rg}/\mathrm{deg} \]
et ses objets admettent des filtrations de Harder-Narasimhan pour $\mu^{-}$. Si $A \in \mathcal{A}^-$, le triangle exact
\[ H^{-1}(A)[1] \to A \to H^0(A) \overset{+1} \longrightarrow \]
donne une suite exacte dans la catÈgorie abÈlienne $\mathcal{A}^-$
\[ 0 \to H^{-1}(A)[1] \to A \to H^0(A) \to 0. \]
Dans cette suite exacte, le terme de gauche correspond ‡ la partie de pentes $>0$ et le terme de droite ‡ la partie de pentes $\leq 0$\footnote{Noter que les objets de torsion sont de pente $-\infty$ dans la nouvelle catÈgorie.}. On a des formules analogues pour $\mathcal{A}^{+}$ en posant cette fois-ci $\mathrm{deg}^{+}=\mathrm{rg}$ et $\mathrm{rg}^{+}=-\mathrm{deg}$. 
\\

Choisissons pour catÈgorie abÈlienne la catÈgorie $\mathrm{Coh}_X$ des faisceaux cohÈrents sur $X$. Soit $\mathcal{F} \in  \mathrm{Coh}_X^{-}$. Comme
\begin{align*}
\mathrm{Ext}_{\mathrm{Coh}_X^{-}}^1(H^0(\mathcal{F}),H^{-1}(\mathcal{F})[1]) &=\mathrm{Ext}_{D^b(\mathrm{Coh}_X)}^1(H^0(\mathcal{F}),H^{-1}(\mathcal{F})[1]) \\
& =\mathrm{Ext}_{\mathrm{Coh}_X}^2(H^0(\mathcal{F}),H^{-1}(\mathcal{F}))=0
\end{align*}
(car $X$ est noethÈrien rÈgulier de dimension $1$), $\mathcal{F} \simeq H^{-1}(\mathcal{F})[1] \oplus H^0(\mathcal{F})$ (non canoniquement) ; on peut donc penser ‡ un ÈlÈment de $\mathrm{Coh}_{X}^{0,-}$ comme ‡ un couple $(\mathcal{F}',\mathcal{F}'')$ de faisceaux cohÈrents sur $X$, avec $\mathcal{F}'$ ‡ pentes strictement nÈgatives, $\mathcal{F}''$ ‡ pentes positives. De mÍme, le fait que la filtration de Harder-Narasimhan soit scindÈe permet de voir un faisceau cohÈrent sur $X$ comme un couple de faisceaux $(\mathcal{F}',\mathcal{F}'')$ comme prÈcÈdemment. Mais bien s˚r les morphismes entre deux tels couples $(\mathcal{F}',\mathcal{F}'')$ et $(\mathcal{G}',\mathcal{G}'')$ dans les deux catÈgories sont diffÈrents. On a 
\[ \mathrm{Hom}_{\mathrm{Coh}_X}((\mathcal{F}',\mathcal{F}''),(\mathcal{G}',\mathcal{G}'')) = \begin{pmatrix} \mathrm{Hom}(\mathcal{F}',\mathcal{G}') & 0 \\ \mathrm{Hom}(\mathcal{F}',\mathcal{G}'') & \mathrm{Hom}(\mathcal{F}'',\mathcal{G}'') \end{pmatrix} \]
et 
\[ \mathrm{Hom}_{\mathrm{Coh}_X^{-}}((\mathcal{F}',\mathcal{F}''),(\mathcal{G}',\mathcal{G}'')) = \begin{pmatrix} \mathrm{Hom}(\mathcal{F}',\mathcal{G}') & \mathrm{Ext}^1(\mathcal{F}'',\mathcal{G}') \\ 0 & \mathrm{Hom}(\mathcal{F}'',\mathcal{G}'') \end{pmatrix}. \]

Si l'on Ètend additivement ‡ la catÈgorie dÈrivÈe les fonctions rang et degrÈ, la catÈgorie $\mathrm{Coh}_X^{-}$ est munie de fonctions degrÈ, rang et pente dÈfinies par 
\[ \mathrm{deg}^{0,-}=-\mathrm{rg} \quad ; \quad \mathrm{rg}^{0,-}=\mathrm{deg} \quad ; \quad \mu^{0,-}=-\mathrm{rg}/\mathrm{deg} \]
et ses objets admettent des filtrations de Harder-Narasimhan pour $\mu^{0,-}$. Dans la suite exacte dans $\mathrm{Coh}_X^{-}$ :
\[ 0 \to H^{-1}(\mathcal{F})[1] \to \mathcal{F} \to H^0(\mathcal{F}) \to 0 \]
le terme de gauche correspond ‡ la partie de pentes $>0$ et le terme de gauche ‡ la partie de pentes $\leq 0$\footnote{Noter que les objets de torsion sont de pente $-\infty$ dans la nouvelle catÈgorie.}. On a des formules analogues pour $\mathrm{Coh}_X^{+}$ en posant cette fois-ci $\mathrm{deg}^{+}=\mathrm{rg}$ et $\mathrm{rg}^{+}=-\mathrm{deg}$. Cette observation permet de donner un sens ‡ l'ÈnoncÈ suivant. 

\begin{proposition}\label{recupere}
Les catÈgories abÈliennes $\mathrm{Coh}_X$ et $(\mathrm{Coh}_X^{-})^{+}$ sont naturellement Èquivalentes et via cette Èquivalence, les fonctons rang et degrÈ se correspondent.
\end{proposition}
\begin{proof}
Par dÈfinition, $\mathrm{Coh}_X^{-}$ est une sous-catÈgorie de $D^b(\mathrm{Coh}_X)$, donc on a un foncteur exact (\cite[Appendix]{beilinson})
\[ D^b(\mathrm{Coh}_X^{-}) \to D^b(\mathrm{Coh}_X), \]
induisant le foncteur identitÈ sur $\mathrm{Coh}_X^-$, et donc un foncteur exact $\iota : (\mathrm{Coh}_X^{-})^{+} \to D^b(\mathrm{Coh}_X)$. Soit $\mathcal{F} \in (\mathrm{Coh}_X^{-})^{+}$ ; on a $\mathcal{F}=\mathcal{F}' \oplus \mathcal{F}''[-1]$, avec $\mathcal{F}' \in \mathrm{Coh}_X^-$ ‡ pentes $\leq 0$ et $\mathcal{F}'' \in \mathrm{Coh}_X^-$ ‡ pentes $>0$. On a donc $\mathcal{F}' \in \mathrm{Coh}_X$ ‡ pentes $\geq 0$ et $\mathcal{F}''=\mathcal{F}'''[1]$, avec $\mathcal{F}''' \in \mathrm{Coh}_X$ ‡ pentes $<0$. Finalement, $\iota(\mathcal{F})=\mathcal{F}''' \oplus \mathcal{F}'$. Ceci prouve que $\iota(\mathcal{F}) \in \mathrm{Coh}_X$ et aussi que $\iota$ est essentiellement surjectif.

Pour la pleine fidÈlitÈ, on regarde
\begin{align*} \mathrm{Hom}_{(\mathrm{Coh}_X^{-})^{+}}(\mathcal{F},\mathcal{G}) &= \mathrm{Hom}_{(\mathrm{Coh}_X^{-})^{+}}((\mathcal{F}',\mathcal{F}''),(\mathcal{G}',\mathcal{G}'')) \\
&= \begin{pmatrix} \mathrm{Hom}_{\mathrm{Coh}_X^{-}}(\mathcal{F}',\mathcal{G}') & \mathrm{Ext}_{\mathrm{Coh}_X^{-}}^1(\mathcal{F}'',\mathcal{G}') \\ 0 & \mathrm{Hom}_{\mathrm{Coh}_X^{-}}(\mathcal{F}'',\mathcal{G}'') \end{pmatrix}.
\end{align*}
Or :
\[ \mathrm{Hom}_{\mathrm{Coh}_X^{-}}(\mathcal{F}',\mathcal{G}') = \mathrm{Hom}_{\mathrm{Coh}_X}(\mathcal{F}',\mathcal{G}'), \]
\[ \mathrm{Hom}_{\mathrm{Coh}_X^{-}}(\mathcal{F}'',\mathcal{G}'') = \mathrm{Hom}_{D^b(\mathrm{Coh}_X)}(\mathcal{F}'''[1],\mathcal{G}'''[1]) = \mathrm{Hom}_{\mathrm{Coh}_X}(\mathcal{F}''',\mathcal{G}''') \]
et
\[ \mathrm{Ext}_{\mathrm{Coh}_X^{-}}^1(\mathcal{F}'',\mathcal{G}') = \mathrm{Ext}_{D^b(\mathrm{Coh}_X)}^1(\mathcal{F}'''[1],\mathcal{G}') = \mathrm{Hom}_{\mathrm{Coh}_X}(\mathcal{F}''',\mathcal{G}'). \]
D'o˘
\begin{align*}
\mathrm{Hom}_{(\mathrm{Coh}_X^{-})^{+}}(\mathcal{F},\mathcal{G}) &= \begin{pmatrix} \mathrm{Hom}_{\mathrm{Coh}_X}(\mathcal{F}',\mathcal{G}') & \mathrm{Hom}_{\mathrm{Coh}_X}(\mathcal{F}''',\mathcal{G}') \\ 0 & \mathrm{Hom}_{\mathrm{Coh}_X}(\mathcal{F}''',\mathcal{G}''') \end{pmatrix} \\ 
&= \mathrm{Hom}_{\mathrm{Coh}_X}((\mathcal{F}''',\mathcal{F}'),(\mathcal{G}''',\mathcal{G}')) \\
&= \mathrm{Hom}_{\mathrm{Coh}_X}(\iota(\mathcal{F}),\iota(\mathcal{G})). 
\end{align*} 
Enfin, le fait que les fonctions rang et degrÈ se correspondent via cette Èquivalence est facile.
\end{proof}

\begin{remarks} a) Cela entraÓne en particulier que $D^b(\mathrm{Coh}_X^-)=D^b(\mathrm{Coh}_X)$.

b) On a exploitÈ le fait que la filtration de Harder-Narasimhan des faisceaux cohÈrents sur $X$ Ètait scindÈe. Pour une catÈgorie abÈlienne $\mathcal{A}$ avec un formalisme de Harder-Narasimhan comme ci-dessus, on peut vÈrifier qu'on a toujours que l'image de $(\mathcal{A}^-)^+$ par $\iota$ tombe dans $\mathcal{A}$, mais ces catÈgories n'ont pas de raison d'Ítre Èquivalentes.
\end{remarks} 

\section{Une description alternative de la catÈgorie $\mathrm{Coh}_X^{-}$} \label{perversitÈ}

\subsection{AlgËbres sympathiques}
Les algËbres sympathiques ont ÈtÈ introduites par Colmez pour dÈfinir les espaces de Banach-Colmez.

\begin{definition} \label{defsympa}
Une $C$-algËbre de Banach $R$ est dite \textit{sympathique} si elle est uniforme (ce qui signifie que $R^{\circ}$ est bornÈ, ou, de faÁon Èquivalente, qu'il existe sur $R$ une norme multiplicative pour les puissances induisant la topologie de $R$) et \textit{$p$-close} : tout ÈlÈment de $1+R^{\circ \circ}$ admet une racine $p$-Ëme.
\end{definition}

\begin{remarque} 
Une algËbre de Banach $R$ est uniforme si et seulement si la norme qui dÈfinit sa topologie est Èquivalente ‡ la semi-norme spectrale de $R$ (\cite[Def. 2.8.1]{KL}). La dÈfinition prÈcÈdente est donc la mÍme que celle de Colmez (\cite[\S 4]{bc}), la condition de connexitÈ -- qui n'est pas essentielle -- mise ‡ part.  
\end{remarque}

\begin{proposition}\label{sympabase}
Une algËbre sympathique est perfectoÔde. Pour toute $C$-algËbre perfectoÔde $R$, il existe une $R$-algËbre sympathique pro-finie-Ètale. En particulier, les spectres affinoÔdes d'algËbres sympathiques forment une base de la topologie du site $\mathrm{Perf}_{C,\mathrm{pro\acute{e}t}}$. 
\end{proposition}
\begin{proof}
Pour la preuve, voir \cite[Lemme 1.15 (iii)]{bc} et \cite[Prop. 4.20 (i)]{bc}.
\end{proof}

\begin{examples} a) Le corps $C$ est sympathique. 

b) La $C$-algËbre $C\langle T^{1/p^{\infty}} \rangle$ (obtenue en complÈtant $p$-adiquement $\bigcup_{n\geq 1} \O_C[T^{1/p^n}]$ puis en inversant $p$) n'est pas sympathique, bien qu'elle soit perfectoÔde. En effet, soit $a \in \O_C$, avec $|p|<|a|$. Alors 
\[ y_a=\sum_{k=0}^{\infty} {1/p \choose k} a^k T^k \]
est une racine $p$-Ëme de $1+aT$ dans $C [\![ T^{1/p^{\infty}}]\!]$, qui contient $C\langle T^{1/p^{\infty}} \rangle$. Si $1+aT$ avait une racine $p$-Ëme dans $C\langle T^{1/p^{\infty}} \rangle$, $y_a$ serait dans $\O_C\langle T^{1/p^{\infty}} \rangle$ : c'est absurde, puisque par choix de $a$ le coefficient de $T$ n'est pas dans $\O_C$. 
\end{examples}

\begin{proposition}
Une $C$-algËbre de Banach uniforme $R$ est sympathique si et seulement si l'application logarithme $\log : 1+R^{\circ \circ} \to R$ est surjective.
\end{proposition}
\begin{proof}
Supposons $R$ sympathique. Soit $x\in R$. On choisit $n$ assez grand pour que $p^nx$ soit dans l'image de $\log$ (ce qui est possible, puisque l'exponentielle converge sur un voisinage de $0$ et y dÈfinit un inverse du logarithme). On a donc $p^n x=\exp(y')$, et comme $y' \in 1+R^{\circ \circ}$, il existe par hypothËse $y \in 1+R^{\circ \circ}$ tel que $y^{p^n}=y'$, i.e. tel que $\exp(y)=x$.

RÈciproquement, soit $y \in 1+R^{\circ \circ}$. Il existe $y'\in 1+R^{\circ \circ}$ tel que $\log(y')=p^{-1}x$ et donc $\log(y^{'-p}y)=0$. Donc $y^{'-p}y \in \mu_{p^{\infty}}$ et $y$ admet donc une racine $p$-Ëme. 
\end{proof}

\begin{remarque}
Soit $h\geq 1$ et $d\in \z$. On montre (\cite[Cor. 5.2.12]{KL}) que
\[ B(R)^{\varphi^h=p^d} = B^+(R)^{\varphi^h=p^d} \]
et comme $B^+(R)= \cap_{n\geq 0} \varphi^n(B_{\rm cris}^+(R))$, c'est aussi la mÍme chose que $B_{\rm cris}^+(R)^{\varphi^h=p^d}$. Ainsi dans la dÈfinition du corollaire \ref{explicite}, on aurait aussi bien pu dÈfinir $\mathbf{U}_{\lambda}$ avec l'anneau $B(R)$. 
\end{remarque}

\begin{corollaire}\label{thetasympa}
Une $C$-algËbre de Banach uniforme $R$ est sympathique si et seulement si l'application $\theta : B(R^{\flat})^{\varphi=p} \to R$ de Fontaine est surjective.
\end{corollaire}
\begin{proof}
En effet comme on l'a observÈ dans la section \ref{revunivpdiv}, la thÈorie des groupes $p$-divisibles\footnote{Entre autres possibilitÈs...} permet d'identifier 
\[ \theta : B(R^{\flat})^{\varphi=p} \to R \]
et 
\[ \underset{x \mapsto x^p}\varprojlim 1+R^{\circ \circ} \to R, ~ (x_n) \mapsto \log(x_0). \]
En outre, l'argument de la preuve prÈcÈdente montre que $R$ est sympathique si et seulement la deuxiËme application est surjective.
\end{proof}

\subsection{La catÈgorie $\mathrm{Coh}_X^-$ comme catÈgorie de faisceaux \og pervers cohÈrents \fg{} (au sens de \cite{bridgeland})}
Soit $S$ un espace perfectoÔde sur $C^{\flat}$. On peut penser en premiËre approximation ‡ l'espace adique $X_S$ comme ‡ une famille $(X_{k(s)})$ de courbes de Fargues-Fontaine usuelles, indexÈe par les points gÈomÈtriques de $S$. NÈanmoins il faut prendre garde au fait qu'il n'y a pas de morphisme naturel d'espaces adiques $X_S \to S$ ! En caractÈristique $p$ cela vient du fait qu'on a quotientÈ par $\varphi$ ; en caractÈristique mixte, c'est encore pire, il n'y a mÍme pas de tel morphisme au niveau de $Y_S$. Toutefois, bien que $X_S$ ne vive pas au dessus de $S$, la construction de $X_S$ est fonctorielle en $S$. On vÈrifie de plus (en Ètendant les scalaires ‡ $\qp^{\rm cyc}$ et en basculant en caractÈristique $p$) que si $f : S\to S'$ est un morphisme pro-Ètale (resp. surjectif), le morphisme induit $X_S \to X_{S'}$ est pro-Ètale (resp. surjectif). De plus, $f$ commute aux limites projectives finies. 

En d'autres termes, on a un morphisme de topos $\tau$ du topos associÈ au gros site pro-Ètale de $X$ vers le topos $(\mathrm{Perf}_{C^{\flat},\mathrm{pro\acute{e}t}})^{\sim}$. En particulier, si $\mathcal{F}$ est un complexe de faisceaux cohÈrents sur $X$, on peut lui associer un complexe de faisceaux $R\tau_* \mathcal{F}$ sur $\mathrm{Perf}_{C^{\flat},\mathrm{pro\acute{e}t}}$. 

\begin{proposition}\label{bridgeland}
Soit $\mathcal{F} \in \mathrm{Coh}_X$. Alors :
\begin{itemize} 
\item Si tous les quotients de la filtration de Harder-Narasimhan de $\mathcal{F}$ sont ‡ pentes $\geq 0$, $R^i \tau_*\mathcal{F}=0$ pour tout $i\neq 0$.
\item Si tous les quotients de la filtration de Harder-Narasimhan de $\mathcal{F}$ sont ‡ pentes $<0$, $R^i \tau_*\mathcal{F}=0$ pour tout $i\neq 1$. 
\end{itemize}
\end{proposition}
\begin{proof}
Le thÈorËme de classification des fibrÈs montre qu'il suffit de vÈrifier que :
\begin{itemize}
\item $R^i \tau_* \O_X(\lambda)=0$ pour tout $i\neq 0$, si $\lambda \geq 0$ ;
\item $R^i \tau_* \O_X(\lambda)=0$ pour tout $i \neq1$, si $\lambda < 0$ ;
\item $R^i\tau_* \iota_{x,*} B_{\rm dR}^+(C_x)/t^k B_{\rm dR}^+(C_x)=0$ pour tout $i\neq 0$, si $x$ est un point fermÈ de $X$, $C_x$ le corps rÈsiduel de $X$ en $x$ et $k>0$. 
\end{itemize}
Rappelons que $R^i \tau_* \mathcal{F}$ est le faisceau associÈ au prÈfaisceau $S \mapsto H^i(X_S,\mathcal{F}_S)$. On peut donc supposer $S=\mathrm{Spa}(R,R^+)$ affinoÔde perfectoÔde.
\\

VÈrifions d'abord le troisiËme point. Le corps $C_x$ est un corps perfectoÔde de basculÈ $C^{\flat}$ et donne naissance ‡ un dÈbasculement $S^{\sharp}$ de $S$ au-dessus de $C_x$ par l'Èquivalence de Scholze entre espaces perfectoÔdes sur $C_x$ et sur $C^{\flat}$, et ‡ un morphisme $\iota : S^{\sharp} \to X_S$ (par produit fibrÈ de $\iota$ et $X_S \to X$, puisque $\mathrm{Hom}(S^{\sharp},X)=\mathrm{Hom}(S,C^{\flat})$). On a
\[ H^i(X_S,\iota_* \mathbb{B}_{\mathrm{dR},S^{\sharp}}^+/t^k\mathbb{B}_{\mathrm{dR},S^{\sharp}}^+)= H^i(S^{\sharp},\mathbb{B}_{\mathrm{dR},S^{\sharp}}^+/t^k\mathbb{B}_{\mathrm{dR},S^{\sharp}}^+) \]
qui est nul si $i>0$ car $S^{\sharp}$ est affinoÔde perfectoÔde (\cite[Th. 6.5 (ii)]{Shodge}).
\\

Traitons maintenant le cas des fibrÈs $\O_{X_S}(\lambda)$, $\lambda \in \q$. Si $i=0$, il s'agit de calculer $B(R)^{\varphi^h=p^d}$ si $\lambda=d/h$. On a pour tout $f\in B(R)$ :
\[ \| \varphi(f)\|_{\rho} = \| f \|_{\rho^{1/p}}^p. \]
Donc si $\varphi^h(f)=p^d f$, on a par une rÈcurrence immÈdiate que pour tout $k\geq 1$
\[ \|f\|_{\rho^{p^{kh}}} = \rho^{-dp^{kh}} \|f\|_{\rho}^{p^{kh}}. \]
Quitte ‡ multiplier $f$ par une constante, on peut supposer $\|f\|_{\rho} < 1$. On en dÈduit que si $d \leq 0$, le membre de droite tend vers $0$, et donc 
\[ \lim_{\rho \to 0} \|f\|_{\rho} = 0. \]
Donc $f$ a une singularitÈ Èliminable en $0$ (la preuve de \cite[Prop. 1.9.1]{FF} s'adapte sans problËme au cas d'une algËbre perfectoÔde $R$) et donc $f \in W(R)$ (pour le voir, on se ramËne au cas connu des corps : cf. \cite[Lem. 5.2.5]{KL}). Or on voit directement que $W(R)^{\varphi^h=p^d}$ est nul dËs que $d\neq 0$ et vaut $W(R^{\varphi=1})=\qp^{\pi_0(S)}$ si $d=0$. Par consÈquent, on a bien $R^0\tau_* \O_X(\lambda)=0$ si $\lambda<0$. On a aussi obtenu au passage que $R^0\tau_* \O_X=\qp$.  

Pour comprendre ce qui se passe en degrÈ $i>0$, notons pour commencer qu'on peut supposer $\lambda=d$ entier, quitte ‡ remplacer $X$ par $X_h$, $h\geq 1$. En effet, si $\lambda=d/h$, $\O_X(\lambda)=\pi_{h*}\O_{X_h}(d)$, avec $\pi_h : X_h \to X$ Ètale et donc
\[ H^i(X_S,\O_{X_S}(\lambda))=H^i(X_h,\O_{X_{S,h}}(\lambda)). \]
Si $\pi : Y_S \to X_S$ est la surjection canonique, $\pi^*\O_{X_S}(d)=\O_{Y_S}$ par dÈfinition de $\O_X(d)$. Pour tout $i>0$, $H^i(Y_S,\O_{Y_S})$ est nul. En effet, on peut Ècrire $Y_S=\cup_{n,m} Y_{S,n,m}$ comme ci-dessus et $H^0(Y_S,\O_{Y_S})=\varprojlim H^0(Y_{S,n,m}, \O_{Y_S})$, les morphismes de transition Ètant continus et d'image dense. On a donc une suite exacte pour tout $i>0$ :
\[ 0 \to R^1\varprojlim H^{i-1}(Y_{S,n,m},\O_{Y_S}) \to H^i(Y_S,\O_{Y_S}) \to \varprojlim H^i(Y_{S,n,m},\O_{Y_S}) \to 0. \]
Or chaque $Y_{S,n,m}$ est affinoÔde perfectoÔde et donc le terme de droite est nul pour $i>0$ (\cite[Lemma 9.2.8]{KL}). Cela donne l'annulation cherchÈe pour $i>1$. Pour $i=1$, on utilise que de plus 
\[ R^1\varprojlim H^{0}(Y_{S,n,m},\O_{Y_S})= 0, \]
d'aprËs la proposition \ref{mittag}. De la suite spectrale de Hochschild-Serre pour $\pi : Y_S \to X_S$ on dÈduit que 
\[ H^i(X_S,\O_{X_S}(d))=0 \]
si $i>1$ et que
\begin{align*}
H^1(X_S,\O_{X_S}(d)) &= \mathrm{coker}(H^0(Y_S,\pi^*\O_{X_S}(d)) \overset{\mathrm{Id}-\varphi} \longrightarrow H^0(Y_S,\pi^*\O_{X_S}(d))) \\
&= \mathrm{coker}(B(R) \overset{\mathrm{Id}-p^{-d}\varphi} \longrightarrow B(R)),
\end{align*}
Or ce dernier groupe est nul si $d>0$, d'aprËs \cite[Prop. 6.2.2]{KL}. 
 
Il ne reste donc plus ‡ Ètudier que le cas $i=1$, $\lambda=0$, qui est plus dÈlicat. Comme la question est locale pour la topologie pro-Ètale, on peut supposer $S=\mathrm{Spa}(R^{\flat},R^{\flat +})$, avec $R$ une $C$-algËbre sympathique (proposition \ref{sympabase}). Regardons la suite exacte de fibrÈs
\[ 0 \to \O_{X_S} \to \O_{X_S}(1) \to i_{\infty,*} \O_{S^{\sharp}} \to 0. \]
(avec $S^{\sharp}=\mathrm{Spa}(R,R^+)$) et prenons la suite exacte longue de cohomologie qui s'en dÈduit en appliquant $\Gamma(X_S,\cdot)$. On voit qu'il suffit de montrer que la flËche $$H^0(X_S,\O_{X_S}(1)) \to H^0(X_S,i_{\infty,*} \O_{S^{\sharp}})$$ est surjective pour avoir que $H^1(X_S,\O_{X_S})=0$, puisque l'on sait dÈj‡ que $H^1(X_S,\O_{X_S}(1))=0$. Mais cette flËche n'est autre que $\theta : B(R^{\flat})^{\varphi=p} \to R$. Le corollaire \ref{thetasympa} permet de conclure que $R^1\tau_* \O_X=0$.
\end{proof}

\begin{remarque} 
La preuve montre que $R^0\tau_* \O_X=\qp$. Si $\lambda=d/h >0$, $R^0\tau_* \O_X(\lambda)= \mathbf{U}_{\lambda}$ qui est non nul. Si $\lambda=d/h<0$, $R^1\tau_* \O_X(\lambda) \neq 0$ : cela se vÈrifie aisÈment ‡ l'aide des calculs prÈcÈdents et de la suite exacte
\[ 0 \to \O_{X_E}(d) \to \O_{X_E} \to \iota_{\infty,*} B_{\rm dR}^+/t^d B_{\rm dR}^+ \to 0, \]
o˘ $E=\q_{p^h}$.
\end{remarque}

De la proposition et de la remarque prÈcÈdentes, on dÈduit le

\begin{corollaire}
La catÈgorie $\mathrm{Coh}_X^{-}$ est exactement la sous-catÈgorie pleine de $D^b(\mathrm{Coh}_X)$ suivante
\[ \left\{ \mathcal{F}, H^i(\mathcal{F})=0 ~\mathrm{si} ~i\neq 0, -1, ~ R^1\tau_* H^0(\mathcal{F})=0, R^0\tau_* H^{-1}(\mathcal{F})=0 \right\}. \]
Autrement dit, le foncteur $R^0\tau_*$ restreint ‡ la catÈgorie $\mathrm{Coh}_X^{-}$ est exact. 
\end{corollaire}

\section{La catÈgorie $\mathcal{BC}$ en termes de la courbe} \label{sectionfinale}

\subsection{La catÈgorie $\mathcal{BC}$ comme coeur abÈlien de $D^b(\mathrm{Coh}_X)$}
On vient de voir que le foncteur $R^0\tau_*$ de $\mathrm{Coh}_X^-$ dans la catÈgorie des faisceaux sur $\mathrm{Perf}_{C,\mathrm{pro\acute{e}t}}$ Ètait exact. On va voir maintenant qu'on a en fait beaucoup mieux. Le thÈorËme \ref{eqscholze} de Scholze permet d'identifier les sites $\mathrm{Perf}_{C,\mathrm{pro\acute{e}t}}$ et $\mathrm{Perf}_{C^{\flat},\mathrm{pro\acute{e}t}}$, ce que l'on fait dÈsormais, en particulier dans l'ÈnoncÈ suivant, qui est le rÈsultat principal de ce texte. 
\begin{theoreme}\label{mainequi}
Le foncteur $R^0\tau_*$ rÈalise une Èquivalence de catÈgories entre $\mathrm{Coh}_X^{-}$ et $\mathcal{BC}$.
\end{theoreme}
\begin{proof}
Montrons tout d'abord que l'image de $\mathrm{Coh}_X^-$ par $R^0\tau_*$ tombe dans $\mathcal{BC}$. Le groupe de Grothendieck $K_0(\mathrm{Coh}_X^-)$ est isomorphe ‡ $K_0(\mathrm{Coh}_X)$ via la flËche $[\mathcal{F}] \mapsto [H^0(\mathcal{F})]-[H^{-1}(\mathcal{F})]$. Or
\[ K_0(\mathrm{Coh}_X) \simeq K_0(\mathrm{Fib}_X) \simeq \z [\O_X] \oplus \mathrm{Pic}(X) \simeq \z [\O_X] \oplus \z [\O_X(1)]. \]
Comme la catÈgorie $\mathcal{BC}$ est par dÈfinition abÈlienne et stable par extensions, il suffit donc de vÈrifier que $R^0\tau_* \O_X, R^0\tau_* \O_X(1) \in \mathcal{BC}$. Or on a $R^0\tau_* \O_X(1)=\mathbf{U}_1$ et $R^0\tau_* \O_X=\qp$ comme on l'a vu au cours de la preuve de la proposition \ref{bridgeland}.
\\

Soit $\tilde{\mathcal{BC}}$ la catÈgorie dÈfinie de faÁon analogue ‡ $\mathcal{BC}$ en remplaÁant faisceaux de $\qp$-espaces vectoriels par faisceaux de groupes abÈliens. La catÈgorie $\mathcal{BC}$ est une sous-catÈgorie de $\tilde{\mathcal{BC}}$. Pour prouver le thÈorËme, on va montrer que $R^0\tau_*$ induit une Èquivalence entre $\mathrm{Coh}_X^{-}$ et $\tilde{\mathcal{BC}}$. Observons alors que si $\mathcal{F}, \mathcal{G} \in \mathrm{Coh}_X^-$, tout ÈlÈment de $\mathrm{Hom}_{\mathrm{Coh}_X^-}(\mathcal{F},\mathcal{G})$ induit par $R^0\tau_*$ un morphisme $\qp$-linÈaire entre les faisceaux de $\qp$-vectoriels $R^0 \tau_*(\mathcal{F})$ et $R^0 \tau_*(\mathcal{G})$. On en dÈduit donc que les trois catÈgories $\mathrm{Coh}_X^-$, $\mathcal{BC}$ et $\tilde{\mathcal{BC}}$ sont nÈcessairement Èquivalentes.

L'exactitude du foncteur $R^0\tau_* : \mathrm{Coh}_X^- \to \tilde{\mathcal{BC}}$ donne une flËche pour tout $\mathcal{F}, \mathcal{G} \in \mathrm{Coh}_X^-$ et tout $i\geq 0$
\[ (*)_{i,\mathcal{F},\mathcal{G}} : \mathrm{Ext}_{\mathrm{Coh}_X^-}^i(\mathcal{F},\mathcal{G}) \to \mathrm{Ext}_{\z-\mathrm{Sh}}^i(R^0\tau_*\mathcal{F},R^0\tau_*\mathcal{G}).  \]
Par dÈfinition de $\tilde{\mathcal{BC}}$ comme plus petite-sous catÈgorie abÈlienne stable par extensions de la catÈgorie des faisceaux de groupes abÈliens sur $\mathrm{Perf}_{C,\mathrm{pro\acute{e}t}}$ contenant les $\qp$-Espaces Vectoriels de dimension finie et les $C$-Espaces Vectoriels de dimension finie, il s'agit donc pour dÈmontrer le thÈorËme de prouver que la flËche ci-dessus est un isomorphisme pour $i=0, 1$ et pour tout $\mathcal{F}, \mathcal{G} \in \mathrm{Coh}_X^-$.

\begin{proposition} \label{breenbis}
On suppose que $\mathcal{F}, \mathcal{G}$ sont des $\qp$-Espaces Vectoriels de dimension finie ou des $C$-Espaces Vectoriels de dimension finie. Alors $(*)_{i,\mathcal{F},\mathcal{G}}$ est un isomorphisme pour $i=0,1, 2$. 
\end{proposition}
\begin{proof}
C'est un corollaire des rÈsultats de la section \ref{extproÈtales} et du calcul des $\mathrm{Ext}^i$ entre faisceaux cohÈrents sur la courbe.
\end{proof}

Montrons pour commencer que l'ÈnoncÈ de la proposition \ref{breenbis} reste valable lorsque $\mathcal{F}$ et $\mathcal{G}$ sont ‡ pentes entre $0$ et $1$, i.e. que $(*)_{i,\mathcal{F},\mathcal{G}}$ est un isomorphisme lorsque $\mathcal{F}$ et $\mathcal{G}$ sont ‡ pentes entre $0$ et $1$ et $i=0, 1, 2$. Il existe alors $V, V'$ deux $\qp$-espaces vectoriels de dimension finie, $W, W'$ deux $C$-espaces vectoriels de dimension finie et des suites exactes
\[ 0 \to V \otimes \O_X \to \mathcal{F} \to i_{\infty,*} W \to 0 \quad ; \quad 0 \to V' \otimes \O_X \to \mathcal{G} \to i_{\infty,*} W' \to 0. \]
D'o˘ des suite exactes
\begin{align*}
0 \to \mathrm{Hom}_{\mathrm{Coh}_X^-}(i_{\infty,*} W,\mathcal{G}) \to \mathrm{Hom}_{\mathrm{Coh}_X^-}(\mathcal{F},\mathcal{G}) & \to \mathrm{Hom}_{\mathrm{Coh}_X^-}(V \otimes \O_X,\mathcal{G}) \\
& \to \mathrm{Ext}_{\mathrm{Coh}_X^-}^1(i_{\infty,*} W,\mathcal{G}), \end{align*}

\begin{align*}  \mathrm{Hom}_{\mathrm{Coh}_X^-}(V \otimes \O_X,\mathcal{G}) \to \mathrm{Ext}_{\mathrm{Coh}_X^-}^1(i_{\infty,*} W,\mathcal{G}) \to \mathrm{Ext}_{\mathrm{Coh}_X^-}^1(\mathcal{F},\mathcal{G}) & \to \mathrm{Ext}_{\mathrm{Coh}_X^-}^1(V \otimes \O_X,\mathcal{G}) \\ & \to \mathrm{Ext}_{\mathrm{Coh}_X^-}^2(i_{\infty,*} W,\mathcal{G}), \end{align*}

et
\[  \mathrm{Ext}_{\mathrm{Coh}_X^-}^1(V \otimes \O_X,\mathcal{G}) \to \mathrm{Ext}_{\mathrm{Coh}_X^-}^2(i_{\infty,*} W,\mathcal{G}) \to \mathrm{Ext}_{\mathrm{Coh}_X^-}^2(\mathcal{F},\mathcal{G})  \to \mathrm{Ext}_{\mathrm{Coh}_X^-}^2(V \otimes \O_X,\mathcal{G}). \]

Ce dernier groupe est nul dans les deux catÈgories considÈrÈes, puisque $\mathrm{Ext}_{\mathrm{Coh}_X^-}^2(V \otimes \O_X,V'\otimes \O_X)$ (resp. $\mathrm{Ext}_{\z-\mathrm{Sh}}^2(V V')$) et $\mathrm{Ext}_{\mathrm{Coh}_X^-}^2(V \otimes \O_X,i_{\infty,*}W')$ (resp. $\mathrm{Ext}_{\z-\mathrm{Sh}}^2(V,W'\otimes \mathbf{G}_a)$) le sont, d'aprËs le thÈorËme \ref{breen}. On peut donc supposer que $\mathcal{F}$ est un $\qp$-Espace Vectoriel ou un $C$-Espace Vectoriel de dimension finie. On Ècrit alors

\begin{align*} 0 \to \mathrm{Hom}_{\mathrm{Coh}_X^-}(\mathcal{F},V' \otimes \O_X) \to \mathrm{Hom}_{\mathrm{Coh}_X^-}(\mathcal{F},\mathcal{G}) & \to \mathrm{Hom}_{\mathrm{Coh}_X^-}(\mathcal{F},i_{\infty,*} W') \\
& \to \mathrm{Ext}_{\mathrm{Coh}_X^-}^1(\mathcal{F},V' \otimes \O_X), \end{align*}

\begin{align*} \mathrm{Hom}_{\mathrm{Coh}_X^-}(\mathcal{F},i_{\infty,*} W') & \to \mathrm{Ext}_{\mathrm{Coh}_X^-}^1(\mathcal{F},V' \otimes \O_X) \\
& \to \mathrm{Ext}_{\mathrm{Coh}_X^-}^1(\mathcal{F},\mathcal{G})  \to \mathrm{Ext}_{\mathrm{Coh}_X^-}^1(\mathcal{F},i_{\infty,*} W')  \to \mathrm{Ext}_{\mathrm{Coh}_X^-}^2(\mathcal{F},V' \otimes \O_X) \end{align*}

et

\[ \mathrm{Ext}_{\mathrm{Coh}_X^-}^1(\mathcal{F},i_{\infty,*} W')  \to \mathrm{Ext}_{\mathrm{Coh}_X^-}^2(\mathcal{F},V' \otimes \O_X) \to \mathrm{Ext}_{\mathrm{Coh}_X^-}^2(\mathcal{F},\mathcal{G})  \to \mathrm{Ext}_{\mathrm{Coh}_X^-}^2(\mathcal{F},i_{\infty,*} W'). \]

Ce dernier groupe est nul dans les deux catÈgories considÈrÈes, d'aprËs le thÈorËme \ref{breen}. Le rÈsultat cherchÈ se dÈduit donc de la proposition \ref{breenbis}. 
\\

Pour prouver que $(*)_{i,\mathcal{F},\mathcal{G}}$ est un isomorphisme pour $i=0, 1$ dans le cas gÈnÈral, nous ferons usage du lemme suivant :

\begin{lemme}\label{suitesexactes}
On a les suites exactes suivantes dans $\mathrm{Coh}_X^-$ pour $d>1$ et $k>0$ :
\begin{eqnarray} 0 \to  \mathcal{O}_X \to \mathcal{O}_X(1) \oplus \mathcal{O}_X(d-1) \to \mathcal{O}_X(d) \to 0, \label{se1} \end{eqnarray}
\begin{eqnarray} 0 \to \mathcal{O}_X \to \mathcal{O}_X(k) \to i_{\infty,*} B_{\rm dR}^+/t^k B_{\rm dR}^+ \to 0, \label{se2} \end{eqnarray}
\begin{eqnarray} 0 \to \mathcal{O}_X \to i_{\infty,*} B_{\rm dR}^+/t^{k} B_{\rm dR}^+ \to \mathcal{O}_X(-k)[1] \to 0. \label{se3} \end{eqnarray}
\end{lemme}
\begin{proof}
Ces suites exactes se dÈduisent respectivement des suites exactes suivantes dans $\mathrm{Coh}_X^-$, vues comme triangles exacts dans la catÈgorie dÈrivÈe : 
\[ 0 \to  \mathcal{O}_X \to \mathcal{O}_X(1) \oplus \mathcal{O}_X(d-1) \to \mathcal{O}_X(d) \to 0, \]
\[ 0 \to \mathcal{O}_X \to \mathcal{O}_X(k) \to i_{\infty,*} B_{\rm dR}^+/t^k B_{\rm dR}^+ \to 0, \]
\[ 0 \to \mathcal{O}_X(-k) \to \mathcal{O}_X \to i_{\infty,*} B_{\rm dR}^+/t^{k} B_{\rm dR}^+ \to 0. \]
Les deux derniËres suites exactes, induites par la multiplication par un ÈlÈment de $H^0(X,\O_X(k))=B^{\varphi=p^k}$ sont les mÍmes ‡ torsion prËs et la traduction gÈomÈtrique de la \og suite exacte fondamentale \fg{} en thÈorie de Hodge $p$-adique : cf \cite[8.2.1.3]{FF}. L'existence de la premiËre suite exacte se dÈmontre comme dans \cite[8.20]{bc}.
\end{proof}
On dÈduit du lemme que pour tout $\mathcal{F} \in \mathrm{Coh}_X^-$, il existe une suite exacte \textit{dans $\mathrm{Coh}_X^-$}
\[ 0 \to V \otimes \O_X \to \mathcal{F}' \to \mathcal{F} \to 0, \]
avec $\mathcal{F}'$ un fibrÈ ‡ pentes entre $0$ et $1$ et $V$ un $\qp$-espace vectoriel de dimension finie. 
\\

Soit donc $\mathcal{F}, \mathcal{G} \in \mathrm{Coh}_X^-$. On Ècrit 
\[ 0 \to V \otimes \O_X \to \mathcal{F}' \to \mathcal{F} \to 0 \quad ; \quad  0 \to V' \otimes \O_X \to \mathcal{G}' \to \mathcal{G} \to 0, \]
avec $\mathcal{F}', \mathcal{G}'$ ‡ pentes entre $0$ et $1$, $V, V'$ des $\qp$-espaces vectoriels de dimension finie. On a une suite exacte
\begin{align*} 0 \to \mathrm{Hom}_{\mathrm{Coh}_X^-}(\mathcal{F},\mathcal{G}) \to \mathrm{Hom}_{\mathrm{Coh}_X^-}(\mathcal{F}',\mathcal{G}) \to \mathrm{Hom}_{\mathrm{Coh}_X^-}(V \otimes \O_X,\mathcal{G}) & \to \mathrm{Ext}_{\mathrm{Coh}_X^-}^1(\mathcal{F},\mathcal{G}) \\ & \to \mathrm{Ext}_{\mathrm{Coh}_X^-}^1(\mathcal{F}',\mathcal{G}) \to \mathrm{Ext}_{\mathrm{Coh}_X^-}^1(V \otimes \O_X,\mathcal{G}). \end{align*} 

On peut donc supposer, ce que l'on fera, $\mathcal{F}$ ‡ pentes entre $0$ et $1$. On a des diagrammes commutatifs de suites exactes
\[ \resizebox{1.1\textwidth}{!} {\xymatrix{
  \mathrm{Hom}_{\mathrm{Coh}_X^-}(\mathcal{F},V' \otimes \O_X) \ar[r] \ar[d] & \mathrm{Hom}_{\mathrm{Coh}_X^-}(\mathcal{F},\mathcal{G}') \ar[r] \ar[d] & \mathrm{Hom}_{\mathrm{Coh}_X^-}(\mathcal{F},\mathcal{G}) \ar[r] \ar[d] & \mathrm{Ext}_{\mathrm{Coh}_X^-}^1(\mathcal{F},V' \otimes \O_X) \ar[d] \ar[r] & \mathrm{Ext}_{\mathrm{Coh}_X^-}^1(\mathcal{F},\mathcal{G}') \ar[d] \\
    \mathrm{Hom}_{\z-\mathrm{Sh}}(R^0\tau_*\mathcal{F},R^0\tau_*V' \otimes \O_X) \ar[r] & \mathrm{Hom}_{\z-\mathrm{Sh}}(R^0\tau_*\mathcal{F},R^0\tau_*\mathcal{G}') \ar[r] & \mathrm{Hom}_{\z-\mathrm{Sh}}(R^0\tau_*\mathcal{F},R^0\tau_*\mathcal{G}) \ar[r] & \mathrm{Ext}_{\z-\mathrm{Sh}}^1(R^0\tau_*\mathcal{F},R^0\tau_*V' \otimes \O_X) \ar[r] & \mathrm{Ext}_{\z-\mathrm{Sh}}^1(R^0\tau_*\mathcal{F},R^0\tau_*\mathcal{G})
  }} \]

et 

\[ \resizebox{1.1\textwidth}{!} {\xymatrix{
  \mathrm{Ext}_{\mathrm{Coh}_X^-}^1(\mathcal{F},V' \otimes \O_X) \ar[r] \ar[d] & \mathrm{Ext}_{\mathrm{Coh}_X^-}^1(\mathcal{F},\mathcal{G}') \ar[r] \ar[d] & \mathrm{Ext}_{\mathrm{Coh}_X^-}^1(\mathcal{F},\mathcal{G}) \ar[r] \ar[d] & \mathrm{Ext}_{\mathrm{Coh}_X^-}^2(\mathcal{F},V' \otimes \O_X) \ar[d] \ar[r] & \mathrm{Ext}_{\mathrm{Coh}_X^-}^2(\mathcal{F},\mathcal{G}') \ar[d] \\
    \mathrm{Ext}_{\z-\mathrm{Sh}}^1(R^0\tau_*\mathcal{F},R^0\tau_*V' \otimes \O_X) \ar[r] & \mathrm{Ext}_{\z-\mathrm{Sh}}^1(R^0\tau_*\mathcal{F},R^0\tau_*\mathcal{G}') \ar[r] & \mathrm{Ext}_{\z-\mathrm{Sh}}^1(R^0\tau_*\mathcal{F},R^0\tau_*\mathcal{G}) \ar[r] & \mathrm{Ext}_{\z-\mathrm{Sh}}^2(R^0\tau_*\mathcal{F},R^0\tau_*V' \otimes \O_X) \ar[r] & \mathrm{Ext}_{\z-\mathrm{Sh}}^2(R^0\tau_*\mathcal{F},R^0\tau_*\mathcal{G})
  }} \]

  On veut montrer que les flËches verticales du milieu $(*)_{0,\mathcal{F},\mathcal{G}}$ et $(*)_{0,\mathcal{F},\mathcal{G}}$ sont des isomorphismes ; gr‚ce au lemme des cinq, il suffit de le vÈrifier pour les autres flËches verticales. Tous les fibrÈs qui apparaissent Ètant dÈsormais ‡ pentes entre $0$ et $1$, c'est gagnÈ.
\end{proof}

\begin{remarque}
Si $\mathcal{F}$ et $\mathcal{G}$ sont deux fibrÈs ‡ pente entre $0$ et $1$, le fait que $(*)_{0,\mathcal{F},\mathcal{G}}$ est un isomorphisme peut aussi se dÈmontrer ‡ l'aide de la thÈorie des groupes $p$-divisibles. En effet, on peut supposer $\mathcal{F}=\O_X(\lambda)$, $\mathcal{G}=\O_X(\mu)$ avec $\lambda, \mu \in \q \cap [0,1]$. Donnons-nous alors deux groupes $p$-divisibles $G_{\lambda}, G_{\mu}$ sur $\O_C$ de revÍtements universels $\mathbf{U}_{\lambda},\mathbf{U}_{\mu}$ ; on note $M(\lambda), M(\mu)$ les modules de DieudonnÈ sur $\breve{\mathbf{Q}}_p$ correspondants. 
D'une part on sait (cf. \cite{durham}, Th. 7.18 et paragraphe aprËs 7.27) que le foncteur de $\varphi-\mathrm{Mod}_{B_{\rm cris}^+}$ vers $\mathrm{Bun}_X$ est pleinement fidËle, c'est-‡-dire que
\[ \mathrm{Hom}_{\mathrm{Coh}_X}(\O_X(\lambda),\O_X(\mu)) = \mathrm{Hom}_{B_{\rm cris}^+,\varphi}(M(\lambda) \otimes_L B_{\rm cris}^+,M(\mu) \otimes_L B_{\rm cris}^+). \]
D'autre part on a vu dans la remarque \ref{utile} que  
\[ \mathrm{Hom}_{\mathcal{BC}}(\mathbf{U}_{\lambda},\mathbf{U}_{\mu})= \mathrm{Hom}(\tilde{G}_{\lambda,\eta}^{\rm ad},\tilde{G}_{\mu,\eta}^{\rm ad})=\mathrm{Hom}_{B_{\rm cris}^+,\varphi}(M(\lambda) \otimes_L B_{\rm cris}^+,M(\mu) \otimes_L B_{\rm cris}^+), \]
ce qui conclut.
\end{remarque}

\begin{remarque} \label{remcarp2} 
Les mÈthodes de cet article montreraient aussi que si $E$ est un corps local de caractÈristique $p$, la catÈgorie des $E$-espaces de Banach-Colmez (cf. remarque \ref{remcarp}) est Èquivalente ‡ la catÈgorie $\mathrm{Coh}_{X_E}^-$, o˘ $X_E$ dÈsigne la courbe de Fargues-Fontaine pour le corps local $E$ (et $F=C^{\flat}$). Le thÈorËme \ref{breen} reste valable en remplaÁant $\mathrm{Perf}_C$ par $\mathrm{Perf}_{C^{\flat}}$ et $\qp$ par $E$, ‡ ceci prËs que $\mathrm{Ext}^1(\mathbf{G}_a,\mathbf{G}_a)=0$. En effet, le faisceau $\mathbf{G}_a$ est reprÈsentÈ par $\mathbf{A}_{C^{\flat}}^{1,\mathrm{perf}}$ et 
\[ \forall i>0, ~ H^i(\mathbf{A}_{C^{\flat}}^{1,\mathrm{perf}},\mathbf{G}_a)=0, \]
d'aprËs \cite[Lem. 4.10 (v)]{Shodge} (comparer avec le thÈorËme \ref{cohga}).  
%La catÈgorie des groupes quasi-algÈbriques unipotents de Serre sur $C^{\flat}$ (ou sur n'importe quelle base $S$ parfaite de caractÈristique $p$) devrait quant ‡ elle Ítre reliÈe ‡ la catÈgorie des \og faisceaux cohÈrents sur $\mathrm{Spa}(C^{\flat})/\varphi^{\z}$ \fg{} (ou sur $S/\varphi^{\z}$)\todo{‡ prÈciser !}.
\end{remarque}

Comme corollaire de la preuve du thÈorËme, on obtient le rÈsultat suivant. 
\begin{corollaire}
Les groupes $\mathrm{Hom}(\mathcal{F},\mathcal{G})$ et $\mathrm{Ext}^i(\mathcal{F},\mathcal{G})$ pour $\mathcal{F}, \mathcal{G} \in \{ \qp, \mathbf{G}_a \}$ sont les mÍmes dans la catÈgorie des faisceaux pro-Ètales de groupes abÈliens et dans celle des faisceaux pro-Ètales de $\qp$-espaces vectoriels. 
\end{corollaire}

\begin{remarque} 
On pourrait certainement le prouver par des mÈthodes semblables ‡ celles de la section \ref{extproÈtales} ; cela demanderait d'expliciter en bas degrÈ les complexes qui apparaissent dans \cite{maclane} et \cite{illusie}.
\end{remarque}

\subsection{Lien avec la dÈfinition originale de Colmez}
Rappelons la dÈfinition originale de Colmez des espaces de Banach-Colmez, appelÈs par lui \textit{Espaces de Banach de dimension finie} (\cite{bc}).

Pour Colmez, un \textit{Espace de Banach} est un foncteur covariant de la catÈgorie des $C$-algËbres sympathiques dans la catÈgorie des $\qp$-espaces de Banach. Une suite d'Espaces de Banach est dite \textit{exacte} si elle l'est quand on l'Èvalue sur toute algËbre sympathique. Un \textit{Espace de Banach de dimension finie} est alors par dÈfinition un Espace de Banach obtenu par extension et quotient comme dans la dÈfinition des Banach-Colmez ‡ partir des $\qp$-Espaces Vectoriels de dimension finie et des $C$-Espaces Vectoriels de dimension finie (que l'on voit ici par restriction comme foncteurs covariants sur la catÈgorie des $C$-algËbre sympathiques). Nous noterons $\mathcal{BC}'$ la catÈgorie des Espaces de Banach de dimension finie.

CommenÁons par une caractÈrisation des algËbres sympathiques.
\begin{proposition}\label{caracterisation}
Soit $S=\mathrm{Spa}(R^{\flat},R^{\flat \circ})$, avec $R$ une $C$-algËbre perfectoÔde. On a Èquivalence entre
\begin{itemize}
\item $R$ est sympathique ;
\item $H^1(S,\qp)=0$.
\end{itemize}
\end{proposition}
\begin{proof}
Comme $H^1(S,\O_S)=0$ pour tout affinoÔde perfectoÔde, l'Èquivalence des deux derniers points est immÈdiate. L'argument de la preuve de \ref{bridgeland} montre que $R$ est sympathique si et seulement si $H^1(X_S,\O_{X_S})=0$. Par ailleurs on a remarquÈ au cours de la preuve de la proposition \ref{bridgeland} que $\tau_* \O_X= \qp$ et $R^i\tau_* \O_X=0$ pour $i>0$ d'aprËs la mÍme proposition. Donc la suite spectrale de Leray pour $\tau$ donne pour tout $i \geq 0$ un isomorphisme
\[ H^i(X_S,\O_{X_S}) = H^i(S,\qp), \]
ce qui donne l'Èquivalence des deux premiers points.
\end{proof}

\begin{example} Si $K$ est un corps perfectoÔde contenant $C$, $K$ est sympathique si et seulement si $H^1(\mathcal{G}_K,\qp)=0$ pour tout $i>0$. La suite exacte de faisceaux
\[ 0 \to \zp \to W(\mathbf{G}_a^{\flat}) \overset{\wp :=F-1} \longrightarrow W(\mathbf{G}_a^{\flat}) \to 0 \]
montre que le corps perfectoÔde $K$ est sympathique si $\wp$ est surjectif sur $W(K^{\flat})$. On montre de plus (\cite[Prop. 3.10]{zpext}) que $W(K^{\flat})/\wp W(K^{\flat})$ est isomorphe comme groupe topologique au complÈtÈ $p$-adique de $\bigoplus_{\mathfrak{B}} \zp$, o˘ $\mathfrak{B}$ est une base de $K^{\flat}/\wp K^{\flat}$ sur $\mathbf{F}_p$. Le corps perfectoÔde $K$ est donc sympathique si $\wp$ est surjectif sur $K^{\flat}$ (i.e. $H^1(\mathcal{G}_K,\zp)$ est nul ssi $H^1(\mathcal{G}_K,\mathbf{F}_p)$ est nul). Par exemple, tout corps dont le basculement est un corps valuÈ parfait sphÈriquement complet contenant $C^{\flat}$ convient (en effet, soit $L$ un corps parfait sphÈriquement complet et $a\in L$. Si $|a|>1$, la sÈrie $\sum_{n<0} a^{p^n}$ converge car $L$ est sphÈriquement complet vers une racine $x$ de l'Èquation $\wp(x)=a$. Si $|a|\leq 1$, remplacer $a$ par $a+b^p-b$, avec $|b|>1$). Il existe donc en particulier des corps sympathiques qui ne sont pas algÈbriquement clos.
\end{example} 

\begin{corollaire}
Si $R$ et $R'$ sont deux $C$-algËbres perfectoÔdes telles que $R^{\flat} \simeq R^{'\flat}$, $R$ est sympathique si et seulement si $R'$ l'est.
\end{corollaire}
\begin{proof}
C'est un corollaire immÈdiat de la proposition \ref{caracterisation}.
\end{proof}

\begin{proposition}\label{deforiginale}
Les catÈgories $\mathcal{BC}$ et $\mathcal{BC}'$ sont Èquivalentes.
\end{proposition}
\begin{proof}
Notons tout d'abord que tout objet de $\mathcal{BC}$ s'Ècrit comme un quotient d'un espace de Banach-Colmez \textit{effectif}, c'est-‡-dire extension d'un $C$-Espace Vectoriel de dimension finie par un $\qp$-Espace Vectoriel de dimension finie, par un $\qp$-Espace Vectoriel de dimension finie. Cela se ramËne en effet, gr‚ce au thÈorËme \ref{mainequi} ‡ vÈrifier que tout faisceau cohÈrent sur la courbe s'Ècrit comme quotient d'un fibrÈ ‡ pentes entre $0$ et $1$ par un fibrÈ semi-stable de pente $0$. Or ceci est une consÈquence immÈdiate de la classification et du lemme \ref{suitesexactes}. 
 
Montrons alors que le foncteur d'oubli qui ‡ un objet de $\mathcal{BC}$ associe le foncteur sous-jacent sur la catÈgorie des $C$-algËbres sympathiques rÈalise une Èquivalence entre $\mathcal{BC}$ et $\mathcal{BC}'$. Le paragraphe prÈcÈdent et la proposition \ref{caracterisation} montrent que l'image par ce foncteur d'un espace de Banach-Colmez est bien un objet de $\mathcal{BC}'$\footnote{Et explique au passage pourquoi la dÈfinition de Colmez des suites exactes d'Espaces de Banach Ètait \og raisonnable \fg{}.}. La pleine fidÈlitÈ est un corollaire du fait que les algËbres sympathiques forment une base de la topologie pro-Ètale (proposition \ref{sympabase}). Enfin l'essentielle surjectivitÈ dÈcoule du fait qu'on a a aussi dans $\mathcal{BC}'$ :
\[ \mathrm{Ext}_{\mathcal{BC}'}^1(W \otimes \O, V) = \mathrm{Hom}_C(W,V \otimes C) \]
(\cite[Prop. 9.16]{bc}).
\end{proof}

On retrouve ainsi l'un des rÈsultats principaux de \cite{bc}. 
\begin{corollaire}
La catÈgorie $\mathcal{BC}'$ est abÈlienne.
\end{corollaire}

En bonus, on a le
\begin{corollaire}
La catÈgorie abÈlienne $\mathcal{BC}$ ne dÈpend que de $C^{\flat}$. 
\end{corollaire}

\begin{corollaire}
La catÈgorie $\mathcal{BC}$ permet de reconstruire la courbe de Fargues-Fontaine $X^{\rm sch}$.
\end{corollaire}
\begin{proof}
En effet, le thÈorËme \ref{mainequi} et la proposition \ref{recupere} permettent de reconstruire $\mathrm{Coh}_{X^{\rm sch}}$ ‡ partir de $\mathcal{BC}$, puisqu'ils montrent que
\[ \mathrm{Coh}_{X^{\rm sch}} \simeq \mathcal{BC}^+ \]
(avec les notations du paragraphe \ref{generalites}). Or un thÈorËme de Gabriel (\cite[Ch VI, \S 3]{gabriel}) affirme qu'un schÈma noethÈrien $Z$ peut Ítre reconstruit ‡ partir de la catÈgorie abÈlienne $\mathrm{Coh}_Z$ de ses faisceaux cohÈrents, ‡ Èquivalence prËs ($Z$ s'identifie comme espace topologique ‡ l'ensemble des sous-catÈgories de Serre irrÈductibles de $\mathrm{Coh}_Z$ avec pour ouverts les $D(\mathcal{I})$, ensemble des sous-catÈgories de Serre ne contenant pas une sous-catÈgorie de Serre $\mathcal{I}$ fixÈe ; le faisceau structural ÈvaluÈ sur l'ouvert $D(\mathcal{I})$ est le centre de la catÈgorie abÈlienne $\mathrm{Coh}_Z/\mathcal{I}$).
\end{proof}

\begin{corollaire}\label{diamant}
Les espaces de Banach-Colmez sont des diamants sur $\mathrm{Spa}(C)$. 
\end{corollaire}
\begin{proof}
On vient de voir que tout espace de Banach-Colmez Ètait le quotient par un $\qp$-Espace Vectoriel de dimension finie (donc par une relation d'Èquivalence pro-Ètale) du revÍtement universel d'un groupe $p$-divisible, qui est un espace perfectoÔde.
\end{proof}

\begin{question} Est-il vrai que les seuls objets reprÈsentables de $\mathcal{BC}$ sont ceux de $\mathcal{BC}^{\rm rep}$ ?
\end{question}

\begin{remarque} Dans la question prÈcÈdente, on a laissÈ de cÙtÈ les $C$-Espaces Vectoriels de dimension finie, qui ne sont pas reprÈsentables par un espace \textit{perfectoÔde}. Si toutefois on se demande quels espaces de Banach-Colmez sont reprÈsentables par des espaces adiques sur $C$, on peut s'attendre ‡ obtenir exactement les sommes directes d'un objet de $\mathcal{BC}^{\rm rep}$ et d'un $C$-Espace Vectoriel de dimension finie, c'est-‡-dire prÈcisÈment les revÍtements universels de \textit{groupes analytiques rigides de type $p$-divisible} au sens de \cite{granalyt}.
\end{remarque}

\subsection{Le \og drÙle de corps \fg{} de Colmez} 

Notons $\mathrm{Coh}_X^{\rm tors}$ la sous-catÈgorie de $\mathrm{Coh}_X$ formÈe des faisceaux de torsion, qui est la sous-catÈgorie de Serre des objets de rang $\mathrm{rg}=0$. Le foncteur fibre gÈnÈrique identifie le quotient $\mathrm{Coh}_X/\mathrm{Coh}_X^{\rm tors}$ ‡ la catÈgorie des $\mathrm{Frac}(B_e)$-espaces vectoriels ($B_e=B_{\rm cris}^+[1/t]$ et $\mathrm{Frac}(B_e)$ est le corps des fonctions de $X$). En particulier, cette catÈgorie est semi-simple avec un seul objet simple.

Si l'on fait la mÍme manipulation avec $\mathrm{Coh}_X^-$, i.e. si l'on quotiente par la sous-catÈgorie de Serre formÈe des objets de rang $\mathrm{rg}^-=0$, qui est la sous-catÈgorie des fibrÈs semi-stables de pente $0$ placÈs en degrÈ $0$, on obtient aussi une catÈgorie semi-simple avec un seul objet simple $S$ ‡ isomorphisme prËs : cela se dÈduit immÈdiatement des suites exactes du lemme \ref{suitesexactes}. Par consÈquent,
\[ \mathscr{C} := \mathrm{End}_{\mathrm{Coh}_X^-/(\mathrm{Coh}_X^-)^{\mathrm{rg}^-=0}}(S) \]
est une algËbre ‡ division. 

Le thÈorËme \ref{mainequi} identifie $\mathrm{Coh}_X^-/(\mathrm{Coh}_X^-)^{\mathrm{rg}^-=0}$ ‡ $\mathcal{BC}/(\qp-\mathrm{EV ~ de ~ d.f.})$. Comme on peut choisir pour $S$ la classe $[i_{\infty,*} C]$ de $i_{\infty,*} C$ dans le quotient, $\mathscr{C}$ est l'algËbre ‡ division ÈtudiÈe par Colmez dans \cite[\S 5 et \S 9]{bc}. On a un morphisme d'algËbres
\[ C= \mathrm{End}_{\mathrm{Coh}_X^-}(i_{\infty,*} C) \to \mathrm{End}_{\mathrm{Coh}_X^-/(\mathrm{Coh}_X^-)^{\mathrm{rg}^-=0}}([i_{\infty,*} C]), \]
qui fait de $C$ une sous-algËbre commutative de $\mathscr{C}$. On pourrait Ègalement choisir comme reprÈsentant de $S$ la classe de $B^{\varphi^h=p}$, $h\geq 1$, ou celle de $i_{x,*} C_x$, $x$ point fermÈ de $x$ ($C_x$ est un corps perfectoÔde de basculÈ $C^{\flat}$). On en dÈduit que $\mathscr{C}$ contient Ègalement comme sous-algËbres toutes les algËbres ‡ division d'invariant $1/h$, $h\geq 1$ et tous les dÈbasculements de $C^{\flat}$ en caractÈristique zÈro !

Colmez prouve que $C$ est une sous-algËbre commutative maximale de $\mathscr{C}$ (\cite[Prop. 5.30]{bc}) et que le centre de $\mathscr{C}$ est rÈduit ‡ $\qp$ (\cite[Prop. 9.22]{bc})\footnote{Peut-on le dÈmontrer plus facilement avec la dÈfinition de $\mathscr{C}$ adoptÈe ci-dessus ?}. La question suivante est due ‡ Laurent Fargues.
\begin{question} La courbe $X$ est-elle, en un sens ‡ prÈciser, \og la variÈtÈ de Severi-Brauer attachÈe ‡ l'algËbre ‡ division $\mathscr{C}$ \fg{} ?
\end{question}

\begin{remarque}
Si $\mathbf{H}$ est l'algËbre des quaternions de Hamilton, et $X_{\mathbf{R}}$ la variÈtÈ de Severi-Brauer qui lui correspond (une conique sans point rÈel), on vÈrifie que la catÈgorie $\mathrm{Coh}_{X_{\mathbf{R}}}^-/(\mathrm{Coh}_{X_{\mathbf{R}}}^-)^{\mathrm{rg}^-=0}$ a un unique objet simple, dont l'algËbre des endomorphismes est isomorphe ‡ $\mathbf{H}$. La question prÈcÈdente va donc dans le sens de l'analogie, mentionnÈe dans \cite{gtorseurs}, entre la courbe de Fargues-Fontaine et la forme tordue $X_{\mathbf{R}}$ de la droite projective (\cite{simpson}).
\end{remarque}

\section{Appendice : faisceaux de pÈriodes} \label{appendice}

Soit $X$ un espace adique sur $\mathrm{Spa}(\qp,\zp)$. 

\begin{definition}
On considËre les faisceaux pro-Ètales suivants.
\begin{itemize}
\item Le faisceau $\mathbb{A}_{\rm inf}= W(\widehat{\O}_{X^{\flat}}^+)$. On a un morphisme de faisceaux $\theta : \mathbb{A}_{\rm inf} \to \widehat{\O}_X^+$, qui s'Ètend en $\theta : \mathbb{A}_{\rm inf}[1/p] \to \widehat{\O}_X$.  
\item Le faisceau $\mathbb{B}_{\rm dR}^+ = \varprojlim_k \mathbb{A}_{\rm inf}[1/p]/(\ker(\theta))^k$. Il est muni d'une filtration dÈfinie par $\mathrm{Fil}^i \mathbb{B}_{\rm dR}^+ = \ker(\theta)^i$. 
\item Soit $t$ un gÈnÈrateur de $\mathrm{Fil}^1 \mathbb{B}_{\rm dR}^+$ (un tel ÈlÈment existe localement pour la topologie pro-Ètale, est unique ‡ une unitÈ prËs et n'est pas un diviseur de zÈro). On pose $\mathbb{B}_{\rm dR} = \mathbb{B}_{\rm dR}^+[1/t]$ et $\mathrm{Fil}^i \mathbb{B}_{\rm dR} = \sum_{j\in \z} t^{-j} \mathrm{Fil}^{i+j} \mathbb{B}_{\rm dR}^+$. 
\item On dÈfinit $\O \mathbb{B}_{\rm dR}^+$ comme le faisceau associÈ au prÈfaisceau dÈfini sur les ouverts pro-Ètales de $X$ de la forme $\mathrm{Spa}(R,R^+)= \varprojlim \mathrm{Spa}(R_i,R_i^+)$ (ces ouverts forment une base de la topologie) comme la limite inductive sur $i$ du complÈtÈ $\ker(\theta)$-adique de
\[ (R_i^+ \widehat{\otimes}_{\zp} \mathbb{A}_{\rm inf}(R,R^+))[1/p], \]
le produit tensoriel complÈtÈ dans la parenthËse Ètant $p$-adique. Ici $$\theta : (R_i^+ \widehat{\otimes}_{\zp} \mathbb{A}_{\rm inf}(R,R^+))[1/p] \to R$$ est le produit tensoriel de $R_i^+ \to R$ et de $\theta : \mathbb{A}_{\rm inf}(R,R^+) \to R$. On a encore un morphisme de faisceaux $\theta : \O \mathbb{B}_{\rm dR}^+ \to \widehat{\O}_X$. On dÈfinit une filtration sur $\O \mathbb{B}_{\rm dR}^+$ par 
$\mathrm{Fil}^i \O\mathbb{B}_{\rm dR}^+ = \ker(\theta)^i$. 
\item On pose $\O\mathbb{B}_{\rm dR} =\O\mathbb{B}_{\rm dR}^+[1/t]$ et $\mathrm{Fil}^i \O\mathbb{B}_{\rm dR} = \sum_{j\in \z} t^{-j} \mathrm{Fil}^{i+j} \O\mathbb{B}_{\rm dR}^+$.
\end{itemize}
\end{definition}  

Soit $X$ un espace adique sur $C$.

\begin{definition}
Soit $I =[a,b] \subset ]0,1]$ un sous-intervalle compact, avec $a, b\in p^{\q}$. Si $\alpha, \beta \in \O_{C^{\flat}}$ sont tels que $|\alpha|=a$, $|\beta|=b$, on dÈfinit :
\[ \mathbb{B}_I = \left(\underset{n}\varprojlim ~ (\mathbb{A}_{\rm inf}[[\alpha]/p,p/[\beta]])/p^n \right)[1/p]. \]
On pose :
\[ \mathbb{B} = \underset{I \subset ]0,1[}\varprojlim ~ \mathbb{B}_I. \]
Si $a \in p^{\q} \cap ]0,1]$, et $\alpha$ est comme prÈcÈdemment, on dÈfinit :
\[ \mathbb{B}_a^+ = \left( \underset{n}\varprojlim ~ (\mathbb{A}_{\rm inf}[[\alpha]/p])/p^n \right)[1/p]. \]
On pose : 
\[ \mathbb{B}^+=\underset{a}\varprojlim ~ \mathbb{B}_a^+. \]
\end{definition}

\begin{proposition} \label{descriptionbi}
Soit $S=\mathrm{Spa}(R,R^+)$ un espace affinoÔde perfectoÔde sur $C$. On a 
\[ H^0(S,\mathbb{B}_I) = B_I(R,R^+) \]
et
\[ H^i(S,\mathbb{B}_I) = 0 \]
si $i>0$.
\end{proposition}
\begin{proof}
Comme le faisceau $\mathbb{B}_I$ est obtenu en complÈtant $p$-adiquement $\mathbb{A}_{\rm inf}[[\alpha]/p,p/[\beta]]$ puis en inversant $p$, on voit qu'il suffit par rÈcurrence sur $m$ de dÈcrire les sections de $\mathbb{A}_{\rm inf}$ et de montrer que sa cohomologie s'annule. C'est ce qui est fait dans \cite[Th. 6.5]{Shodge}, modulo les ÈlÈments tuÈs par toutes les puissances $[p^{\flat}]^{1/n}$ dans $A_{\rm inf}$. Comme $[p^{\flat}]^{1/n}$ est inversible dans $B_I$ pour $n$ assez grand, cela suffit. 
\end{proof}

\begin{proposition} \label{debile}
Soit $G$ un groupe profini. Soit $\tilde{X} \to X$ un recouvrement pro-Ètale galoisien de groupe $G$, avec $\tilde{X}=\mathrm{Spa}(R,R^+)$ affinoÔde perfectoÔde sur $C$. Si $k\geq 1$, notons $\tilde{X}_k$ le produit fibrÈ de $\tilde{X}$ $k$-fois avec lui-mÍme au-dessus de $X$.
\[ \mathbb{B}_I(\tilde{X}_k) = \mathcal{C}^0(G^{k-1}, B_I(R,R^+)). \]
\end{proposition}
\begin{proof}
On sait que $\tilde{X}_k \simeq \tilde{X} \times G^{k-1}$, puisque $\tilde{X} \to X$ est un revÍtement pro-Ètale de groupe $G$. Ecrivons $G$ comme limite inverse de groupes finis $G_i$. Comme dans la preuve de la proposition \ref{descriptionbi}, il suffit de dÈcrire les sections de $\mathbb{A}_{\rm inf}$ sur $\tilde{X}_k$ et pour cela on montre par rÈcurrence sur $m$ que 
\[ W(\O^{\flat +})/p^m(\tilde{X}_k) = \underset{i} \varinjlim ~ \mathrm{LC}(G_i, W(R^{\flat +})) \]
et que la cohomologie de $W(\O^{\flat +})/p^m$ sur $\tilde{X}_k$ en degrÈ positif est nulle. Le deuxiËme point a dÈj‡ ÈtÈ vu dans la preuve prÈcÈdente, puisque $\tilde{X}_k$ est affinoÔde perfectoÔde. Pour le premier, il suffit de le faire pour $m=1$, i.e. pour $\O^+/p$. C'est alors une consÈquence de \cite[Lem. 3.16]{Shodge}.
\end{proof}

Soit $X$ un espace adique localement noethÈrien sur $\mathrm{Spa}(\qp,\zp)$. 
\begin{proposition}
On a une suite exacte de faisceaux sur $X_{C,\mathrm{pro\acute{e}t}}$ : 
\[ 0 \to \qp \to \mathbb{B}[1/t]^{\varphi=1} \to \mathbb{B}_{\rm dR}/\mathbb{B}_{\rm dR}^+ \to 0. \]
\end{proposition}
\begin{proof}
Comme les algËbres affinoÔdes sympathiques forment une base de $X_{C,\mathrm{pro\acute{e}t}}$, l'exactitude se dÈduit de la suite exacte (SEF 3E) de \cite[Prop. 8.25]{bc} et du fait que $\mathbb{B}[1/t]^{\varphi=1}=\mathbb{B}^+[1/t]^{\varphi=1}$ (car pour tout entier $k$, $\mathbb{B}^{\varphi=p^k}=(\mathbb{B}^+)^{\varphi=p^k}$, cf. par exemple \cite[Cor. 5.2.12]{KL}).
\end{proof}

\end{document}